\documentclass[aop]{imsart}

\RequirePackage{amsthm,amsmath,amsfonts,amssymb}
\RequirePackage[numbers,sort&compress]{natbib}
\RequirePackage[colorlinks,citecolor=blue,urlcolor=blue]{hyperref}%% uncomment this for coloring bibliography citations and linked URLs
\RequirePackage{graphicx}%% uncomment this for including figures

\usepackage{amscd} 
\usepackage{mathrsfs}
\usepackage[dvipsnames]{xcolor}
\usepackage[normalem]{ulem}
\usepackage{stmaryrd}
\usepackage[refpage,noprefix,intoc]{nomencl}
\usepackage{lipsum}
\usepackage{etoolbox}
\usepackage{tcolorbox}
\usepackage{tikz}

\startlocaldefs
\numberwithin{equation}{section}
\theoremstyle{plain}
\newtheorem{thm}{Theorem}
\newtheorem{lem}[thm]{Lemma}
\newtheorem{prop}[thm]{Proposition}
\newtheorem{cor}[thm]{Corollary}

\numberwithin{thm}{section}
\theoremstyle{definition}
\newtheorem{dfn}[thm]{Definition}

\newcommand{\tmpframe}[1]{\fbox{#1}}

\usetikzlibrary{positioning}
\usetikzlibrary{arrows}
\providecommand{\R}{}
\providecommand{\Z}{}
\providecommand{\N}{}

\renewcommand{\R}{\mathbb{R}}

\renewcommand{\Z}{\mathbb{Z}}
\renewcommand{\N}{{\mathbb N}}

\newcommand{\E}[1]{{\mathbf E}\left[#1\right]}                           

\newcommand{\V}[1]{{\mathbf{Var}}\left\{#1\right\}}
\newcommand{\va}{{\mathbf{Var}}}
\newcommand{\p}[1]{{\mathbf P}\left\{#1\right\}}

\newcommand{\I}[1]{{\mathbf 1}_{[#1]}}
\newcommand{\set}[1]{\left\{ #1 \right\}}
 
\newcommand{\probC}[2]{\mathbf{P}\set{#1 \; \left|  \; #2 \right. }}

\newcommand{\iids}{\textsc{iid}}
\newcommand{\iid}{\iids\ }
\newcommand{\fl}[1]{\ensuremath{\lfloor #1 \rfloor}}
\newcommand{\pD}[1]{\mathbf{P}_{\widehat{D}^n}\left\{#1 \right\}}

\newcommand\cB{\mathcal B}

\newcommand\cE{\mathcal E}
\newcommand\cF{\mathcal F}
\newcommand\cG{\mathcal G}

\newcommand\cL{{\mathcal L}}

\newcommand\cP{\mathcal P}

\newcommand\cS{{\mathcal S}}
\newcommand\cT{{\mathcal T}}
\newcommand\cU{{\mathcal U}}

\newcommand{\bC}{\mathbf{C}} 
\newcommand{\bD}{\mathbf{D}} 
 
\newcommand{\bF}{\mathbf{F}}

\newcommand{\bP}{\mathbf{P}}

\newcommand{\bT}{\mathbf{T}}

\newcommand{\eqdist}{\ensuremath{\stackrel{\mathrm{d}}{=}}}
\newcommand{\convdist}{\ensuremath{\stackrel{\mathrm{d}}{\longrightarrow}}}
\newcommand{\convas}{\ensuremath{\stackrel{\mathrm{a.s.}}{\longrightarrow}}}
\newcommand{\convprob}{\ensuremath{\stackrel{\mathrm{p}}{\rightarrow}}}

\newcommand{\rT}{\mathrm{T}}

\makeatletter
\newcommand{\leqnomode}{\tagsleft@true\let\veqno\@@leqno}
\newcommand{\reqnomode}{\tagsleft@false}
\makeatother

\newcommand{\bd}{\mathrm{d}}

\endlocaldefs

\begin{document}

\begin{frontmatter}
\title{Discrete snakes with globally centered displacements}
\runtitle{Discrete snakes with globally centered displacements}

\begin{aug}
%%%%%%%%%%%%%%%%%%%%%%%%%%%%%%%%%%%%%%%%%%%%%%%
%% Only one address is permitted per author. %%
%% Only division, organization and e-mail is %%
%% included in the address.                  %%
%% Additional information can be included in %%
\author[A]{\fnms{Louigi}~\snm{Addario-Berry}},
\author[B]{\fnms{Serte}~\snm{Donderwinkel}},
\author[C]{\fnms{Christina}~\snm{Goldschmidt}}
\and
\author[D]{\fnms{Rivka}~\snm{Mitchell}}
%%%%%%%%%%%%%%%%%%%%%%%%%%%%%%%%%%%%%%%%%%%%%%
%% Addresses                                %%
%%%%%%%%%%%%%%%%%%%%%%%%%%%%%%%%%%%%%%%%%%%%%%

\address[A]{Department of Mathematics and Statistics, McGill University, louigi.addario@mcgill.ca}
\address[B]{Bernoulli Institute for Mathematics, Computer Science and Artificial Intelligence, University of Groningen, s.a.donderwinkel@rug.nl}
\address[C]{Department of Statistics and Lady Margaret Hall, University of Oxford, christina.goldschmidt@stats.ox.ac.uk}
\address[D]{Mathematical Institute, University of Oxford, rivka.mitchell@maths.ox.ac.uk}

\end{aug}

\begin{abstract}
We prove a scaling limit for globally centered discrete snakes on size-conditioned critical Bienaymé trees. More specifically, under a global finite variance condition, we prove convergence in the sense of random finite-dimensional distributions of the head of the discrete snake (suitably rescaled) to the head of the Brownian snake driven by a Brownian excursion. When the third moment of the offspring distribution is finite, we further prove uniform functional convergence under a  necessary tail condition on the displacements. We also consider displacement distributions with heavier tails, for which we instead obtain convergence to a variant of the hairy snake introduced by Janson and Marckert. We further give two applications of our main result. Firstly, we obtain a scaling limit for the difference between the height process and the Łukasiewicz path of a size-conditioned critical Bienaymé tree. Secondly, we obtain a scaling limit for the difference between the height process of a size-conditioned critical Bienaymé tree and the height process of its associated looptree. 
\end{abstract}

\begin{keyword}[class=MSC]
\kwd{60J80}
\kwd{60F17}
\kwd{60C05}
\end{keyword}

\begin{keyword}
\kwd{Discrete snakes}
\kwd{branching random walks}
\kwd{branching processes}
\kwd{random trees}
\end{keyword}

\end{frontmatter}
%%%%%%%%%%%%%%%%%%%%%%%%%%%%%%%%%%%%%%%%%%%%%%
%% Please use \tableofcontents for articles %%
%% with 50 pages and more                   %%
%%%%%%%%%%%%%%%%%%%%%%%%%%%%%%%%%%%%%%%%%%%%%%
\tableofcontents

%%%%%%%%%%%%%%%%%%%%%%%%%%%%%%%%%%%%%%%%%%%%%%
%%%% Main text entry area:

\section{Introduction}
We consider a branching random walk whose genealogy is given by the family tree of a Bienaym\'e branching process (which we refer to as a \emph{Bienaym\'e tree}) conditioned to have $n$ vertices. We assume that the offspring distribution $\mu = (\mu_k)_{k \ge 0}$ is critical and has finite, non-zero variance, so that the genealogical tree has the Brownian continuum random tree as its scaling limit.\footnote{To avoid technicalities, we shall also assume that the support of $\mu$ has greatest common divisor 1, so that the event that the tree has size $n$ has strictly positive probability for all $n$ large enough.} Each vertex of the tree is endowed with a spatial location in $\R$: the root is fixed to be at 0; for every other vertex, its location is obtained via a random displacement away from the location of its parent. The random displacements of children of distinct vertices will always be independent but, in general, the displacements of siblings may be dependent and may, moreover, depend on the vertex degree. For a vertex $v$ with~$k$ children, the distribution of the vector of displacements from $v$ to its ordered children is denoted by $\nu_k$. In the sequel, $Y_k = (Y_{k,1},\dots, Y_{k,k})$ always denotes a random vector with law $\nu_k.$ In this paper, we explore conditions on $\mu$ and $\nu = (\nu_k)_{k \ge 1}$ such that the whole object converges to a Brownian motion indexed by the Brownian tree. 

A convenient formulation is via the notion of a \emph{discrete snake}. We imagine exploring the vertices of the tree one by one in depth-first order (we shall give precise definitions in Section~\ref{sec:background} below) and record a list of the spatial locations of the ancestors of the vertex we are currently visiting. In other words, the snake is a process taking values in the set of finite random walk paths (one should imagine it wiggling around as we explore the tree!). In fact, it turns out to be sufficient for many purposes to track the spatial location of the vertex that we are visiting only: this gives the so-called \emph{head of the discrete snake}, which is our primary object of interest. We aim to prove convergence, after an appropriate rescaling, of the head of the discrete snake to the head of the Brownian snake driven by a normalised Brownian excursion (BSBE), first introduced by Le Gall \cite{le1993class, le1995brownian}. This is a stochastic process $(\mathbf{e}, \mathbf{r}) = (\mathbf{e}_t, \mathbf{r}_t)_{0 \le t \le 1}$ taking values in $\R_+ \times \R$, such that $\mathbf{e}$ is a normalised Brownian excursion and, conditionally on~$\mathbf{e}$, the second coordinate $\mathbf{r}$ is a centered Gaussian process taking values in $\mathbb{R}$ with covariance function 
\begin{equation} \label{eq:BSBE}
\mathrm{cov}\left(\mathbf{r}_s, \mathbf{r}_t\right) = \min_{u\in [s\wedge t, s\vee t]}\mathbf{e}_u.
\end{equation}
Let us give some interpretation. For any pair of vertices in the Brownian tree, encoded by $s,t \in [0,1]$, having heights $\mathbf{e}_s$ and $\mathbf{e}_t$, the spatial locations along their genealogical paths evolve as a common Brownian motion until their most recent common ancestor (which lies at distance $\min_{u\in [s\wedge t, s\vee t]}\mathbf{e}_u$ from the root) is reached, and they evolve as independent Brownian motions thereafter.

The problem of proving convergence of rescaled discrete snakes to the BSBE has been studied by a number of authors, under a wide range of different conditions on $\mu$ and $(\nu_k)_{k \ge 1}$. We shall give a review of the literature after we state our main results.

In order to obtain a Brownian limit for the displacements along a lineage, we require appropriate centering and moment conditions, which we now explain. Let $\xi$ be a random variable with distribution $\mu$ and let $\bar{\xi}$ be a size-biased version, that is, having distribution $\bar\mu := (\bar \mu_k)_{k \ge 1}$, where for all $k \ge 1$,
\[
\bar\mu_k = \frac{k\mu_k}{\E{\xi}} = k\mu_k.
\]
(Recall that the offspring distribution $\mu$ is assumed to be critical, so that $\E{\xi} = 1$.)

Conditionally on $\bar\xi$, let $Y_{\bar\xi}=(Y_{\bar\xi,1},\ldots,Y_{\bar\xi,\bar\xi})$ be $\nu_{\bar\xi}$-distributed and, independently, let $U_{\bar\xi}$ be a Uniform($[\bar\xi]$) random variable (where $[m] := \{1,2,\ldots,m\})$. Then we say that the discrete snake is \emph{globally centered} if 
\[
\E{Y_{\bar\xi, U_{\bar\xi}}} = 0.
\]
In other words, the expected displacement of a uniform child of a vertex with a size-biased number of offspring is 0. We define the \emph{global variance} to be
\[
\beta^2 := \E{Y^2_{\bar\xi, U_{\bar\xi}}},
\]
and will prove our results under the condition that $\beta^2 < \infty$. Since 
distances in the tree scale as $n^{1/2}$, the spatial displacements along a lineage will then scale as  $n^{1/4}$. 

\subsection{Main result}\label{sec:main_result}
 
Denote by $\rT_n$ a Bienaym\'e tree with offspring distribution $\mu$, conditioned to have $n$ vertices. Write $v(\rT_n)$ for the vertex set of $\rT_n$ and $\partial \rT_n$ for its set of leaves. 
 Conditionally given $\rT_n$, let $Y=(Y^{(v)},v \in v(\rT_n)\setminus \partial\rT_n)$ be independent random vectors, such that if $v \in v(\rT_n)\setminus \partial\rT_n$ has $k$ children then $Y^{(v)}$ has distribution $\nu_k$. Endow the vertices of $\rT_n$ with spatial locations using the displacement vectors $Y^{(v)}$ as described above.
 We call the pair $\bT_n=(\rT_n,Y)$ a {\em $(\mu,\nu)$-branching random walk conditioned to have $n$ vertices}, or simply a $(\mu,\nu)$-branching random walk. 
 
 Let $H_n = (H_n(i))_{0 \le i \le n}$ and $\widetilde{H}_n = (\widetilde{H}_n(i))_{0 \le i \le 2n}$ be the height and contour processes of $\rT_n$, respectively. 
 Let $R_n(i)$ be the spatial location of the $i$-th vertex visited in a depth-first exploration of $\bT_n$. We call the process $(H_n, R_n)$ the \textit{head of the discrete snake}  (see Section \ref{sec:background} for a careful description of this).
We may alternatively encode the endpoints of the random walk trajectories using the process $(\widetilde{H}_n, \widetilde{R}_n)$, where $\widetilde{R}_n(i)$ is the spatial location of the $i$-th vertex visited in a contour exploration of $\bT_n$. (Compared to $(H_n, R_n)$, this simply revisits some vertices.) 
We interpolate all of these functions linearly between integer times, which turns $H_n$ and $R_n$ into elements of $\bC([0,n],\R)$ and turns $\widetilde{H}_n$ and $\widetilde{R}_n$ into elements of $\bC([0,2n],\R)$.

\begin{figure}[h!]
    \centering
    \includegraphics[scale = 0.65]{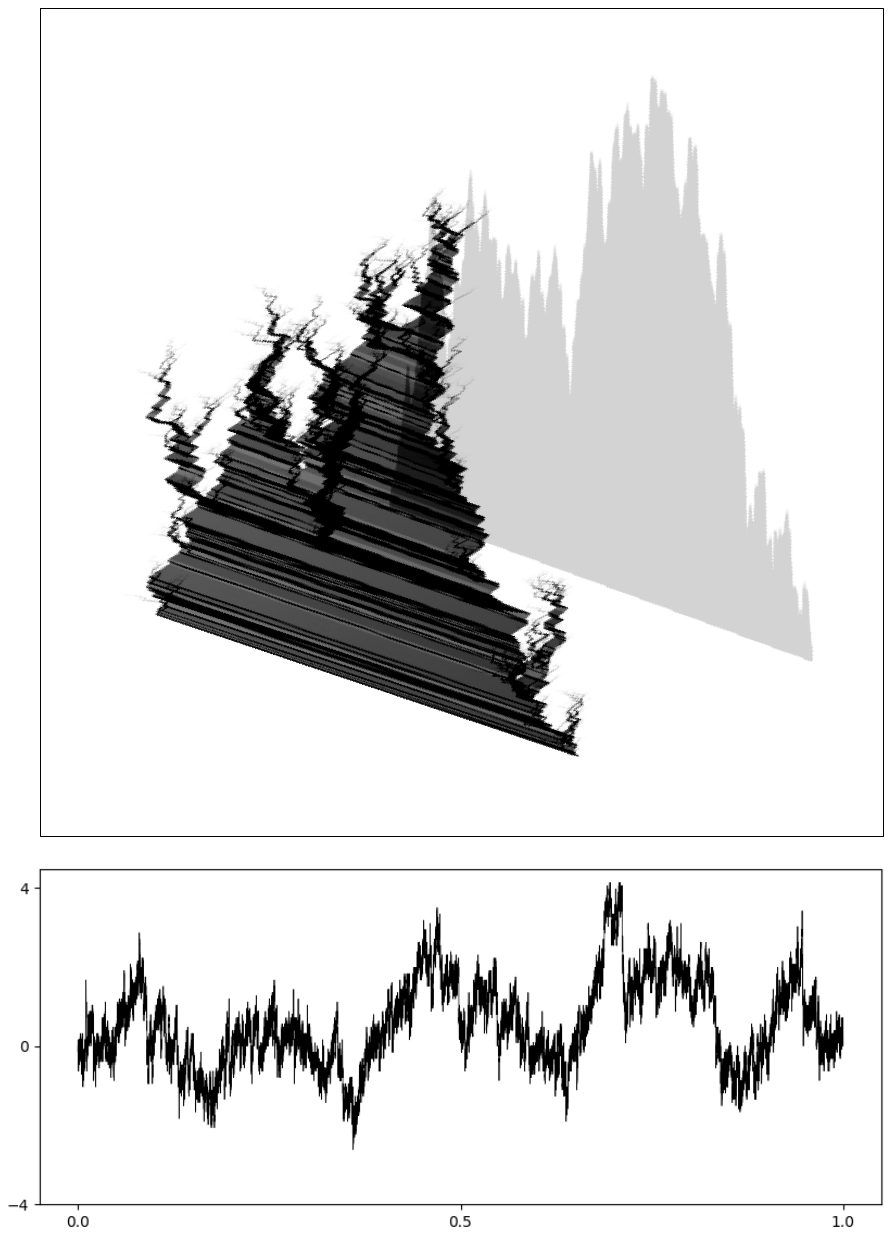}
    \vspace{-0.5cm}
    \caption{Top: a (rescaled) discrete snake; bottom: its head. The underlying tree is a size-conditioned Poisson$(1)$ Bienaymé tree with $n = 25000$ vertices and deterministic displacement distributions given by (\ref{eq:deter_disp}), below. The area under the contour process of the underlying tree is illustrated by the gray shaded region.}
    \label{fig:snake}
\end{figure}

We use two different notions of convergence for a sequence of random elements $(f_n)_{n \ge 1}$ of $\mathbf{C}([0,1], \R)$ such that $f_n(0)=f_n(1)=0$ for all $n \ge 1$. Let $U_1, U_2, \ldots$ be \iid Uniform($[0,1]$) random variables, independent of everything else. For $k \ge 1$, write $U_{(1)}^k, \ldots, U_{(k)}^k$ for the order statistics of $U_1, \ldots, U_k$. For another random element $f$ of $\mathbf{C}([0,1],\R)$ such that $f(0)=f(1)=0$, we say that $f_n \convdist f$ in the sense of \emph{random finite-dimensional distributions} if, for every $k \ge 1$,
\[
(f_n(U_{(1)}^k), \ldots, f_n(U_{(k)}^k)) \convdist (f(U_{(1)}^k), \ldots, f(U_{(k)}^k))
\]
as $n \to \infty$. (We will discuss our choice of this notion of convergence in more detail below.) We will also use the stronger notion of convergence with respect to the topology generated by the uniform norm.

\begin{thm}\label{thm:main}
Let $\mu = (\mu_k)_{k\ge 0}$ be a critical offspring distribution with  variance $\sigma^2 \in (0,\infty)$.  If $\nu=(\nu_k)_{k \ge 1}$ is such that
%%%% annoying hack to get square brackets on assumption %%%
\makeatletter
\def\tagform@#1{\maketag@@@{[\ignorespaces#1\unskip\@@italiccorr]}}
\makeatother
%%%%%%%%%%%%%%%%%%%%%%%%%%%%
% \leqnomode
\begin{equation}\leqnomode\tag{\bf A1}\label{a1}
\E{Y_{\bar\xi, U_{\bar\xi}}} = 0 \quad \text{and} \quad \beta^2 = \E{(Y_{\bar\xi, U_{\bar\xi}})^2}  < \infty,
\end{equation}
% \reqnomode
%%%% annoying hack to return to round brackets %%%
\makeatletter
\def\tagform@#1{\maketag@@@{(\ignorespaces#1\unskip\@@italiccorr)}}
\makeatother
%%%%%%%%%%%%%%%%%%%%%%%%%%%%
\noindent \!\!then as $n\rightarrow \infty$ the following joint convergence holds in the sense of random finite-dimensional distributions:
\begin{equation}\label{eq:main}
\left(\frac{H_n(nt)}{\sqrt{n}}, \frac{R_n(nt)}{n^{1/4}}, \frac{\widetilde{H}_n(2nt)}{\sqrt{n}}, \frac{\widetilde{R}_n(2nt)}{n^{1/4}} \right)_{0 \le t \le 1}\convdist \left(\frac{2}{\sigma}\mathbf{e}_t, \beta\sqrt{\frac{2}{\sigma}} \mathbf{r}_t, \frac{2}{\sigma}\mathbf{e}_t, \beta\sqrt{\frac{2}{\sigma}}\mathbf{r}_t\right)_{0 \le t \le 1}. \end{equation}
The convergence (\ref{eq:main}) holds in distribution in $\mathbf{C}([0,1], \R^4)$ endowed with the topology of uniform convergence if, additionally, 
%%%% annoying hack to get square brackets on assumption %%%
\makeatletter
\def\tagform@#1{\maketag@@@{[\ignorespaces#1\unskip\@@italiccorr]}}
\makeatother
%%%%%%%%%%%%%%%%%%%%%%%%%%%%
\begin{equation}\leqnomode \tag{\bf A2}\label{a2} \p{\max_{1 \le i \le \xi} |Y_{\xi,i}| >y} = o(y^{-4}) \text{ as $y\rightarrow \infty$ and } \E{\xi^3} <\infty.
\end{equation}
%%%% annoying hack to return to round brackets %%%
\makeatletter
\def\tagform@#1{\maketag@@@{(\ignorespaces#1\unskip\@@italiccorr)}}
\makeatother
%%%%%%%%%%%%%%%%%%%%%%%%%%%%
\end{thm}

Theorem~\ref{thm:main} follows immediately from Corollary~\ref{cor:4.2} and Proposition~\ref{prop:tightness} below. The analogue of Theorem \ref{thm:main} holds with $\R^d$-valued displacements for $d > 1$, and with essentially identical proofs to those in the current work; the only change in the conclusion is that the limit $\mathbf{r}$ of the rescaled spatial displacements takes values in $\mathbb{R}^d$ rather than $\R$, and that $\beta^2 $ should be interpreted as the covariance matrix of $Y_{\bar{\xi},U_{\bar{\xi}}}$.

Let $\Phi_n(i)$ be the random walk trajectory associated with path from the root to the $i$-th vertex visited in the contour exploration of $\bT_n$, for $0 \le i \le 2n$. Then $(\widetilde{H}_n, \Phi_n)$ is the \emph{discrete snake driven by~$\widetilde{H}_n$}. By the homeomorphism theorem of Marckert and Mokkadem (Theorem 2.1 of \cite{homeomorphism_thm}), Theorem \ref{thm:main} entails also that $(\widetilde{H}_n, \Phi_n)$ has the BSBE as its scaling limit; see Figure \ref{fig:snake} for an illustration.

By \cite[Corollary 2.5.1]{le2002random} and \cite{marckert} it turns out that the convergence of the parametrisations of the head of the snake via the height and contour processes are essentially equivalent. In particular, in order to prove Theorem \ref{thm:main}, it suffices to show that, under assumption [\ref{a1}], we have
\begin{equation}\label{eq:suffice}
\left(\frac{H_n(nt)}{\sqrt{n}}, \frac{R_n(nt)}{n^{1/4}}\right)_{0 \le t \le 1} \convdist \left(\frac{2}{\sigma}\mathbf{e}_t, \beta\sqrt{\frac{2}{\sigma}} \mathbf{r}_t\right)_{0 \le t \le 1},
\end{equation}
as $n\rightarrow \infty$ in the sense of random finite-dimensional distributions, and  
in $\mathbf{C}([0,1], \R^2)$ endowed with the topology of uniform convergence under the additional assumption [\ref{a2}]. 

It is not clear to us whether the requirement that $\E{\xi^3} < \infty$ in [\ref{a2}] is necessary or just an artefact of our approach to proving tightness. We shall see in the next subsection that the tail condition in [\ref{a2}] \emph{is} necessary.

\subsection{Necessity of the tail condition} 
If we adjust assumption [\ref{a2}] to allow for heavier tails, displacements start to appear near the leaves which are not negligible on the scale $n^{1/4}$. In this case, one can no longer expect a continuous limit process. This is a consequence of the following proposition on the largest displacement in the tree, which we prove in Section \ref{app:intro} in the appendix.
 
\begin{prop}\label{prop:big_displacement}
Let $\mu=(\mu_k)_{k\ge 0}$ be a critical offspring distribution with variance $\sigma^2\in (0,\infty)$. If $\limsup_{y\to \infty}y^4\p{\max_{1\le i \le \xi}|Y_{\xi,i}|>y }>0$, then there exists a $\delta>0$ such that 
\[
\limsup_{n\to \infty}\p{\max_{v \in v(\rT_n) \setminus \partial \rT_n} \max_{j \ge 1}|Y_{j}^{(v)}|>\delta n^{1/4}}>0.
\]
\end{prop}

We may however still obtain a global convergence result on the appropriate scale, under an additional regularity assumption; see [\ref{a3}] below.  Since the large displacements from a vertex with $k$ children need not be independent, in this setting the limit depends on the joint distribution of the displacements from a vertex to its children. 
For $k \ge 1$ and $j\in [k]$ denote by
\[
Y_{k,j}^+ := Y_{k,j} \vee 0 \quad \text{ and } \quad Y_{k,j}^- := (-Y_{k,j}) \vee 0
\]
the positive and negative displacements of the $j$-th child of a vertex with $k$ children, respectively. Further let $Y_k^+ := (Y_{k,j}^+)_{j\in [k]}$ and $Y_k^- = (Y_{k,j}^-)_{j\in [k]}$. 
%%%% annoying hack to get square brackets on assumption %%%
\makeatletter
\def\tagform@#1{\maketag@@@{[\ignorespaces#1\unskip\@@italiccorr]}}
\makeatother
%%%%%%%%%%%%%%%%%%%%%%%%%%%%
\medskip 
\begin{equation}\leqnomode \tag{\bf A3}\label{a3}
 \parbox{0.89\textwidth}{
Suppose that $\E{\xi^3} < \infty$. Furthermore, suppose that there exists a Borel measure $\pi$ on $\R_+^2\setminus \{(0,0)\}$ such that for any $\varepsilon>0$, both $\pi(\R_+\times(\varepsilon,\infty))<\infty$ and $\pi((\varepsilon,\infty)\times\R_+)<\infty$, and there exists $\eta \in [0,2)$ such that for all Borel sets $A \subset \R_+^2 \setminus \{(0,0)\}$ for which $\pi(\partial A)=0$, as $r \to \infty$
\[
r^{4-\eta} \p{ \frac{1}{r} \left(\max_{1 \le i \le \xi} Y^+_{\xi,i}, \max_{1 \le i \le \xi} Y^{-}_{\xi,i}\right) \in A} \to \pi(A).
\]
}
 \end{equation}
 \medskip
%%%% annoying hack to get square brackets on assumption %%%
\makeatletter
\def\tagform@#1{\maketag@@@{(\ignorespaces#1\unskip\@@italiccorr)}}
\makeatother
%%%%%%%%%%%%%%%%%%%%%%%%%%%%

We note the following lemma, whose proof may be found in Section \ref{app:intro} in the appendix. 

\begin{lem} \label{lem:noatoms}
 {[\ref{a3}]} implies that the projection of $\pi$ onto either of its coordinates has no atom in $(0,\infty)$.
\end{lem}

Under assumption [\ref{a3}] we prove convergence results for the head of the discrete snake in the space of non-empty compact subsets of $[0,1] \times \R$ equipped with the Hausdorff topology. In what follows, for a continuous function $f:[0,1]\rightarrow \R$ and a set $S \subset [0,1] \times \R_+^2 \setminus \{(0,0)\}$, write $U(f,S)$ for the union of the graph of $f$ and the vertical line segments $[(t,f(t)-y),(t,f(t)+x)]$ for each $(t,x,y) \in S$. The next theorem relates to the case $\eta = 0$ in [\ref{a3}].

\begin{thm} \label{thm:hairy_4}
Let $\mu = (\mu_k)_{k\ge 0}$ be a critical offspring distribution with  variance $\sigma^2 \in (0,\infty)$, and let $\nu = (\nu_k)_{k \ge 1}$ be such that [\ref{a1}] holds and [\ref{a3}] holds for a given measure $\pi$ with $\eta = 0$. Then, taking $\Xi$ to be a Poisson process on $[0,1]\times \R_+^2\setminus\{(0,0)\}$ with intensity $dt \otimes \pi(dx,dy)$, we have 
\begin{equation}\label{eq:hairy_4}\left(\left(\frac{H_n(nt)}{\sqrt{n}}\right)_{0 \le t \le 1},~ U\left(\frac{R_n(n\cdot)}{n^{1/4}}, \emptyset\right)\right) \convdist \left(\left(\frac{2}{\sigma}\mathbf{e}_t\right)_{0 \le t \le 1},~ U\left(\beta\sqrt{\frac{2}{\sigma}} \mathbf{r}, \Xi\right)\right),\end{equation}
as $n\rightarrow \infty,$ where the convergence in the first coordinate  is in $\mathbf{C}([0,1],\R)$ endowed with the topology of uniform convergence, and the convergence in the second is in the space of non-empty, compact subsets of $[0,1]\times \R$ endowed with the Hausdorff topology.
\end{thm}

We refer to the object on the right-hand side of (\ref{eq:hairy_4}) as the \emph{hairy tour}, in keeping with the previous work of Janson and Marckert \cite{discrete_snakes}.

When $\eta\in (0,2)$, the large jumps dominate the smaller ones to such an extent that, in the limit, we obtain a pure jump process.

\begin{thm} \label{thm:hairy_2} 
Let $\mu = (\mu_k)_{k\ge 0}$ be a critical offspring distribution with  variance $\sigma^2 \in (0,\infty)$, and let $\nu = (\nu_k)_{k \ge 1}$ be such that [\ref{a1}] holds and [\ref{a3}] holds for a given measure $\pi$ with $\eta \in (0,2)$. Then, taking $\Xi$ to be a Poisson process on $[0,1]\times \R_+^2\setminus\{(0,0)\}$ with intensity $dt \otimes \pi(dx,dy)$, we have 
\[
\left(\left(\frac{H_n(nt)}{\sqrt{n}}\right)_{0 \le t \le 1},~ U\left(\frac{R_n(n\cdot)}{n^{1/(4-\eta)}},\emptyset\right) \right) \convdist \left(\left(\frac{2}{\sigma}\mathbf{e}_t\right)_{0 \le t\le 1},~ U(0, \Xi)\right),
\]
as $n\rightarrow \infty,$ where the convergence in the first coordinate  is in $\mathbf{C}([0,1],\R)$ endowed with the topology of uniform convergence, and the convergence in the second is in the space of non-empty, compact subsets of $[0,1]\times \R$ endowed with the Hausdorff topology.
\end{thm}

In contrast to Theorem \ref{thm:main}, in Theorems \ref{thm:hairy_4} and \ref{thm:hairy_2} we need the condition $\E{\xi^3} < \infty$ not just for tightness but also for the convergence of the random finite-dimensional distributions. The reason for this is that we apply a quantitative local central limit theorem which requires a third moment on the offspring distribution. (See Theorem \ref{lem:local_move} for the precise statement.) 

The fact that we obtain a continuous function decorated by intervals in both Theorems~\ref{thm:hairy_4} and~\ref{thm:hairy_2} is really an artefact of the choice to interpolate $R_n$ linearly between integer times. Indeed, the endpoints of the intervals capture the asymptotic behaviour of the two extremities of the displacements away from vertices, but tell us nothing about how the ``point process'' of displacements in between behaves. If we instead consider the graph of $\left(\frac{R_n(\lfloor nt \rfloor)}{n^{1/4}}\right)_{0 \le t \le 1}$ in the case where we do not have $\p{\max_{1 \le i \le \xi}|Y_{\xi,i}| > y} = o(y^{-4})$ there are, in fact, many possible behaviours. We will not undertake any sort of exhaustive classification here, but let us give a couple of illustrative examples. 

Suppose first that the displacements are simply \iid copies of a random variable $Y$ such that, for some Borel measure $\pi$ on $\R_+\setminus \{0\}$ such that for any $\varepsilon>0$, $\pi((\varepsilon,\infty))<\infty$, we have $r^{4} \p{Y \in r A} \to \pi(A)$ as $r \to \infty$ for every Borel set $A \subset \R \setminus \{0\}$ such that $\pi(\partial A)=0$. Then we will not, in the limit, observe two or more $\Theta(n^{1/4})$ displacements away from the same vertex of $\rT_n$ (nor, indeed, from vertices at distance $o(n^{1/2})$ from one another), and so we just obtain the graph of $\mathbf{r}$ decorated by isolated points which occur as a Poisson process of intensity $dt\otimes \pi(dy)$ on $[0,1] \times \R \setminus \{0\}$. 

On the other hand, suppose that we have the following deterministic displacements:
\begin{equation}\label{eq:deter_disp}
Y_{k,j} = \sigma - \frac{2}{\sigma}(k-j) \text{ for } \quad 1 \le j \le k. 
\end{equation}
These displacements have a particular significance, which we will discuss in the next subsection. For the moment, let us just observe that it is straightforward to check that they are globally centered and of finite global variance whenever the offspring distribution is critical and admits a finite third moment. Suppose that there exists a Borel measure $\pi_1$ on $(0,\infty)$ with $\pi_1((\varepsilon,\infty))<\infty$ for all $\varepsilon > 0$, such that $r^4\p{\xi \in r A} \to \pi_1(A)$ as $r \to \infty$ for any Borel set $A \subset (0,\infty)$ with $\pi_1(\partial A)=0$. Then all of the children of a vertex with $\Theta(n^{1/4})$ children will have $\Theta(n^{1/4})$ displacements which are regularly spaced with spacing $2/\sigma$. Again, with high probability, we will not see two vertices of degree $\Theta(n^{1/4})$ within distance $o(n^{1/2})$ in~$\rT_n$. So in the limit for the graph of $\left(\frac{R_n(\lfloor nt \rfloor)}{n^{1/4}}\right)_{0\le t\le 1}$ we will see decorations driven by a Poisson process on $[0,1] \times \R_+$ with intensity $dt\otimes\pi(dx)$ such that when we observe a point $(t,x)$ of the Poisson process, we attach the whole interval $[-2x/\sigma,0]$ to the graph of $\mathbf{r}$ at~$t$.

\subsection{Related work}

As mentioned earlier, versions of the topic studied in this paper have received extensive attention in the literature. One reason for this is that discrete snakes play a crucial role in the study of random planar maps; see \cite{addario2021convergence,chassaing2004random, le_gall, miermont, invariance_principles, albenque}. 

The earliest discrete snake convergence results were proved in models with a fixed offspring distribution. Chassaing and Schaeffer~\cite{chassaing2004random} treated the setting of a Geometric($1/2$) offspring distribution (which results in uniformly random planar trees) with \iid displacements uniform on $\{-1,0,1\}$. Marckert and Mokkadem~\cite{homeomorphism_thm} treated the same offspring distribution, but where the displacements away from a vertex all have the same centered marginal distribution (but may depend on one another) with a $6+\varepsilon$ moment. Gittenberger~\cite{gittenberger2003note} later generalised these results to critical, finite variance offspring distributions with centered (but not necessarily \iids) displacements having finite $8+\varepsilon$ moment.

Work on the \iid displacement case culminated in a paper of Janson and Marckert~\cite{discrete_snakes} which established the following result.

\begin{thm}[\cite{discrete_snakes}, Theorems 1 and 2]\label{thm:janson_markert}
Let $\mu = (\mu_k)_{k\ge 0}$ be a critical offspring distribution with variance $\sigma^2 \in (0,\infty)$ such that $\mu$ has a finite exponential moment. For each $k\ge 1$, let $\nu_k$ be the law of a vector of $k$ \iid copies of a random variable $Y$ with $\E{Y} = 0$ and $\E{Y^2} = \beta^2\in (0,\infty)$. Then
\begin{equation}
\left(\frac{\widetilde{H}_n(2nt)}{\sqrt{n}}, \frac{\widetilde{R}_n(2nt)}{n^{1/4}} \right)_{0 \le t \le 1}\convdist \left(\frac{2}{\sigma}\mathbf{e}_t, \beta\sqrt{\frac{2}{\sigma}} \mathbf{r}_t\right)_{0 \le t \le 1}
\end{equation}
as $n \to \infty$, in the sense of finite-dimensional distributions. The convergence also holds in distribution in $\mathbf{C}([0,1], \R^2)$ endowed with the topology of uniform convergence if and only if 
\begin{equation} 
\p{|Y| >y} = o(y^{-4}) \text{ as $y\rightarrow \infty$.}
\end{equation}
\end{thm}

The finite exponential moment condition on the offspring distribution has subsequently been shown to be unnecessary, and may be weakened to a finite second moment assumption; see, for example, Marzouk~\cite{stable_trees}. Our Theorem~\ref{thm:main} recovers this theorem under the additional assumption of a finite third moment for $\mu$ (and replacing convergence in the sense of finite-dimensional distributions in the first statement with random finite-dimensional distributions).

Janson and Marckert~\cite{discrete_snakes} also considered what happens in some of the ``heavy-tailed'' cases for which the tail condition $\p{|Y|>y} = o(y^{-4})$ fails. In particular, they considered the setting in which 
\[
\p{Y \ge y} \sim a_+ y^{-q}, \quad \p{Y \le -y} \sim a_{-} y^{-q} \quad \text{as $y \to \infty$}
\]
for some constants $a_+, a_- \ge 0$ and $q \in (2,4]$, and prove analogues of Theorems \ref{thm:hairy_4} and \ref{thm:hairy_2} in such cases. They call the limiting object in this setting the \emph{hairy tour}, and the associated snake the \emph{jumping snake}. Their results were an important inspiration for Theorems \ref{thm:hairy_4} and \ref{thm:hairy_2}. 

Marzouk~\cite{stable_trees} later extended Janson and Marckert's results in \cite{discrete_snakes} to the situation where the offspring distribution is in the domain of attraction of a stable law, and the displacements are \iids. 

Returning now to non-\iid displacements, there are several notions of centering and finite variance which have been imposed in order to obtain convergence to the BSBE. Marckert and Miermont~\cite{invariance_principles} worked under the ``local centering'' assumption that $\E{Y_{k,j}} = 0$ for all $1 \le j \le k$. For multi-type Bienaymé trees, \cite{addario2021convergence} establishes convergence of discrete snakes under assumptions that impose in particular that the displacements away from vertices of each type are centered.

Most closely related to our results is a paper of Marckert~\cite{marckert_snakes}, which proves the following theorem. 

\begin{thm}[\cite{marckert_snakes}, Theorem 1] \label{thm:marckert}
Let $\mu = (\mu_k)_{k \ge 0}$ be a critical offspring distribution with $\mu_0 + \mu_1 < 1$ and with bounded support. Suppose further that $\nu = (\nu_k)_{k \ge 1}$ is such that
\begin{equation*}
\E{Y_{\bar\xi, U_{\bar\xi}}} = 0 \quad \text{and} \quad \beta^2 = \E{(Y_{\bar\xi, U_{\bar\xi}})^2}  < \infty,
\end{equation*}
and that there exists $p > 4$ such that
\[
\sup_{1 \le j \le k \le K} \E{|Y_{k,j} - \E{Y_{k,j}}|^p} < \infty.
\]
where $\mu$ is supported by $\{0,\dots,K\}$. Then, as $n\rightarrow \infty$, 
\begin{equation}
\left(\frac{H_n(nt)}{\sqrt{n}}, \frac{R_n(nt)}{n^{1/4}}, \frac{\widetilde{H}_n(2nt)}{\sqrt{n}}, \frac{\widetilde{R}_n(2nt)}{n^{1/4}} \right)_{0 \le t \le 1}\convdist \left(\frac{2}{\sigma}\mathbf{e}_t, \beta\sqrt{\frac{2}{\sigma}} \mathbf{r}_t, \frac{2}{\sigma}\mathbf{e}_t, \beta\sqrt{\frac{2}{\sigma}} \mathbf{r}_t\right)_{0 \le t \le 1}
\end{equation}
in $\mathbf{C}([0,1], \R^4)$ endowed with the topology of uniform convergence.
\end{thm}

The boundedness condition is a necessary requirement of Marckert's proof technique, which is a \emph{tour de force} involving tracking very detailed information about the number of vertices of each possible different degree along a lineage, which converge on appropriate rescaling to a Gaussian field. Our approach removes the boundedness requirement, but we do not obtain such fine information on the limit object.

Finally, we mention a forthcoming work of Duquesne and Rebei \cite{duquesne25}, which proves limit theorems for snakes whose jumps are centered and sibling-independent and such that the underlying family tree converges to a Lévy tree. Our understanding is that the results and technique of \cite{duquesne25} are rather different from those of the current work.

\subsection{A first application}\label{sec:appl1}
One nice consequence of Theorem \ref{thm:main} is a strengthening of a result of Marckert and Mokkadem \cite{marckert} concerning the difference between the height process, $H_n$, and the {\L}ukasiewicz path, here denoted by $W_n$ (and formally defined in Section \ref{sec:background}) of $\rT_n$. It is proved in \cite{marckert} that if $\xi$ is critical with variance $\sigma^2 \in (0,\infty)$ and has a finite exponential moment, then 
\[\left(\frac{\widetilde{H}_n(2nt)}{\sqrt{n}}, \frac{H_n(nt)}{\sqrt{n}}, \frac{W_n(nt)}{\sqrt{n}}\right)_{0 \le t \le 1} \convdist \left(\frac{2}{\sigma}\mathbf{e}_t, \frac{2}{\sigma}\mathbf{e}_t, \sigma \mathbf{e}_t\right)_{0 \le t \le 1}\]
as $n \to \infty$ in $\mathbf{C}([0,1], \R^3)$. (As mentioned after Theorem \ref{thm:janson_markert}, the finite exponential moment condition is unnecessary and may be removed; see Duquesne~\cite{duquesne2003limit} for this result in the context of trees rather than snakes.)

Moreover, under the same assumptions, \cite{marckert} establishes that, for any $\varepsilon > 0$, there exists $\gamma > 0$ such that for $n > 0$ sufficiently large 
\[
\p{\sup_{0 \le i \le n} \bigg|\sigma H_n(i) - 2\sigma^{-1}W_n(i)\bigg| \ge n^{1/4 + \varepsilon}} \le \exp\left(-\gamma n^{\varepsilon}\right).
\]
It is natural to conjecture that, under suitable conditions, the difference varies precisely on the order of $n^{1/4}$. We are able to prove this conjecture in a large degree of generality. It turns out that the difference $(\sigma H_n(i) - 2\sigma^{-1}W_n(i), 0 \le i \le n)$ evolves precisely as the head of a discrete snake (see Lemma \ref{lem:diff_height_luka} for a proof of this fact). 
The relevant displacements are given by $Y_{k,j} = \sigma - (2/\sigma)(k-j)$; this formula already appeared at (\ref{eq:deter_disp}). We have
\[
\sum_{k=1}^{\infty} \mu_k \sum_{j=1}^k \E{Y_{k,j}} = \sum_{k=1}^{\infty} \mu_k \sum_{j=1}^k \left(\hspace{-1pt}\sigma - \frac{2}{\sigma}(k-j)\hspace{-1pt}\right) 
= \sum_{k=1}^{\infty} \mu_k \left(\hspace{-1pt}\sigma k - \frac{k(k-1)}{\sigma}\hspace{-1pt} \right) \hspace{-1pt}= \sigma - \frac{\sigma^2}{\sigma} \hspace{-1pt}= \hspace{-1pt}0,
\]
so that the associated discrete snake is globally centered. Moreover, the global variance is
\begin{align*}
\sum_{k=1}^{\infty} \mu_k \sum_{j=1}^k \E{Y_{k,j}^2} & = \sum_{k=1}^{\infty} \mu_k \sum_{j=1}^k \left(\sigma - \frac{2}{\sigma}(k-j)\right)^2 \\
& = \sum_{k=1}^{\infty} \mu_k \left(\sigma^2 k - 2 k(k-1) + \frac{2}{3 \sigma^2} k(k-1)(2k-1) \right) \\
& = \frac{4}{3\sigma^2}(\E{\xi^3} - 1) - (\sigma^2 + 2),
\end{align*}
which is finite provided that $\E{\xi^3}<\infty$. Also,
\[
\p{\max_{1\le i \le \xi}|Y_{\xi, i}| > y } = \p{\left|\sigma - \frac{2}{\sigma}(\xi - 1)\right|\vee \sigma > y} = o(y^{-4})
\]
as $y \to \infty$ if and only if $\p{\xi > y} = o(y^{-4})$ as $y \to \infty$;
moreover, the latter condition implies $\E{\xi^3} < \infty$. We obtain the following corollary of Theorem \ref{thm:main}. 
\begin{cor}\label{cor:height_queue}
Let $\mu = (\mu_k)_{k \ge 1}$ be a critical offspring distribution with variance $\sigma^2 \in (0,\infty)$. Let $\beta^2 = \frac{4}{3\sigma^2}(\E{\xi^3} - 1) - (\sigma^2 + 2)$. Then
\[\left(\frac{H_n(nt)}{\sqrt{n}}, \frac{\sigma H_n(nt) - 2\sigma^{-1}W_n(nt) }{n^{1/4}}\right)_{0 \le t\le 1} \convdist \left(\frac{2}{\sigma}\mathbf{e}_t, \beta\sqrt{\frac{2}{\sigma}} \mathbf{r}_t\right)_{0 \le t \le 1},\]
as $n\rightarrow \infty$ in $\mathbf{C}([0,1], \R^2)$ if and only if $\p{\xi > y} = o(y^{-4})$ as $y \to \infty$.
\end{cor}

Let us observe that, while Corollary \ref{cor:height_queue} concerns the difference between the height process and the {\L}ukasiewicz path, the joint convergence in Theorem \ref{thm:main} can be used to prove an analogous result for the difference between the {\L}ukasiewicz path and the contour process encoding of the head of the same discrete snake. (We leave the details of this statement to the reader.)

In the case where $\xi$ is bounded, Marckert's result (Theorem~\ref{thm:marckert}) applies, so the corollary is new only in the case of unbounded offspring distributions. In an earlier paper \cite{marckert2004rotation}, Marckert had already observed that the difference between the left and right pathlengths (also known as the \emph{imbalance}) of a size-conditioned Bienaymé tree with offspring distribution $\mu_0 = \mu_2 = 1/2$ converges in distribution after rescaling to $2^{1/4}S$, where $S = \int_0^1 \mathbf{r}_t dt$. We note that such trees are binary, and recall that the left pathlength (resp.\ right pathlength) of a vertex $v$ is the number of vertices in its ancestral lineage who precede (resp.\ succeed) their siblings in the lexicographical order. The left (resp.\ right) pathlength of binary trees is then the sum of the left (resp.\ right) path lengths over all vertices in the tree. Janson~\cite{left_right} later used the method of moments to give an alternate proof of this convergence in distribution. 

It can also be the case that the sequence $(n^{-1/4}\max_{0 \le i \le n} |\sigma H_n(i) - 2\sigma^{-1} W_n(i)|)_{n \ge 1}$ is tight \textit{without} converging in distribution to the maximum modulus of the head of the BSBE; indeed, by Theorem~\ref{thm:hairy_4}, if   $r^4 \p{\xi \in rA} \to \pi_1(A)$ as $r \to \infty$ for all Borel sets $A$ such that $\pi_1(\partial A) = 0$ and a Borel measure $\pi_1$ on $\R_+\setminus \{0\}$ such that for any $\varepsilon>0$, $\pi_1((\varepsilon,\infty))<\infty$, then it is possible to prove that $n^{-1/4}\max_{0 \le i \le n} |\sigma H_n(i) - 2\sigma^{-1}W_n(i)|$ converges in distribution to the maximum modulus of the appropriate hairy tour. If, on the other hand, we have $r^{4-\eta} \p{\xi \in rA} \to \pi_1(A)$ as $r \to \infty$ for some $\eta \in (0,2)$, then Theorem \ref{thm:hairy_2} yields the convergence 
\[
n^{-1/(4-\eta)}\max_{0 \le i \le n} |\sigma H_n(i) - 2\sigma^{-1} W_n(i)| \convdist L,
\]
where $\p{L \le \ell} = \exp(-\int_{\ell}^{\infty} \pi_1(x) dx)$ is the probability that no point of a Poisson point process of intensity $dt \pi_1(dx)$ on $[0,1] \times \R_+$ has second co-ordinate greater than $\ell$.
\subsection{A second application} A second consequence of Theorem \ref{thm:main} concerns the difference between the height process of $\rT_n$ and the height process of the corresponding \emph{looptree}. The looptree corresponding to $\rT_n$, denoted by $\rT_n^\circ$, is the connected multigraph obtained by replacing the edges from a vertex to its children by a cycle going through the parent and all of its children in order (whose length, therefore, equals its number of children plus one). See Figure \ref{fig:loop_tree} for an illustration. (It turns out that it is possible to make sense of a continuum analogue of this notion, as proved by Curien and Kortchemski \cite{igor_nic_looptrees}.) 

\begin{figure}[h!]
\includegraphics[page=10]{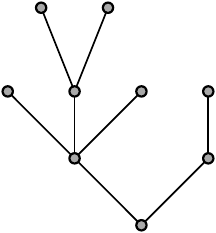}\hspace{2em}
\includegraphics[page=7,scale=0.7]{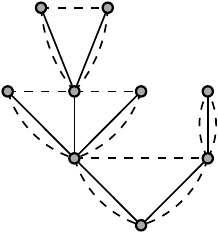}\\
\vspace{1em}
\includegraphics[page=8,scale=0.7]{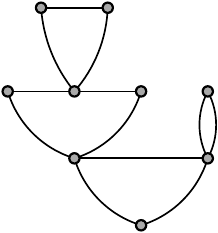}\hspace{2em}
\includegraphics[page=11]{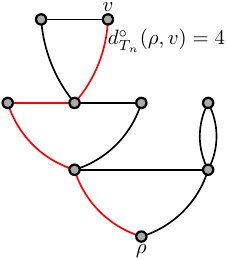}
\caption{In the top left figure, a tree, and in the bottom left figure its corresponding looptree. The top-right figure serves to aid in understanding the construction, and the bottom-right figure illustrates how distances are calculated in the loop-tree.}
\label{fig:loop_tree}
\end{figure}

Vertices in the original tree naturally correspond to vertices in the looptree. Let $v_1,\dots, v_n$ be the vertices of $\rT_n$ listed in lexicographical order. We define the height function of the looptree, denoted $H^\circ_n: [0,n]\rightarrow \R$, to give the graph distance between the root and each of the vertices in the looptree, visited in the order $v_1, \ldots, v_n$. This is the height process (in the usual sense) of the spanning tree of the looptree made up of the union of the geodesic paths from each of its vertices to the root. Formally, using the Ulam--Harris notation (see Section \ref{sec:background} for details) for $0 \le i \le n-1$ let 
\[H_n^\circ(i) := \sum_{(u,uj)\in e(\rT_n)~:~uj \preceq v_{i+1}} \min\{j, c(u,\rT_n) + 1 - j\},\]
where for $u\in v(\rT_n)$, $c(u,\rT_n)$ denotes the number of children of $u$ in $\rT_n$. Finally let $H_n^\circ(n) = 0$, and extend the domain to $[0,n]$ by linear interpolation. For $c\in \R$, it is readily seen that the difference $\left(cH_n(i)-H^\circ_n(i),0\le i \le n \right) $  evolves as the head of a discrete snake whose displacements are given by 
\[
Y_{k,j}=c-\min\{j,k+1-j\}.
\]
Moreover, if we fix $c = \frac{1}{4}\E{\xi^3} + \frac{1}{2} + \frac{1}{4}\p{\xi \in 2\Z + 1}$, then 
\begin{align*}
\sum_{k=1}^{\infty} \mu_k \sum_{j=1}^k \E{Y_{k,j}} &= \sum_{k=1}^{\infty} \mu_k \sum_{j=1}^k \left(c - \min\{j,k+1-j\}\right) \\
&=\sum_{k=1}^\infty\mu_k \left(ck-2\sum_{i=1}^{\lfloor k/2 \rfloor }i-\left\lceil \frac{k}{2} \right\rceil \I{k\in 2\Z +1} \right) \\
&= \sum_{k=1}^{\infty} \mu_k \left(c k - \frac{k}{2}\left(\frac{k}{2}+1 \right) - \frac{1}{4}\I{k \in 2\Z +1}\right)\\
&= c - \frac{1}{4}\E{\xi^3} - \frac{1}{2} - \frac{1}{4}\p{\xi \in 2\Z + 1}=  0,
\end{align*}
so that the associated discrete snake is globally centered. Moreover, the global variance is 
\begin{align}\label{eq:loop_tree_var}
& \sum_{k=1}^{\infty} \mu_k \sum_{j=1}^k \E{Y_{k,j}} \\
& \qquad = \sum_{k=1}^{\infty} \mu_k \sum_{j=1}^k \left(c - \min\{j,k+1-j\}\right)^2 \nonumber\\
& \qquad =\sum_{k=1}^\infty\mu_k \left(c^2 k -c k\left(\frac{k}{2}+1 \right)-\frac{c}{2}\I{k \in 2\Z +1}+2\sum_{i=1}^{\lfloor k/2 \rfloor } i^2+\left\lceil \frac{k}{2} \right\rceil^2 \I{k\in 2\Z +1}\right)\nonumber\\
&\qquad =\sum_{k=1}^\infty\mu_k \left(c^2 k +k\left(\frac{k}{2}+1 \right)\left(\frac{k}{6}+\frac{1}{6} -c\right)+\left(\frac{k^2}{12}+\frac{k}{3}+\frac{1}{4}-\frac{c}{2}\right) \I{k\in 2\Z +1}\right)\nonumber\\
&\qquad =c^2+\frac{\E{\xi^3}}{12}+\left(\frac{1}{4}-\frac{c}{2} \right)\E{\xi^2}+\frac{1}{6}-c +\E{\left(\frac{\xi^2}{12}+\frac{\xi}{3}+\frac{1}{4}-\frac{c}{2}\right)\I{\xi\in 2\Z +1} }\nonumber\\
&\qquad=:\beta^2,\nonumber
\end{align}
\noindent which is finite provided $\E{\xi^3} < \infty$.

Finally, 
\[
\p{\max_{1\le i \le \xi}|Y_{\xi, i}| > y } = \p{\left|c-\lceil \xi/2 \rceil \right|\vee |c-1| > y} = o(y^{-4})
\]
as $y \to \infty$ if and only if $\p{\xi > y} = o(y^{-4})$ as $y \to \infty$
and, moreover, the latter condition implies $\E{\xi^3} < \infty$. We obtain the following corollary of Theorem \ref{thm:main}.
\begin{cor}\label{cor:height_queue_loop}
Let $\mu = (\mu_k)_{k \ge 1}$ be a critical offspring distribution with variance $\sigma^2 \in (0,\infty)$, and let $\xi$ be a random variable with distribution $\mu$. Let $c = \frac{1}{4}\E{\xi^2} + \frac{1}{2} + \frac{1}{4}\p{\xi \in 2\Z + 1}$ and $\beta^2$ be as in (\ref{eq:loop_tree_var}). Then,
\begin{equation}\label{eq:looptree}\left(\frac{H_n(nt)}{\sqrt{n}}, \frac{H_n^\circ(nt)}{\sqrt{n}}, \frac{cH_n(nt) - H^\circ_n(nt) }{n^{1/4}}\right)_{0 \le t\le 1} \convdist \left(\frac{2}{\sigma}\mathbf{e}_t, \frac{2}{\sigma}\mathbf{e}_t,\sqrt{\frac{2}{\sigma}} \beta \mathbf{r}_t\right)_{0 \le t \le 1},\end{equation}
as $n\rightarrow \infty$ in $\mathbf{C}([0,1], \R^3)$ endowed with the topology of uniform convergence if and only if $\p{\xi > y} = o(y^{-4})$ as $y \to \infty$.
\end{cor}

Analogues of this result also hold in the settings of Theorems \ref{thm:hairy_4} and \ref{thm:hairy_2}.
Even the functional convergence for the height process of looptrees of Bienaymé trees, expressed in the second coordinate of \eqref{eq:looptree}, is new, although pointwise convergence was proved by Kortchemski and Marzouk \cite{stable_looptrees}. (The convergence of looptrees in the Gromov--Hausdorff topology was proved in \cite{looptrees_CRT} via spinal decomposition --- see Theorem 1.2 and the generic case in Corollary 1.4 for the application to maps --- and convergence in the Gromov--Hausdorff--Prokhorov topology was shown in \cite[Theorem 15]{Robin_looptrees}.)

\subsection{Overview of the proofs}\label{sec:proof_overview} 
We will prove weak convergence of the head of the snake by making use of the following variant of the usual formulation of weak convergence for a sequence of random continuous functions. (This formulation is inspired by Theorem 20 of \cite{continuum_3}, and can be proved by essentially the same method as the second proof of Theorem~7.5 of \cite{billingsley2013convergence}.)

\begin{prop} \label{prop:rfdds+rtightness}
Let $(f_n)_{n \ge 1}$ and $f$ be random elements of $\mathbf{C}([0,1],\R)$ such that $f_n(0)=f_n(1)=0$ for every $n \ge 1$ and $f(0)=f(1)=0$. Let $U_1, U_2, \ldots$ be \iid $U[0,1]$ random variables, independent of $(f_n)_{n\ge 1}$ and $f$. For $k \ge 1$, write $U_{(1)}^k, U_{(2)}^k,\ldots,U_{(k)}^k$ for the values of $U_1, U_2, \ldots, U_k$ written in increasing order, and set $U_{(0)}^k = 0$ and $U_{(k+1)}^k = 1$. 

Suppose that for each $k \ge 1$ we have
\begin{equation} \label{eq:rfdds}
\left(f_n(U_{(1)}^{k}), \ldots, f_n(U_{(k)}^{k})\right) \convdist \left(f(U_{(1)}^k), \ldots, f(U_{(k)}^k)\right),
\end{equation}
as $n \to \infty$, and that for any $\varepsilon > 0$,
\begin{equation} \label{eq:rtightness}
\lim_{k \to \infty} \limsup_{n \to \infty} \p{\max_{0 \le i \le k} \sup_{s,t \in [U_{(i)}^{k}, U_{(i+1)}^{k}]} |f_n(s) - f_n(t)| > \varepsilon} = 0.
\end{equation}
Then
$f_n \convdist f$ 
as $n \to \infty$, for the topology generated by the uniform norm on $\mathbf{C}([0,1],\R)$.
\end{prop}

We will refer to assumption (\ref{eq:rfdds}) as the \emph{convergence of random finite-dimensional distributions} and to (\ref{eq:rtightness}) as \emph{tightness}. Observe that (\ref{eq:rfdds}) is weaker than the usual convergence of finite-dimensional distributions. However, it is more natural in the context of random trees, and indeed plays a key role in Aldous' theory of continuum random trees as developed in \cite{continuum_3}. (See the appendix of \cite{AldousPitman} for a discussion and for further references.) 

Let $\cT_{2\mathbf{e}}$ be the real tree encoded by $2\mathbf{e}$, where $\mathbf{e}$ is a normalised Brownian excursion. (We refer to the survey of Le Gall~\cite{le_gall_tree_survey} for standard definitions concerning random real trees.) Fix $k \ge 1$ and let $U_1,\dots, U_k$ be \iid Uniform$([0,1])$ random variables. Furthermore, let $\cT_{2\mathbf{e}}^k$ be the subtree of $\cT_{2\mathbf{e}}$ spanned by the images of $0$ and of $U_1,\dots, U_k$ in $\cT_{2\mathbf{e}}$.  Formally, it is useful to think of this as an ordered rooted tree with leaves labeled by $1,2,\ldots,k$ and edge-lengths, where we use the relative ordering of $U_1, U_2, \ldots, U_k$ to determine the planar ordering of the leaves. Using Aldous' line-breaking construction \cite{continuum_3}, we may construct a tree which is equal in distribution to $\cT_{2\mathbf{e}}^k$ as follows. 

Let $J_1, \dots, J_k$ be the first $k$ jump times of a Poisson point process on $[0,\infty)$ with intensity~$tdt$ at time $t$. For $i = 1,\dots, k-1$, sample an {\em attachment point} $A_i\sim \mathrm{Uniform}([0, J_i])$, independent of $(A_j)_{j \neq i}$. Take the completion of each of the line segments $[0,J_1], (J_1, J_2],\dots,$ $ (J_{k-1}, J_k]$, and for each $i \in \{1,\dots, k-1\}$ let $J_i^*$ denote the limit point as $x \downarrow J_{i-1}$. Identify the points $J_i^*$ and $A_i$, and think of the line-segment as being attached to the left side of the branch containing $A_i$ with probability $1/2$ and to the right side with probability $1/2$. Denote the resulting rooted ordered tree with leaf-labels and edge-lengths by $\cT^k$. Then, $\cT_{2\mathbf{e}}^k \eqdist \cT^k$; see \cite[p. 279]{continuum_3}. 

The proof of Theorem \ref{thm:main} (and similarly Theorems \ref{thm:hairy_4} and \ref{thm:hairy_2}) relies on proving that a certain discrete line-breaking construction of $\rT_n$, described formally in Section \ref{sec:encodings}, converges to Aldous' line-breaking construction upon rescaling. The discrete construction builds a tree on $[n]$ by first constructing paths $P^{(1)},\dots, P^{(\ell^*)}$ and then attaching them to one another by identifying one endpoint of each path $P^{(i)}$ with a point in $(P^{(j)})_{j < i}$. The proof of convergence of the random finite-dimensional distributions relies on the observation that, along each path, the sequence of partial sums of the displacements is essentially a random walk trajectory with \iid steps with the same distribution as $Y_{\bar\xi, U_{\bar\xi}}$ and, moreover, that random displacements appearing at branch points do not contribute to the displacements of the discrete snake on the ``macroscopic'' spatial scale of $\Theta(n^{1/4})$. 

For the proofs of tightness, we adapt a method of Haas and Miermont \cite{miermont} used to prove tightness for the height process of a Markov branching tree. (Note that size-conditioned Bienaymé trees are examples of Markov branching trees.) Let $\rT_n^k$ be a subtree of $\rT_n$ spanned by its root and $k$ uniform vertices. The difference $\rT_n\setminus \rT_n^k$ is a forest $\mathrm{F}_n^k$, and to prove tightness we bound the maximum modulus of the spatial locations in each tree in $\mathrm{F}_n^k$. Following Haas and Miermont, we reduce this bound to an expression involving only a size-biased pick from among the trees in $\mathrm{F}_n^k$. The proof of tightness then reduces to proving an explicit tail bound for the maximum modulus of the spatial location of a vertex in $\rT_n$ when rescaled by $n^{-1/4}$. As a key part of our argument, we require a strong control on the total variation distance between the laws of $\bar\xi$ and of the number of children of the root of $\rT_n$, which we denote by $\widehat D^n_1$. For $k \in [n]$, by Kemperman's formula \cite[Chapter 6]{pitman2006combinatorial},
\begin{equation}\label{eq:dn_dist}\p{\widehat D^{n}_1 = k} = \left(\frac{n}{n-1}\right)\frac{\p{S_{n-1} = n-1-k}}{\p{S_{n} = n-1}}\p{\bar\xi = k}, \end{equation}
where $(S_n)_{n \ge 1}$ is a random walk with \iid $\mu$-distributed increments. 
 In order to control this total variation distance, we use a version of the local central limit theorem (\cite[Theorem 13, Chapter VII]{petrov} which, for completeness, we also state below in Theorem \ref{thm:gen_clt}) which holds whenever $\E{\xi^3}<\infty$; this is the origin of the third moment condition in our main theorem.

\subsection{Asymptotic notation}
We will use the following notation related to the asymptotics of random variables $(X_n)_{n\ge }\in \R$. (See Janson~\cite{janson2011probability}.) For $(y_n)_{n\ge 1}\in \R_{> 0}$,
\begin{itemize}
    \item $X_n = o_\bP(y_n)$ means that $X_n/y_n \convprob 0$ as $n\rightarrow \infty$;
    \item $X_n = \omega_\bP(y_n)$ means that $X_n/y_n \convprob \infty$  as $n\rightarrow \infty$;
    \item $X_n = O_\bP(y_n)$ means that for all $\varepsilon > 0$, there exist constants $n_\varepsilon$, $C_\varepsilon>0$ such that for all $n \ge n_\varepsilon$, $\p{X_n \le C_\varepsilon y_n} \ge 1 - \varepsilon$;
    \item  $X_n = \Omega_\bP(y_n)$ means that for all $\varepsilon > 0$ there exist constants $n_\varepsilon$, $C_\varepsilon>0$ such that for all $n \ge n_\varepsilon$, $\p{X_n \ge C_\varepsilon y_n} \ge 1 - \varepsilon$; 
    \item $X_n = \Theta_\bP(y_n)$ means that $X_n = O_\bP(y_n)$ and $X_n = \Omega_\bP(y_n)$;
    \item Lastly, ``with high probability'' always means ``with probability tending to $1$ as $n \to \infty$.
\end{itemize}

\section{Trees, branching random walks, and their encodings}\label{sec:background}

We require a number of different tree models, which we now define.

First, a {\em tree} is simply a connected acyclic graph $T=(v(T),e(T))$. A {\em rooted tree} consists of a tree together with a distinguished root vertex $\rho = \rho(T) \in v(T)$.
Given a rooted tree $T$ and a vertex $v$ of $T$ write $C(v,T)$ for the set of children of $v$ in $T$ and $c(v,T)=|C(v,T)|$; vertex~$v$ is a {\em leaf} of $T$ if $c(v,T)=0$. We write $\partial T$ for the set of leaves of $T$. Also, for a non-root vertex $v$ we write $p(v)=p(v,T)$ for the parent of $v$ in~$T$. For vertices $v,w \in v(T)$ we write $v \prec w$ if $v$ is an ancestor of $w$, and for an edge $e$ we also write $e \prec v$ if at least one endpoint of $e$ is an ancestor of $v$. For $S\subset v(T)$, the \emph{subtree of $T$ spanned by $S$} is the minimal subtree of $T$ containing all elements of $S$.

Letting $\N^0:=\{\emptyset\}$, the {\em Ulam--Harris tree} is the rooted tree with root $\emptyset$ and vertex set 
\[
\cU := \bigcup_{n\ge 0} \N^n
\] 
in which, for each $v \in \cU$, the set of children of $v$ is $\{vi,i \in \N\}$. (Here, and in the sequel, for a string $v=(v_1,\ldots,v_k)$ we write $vi:=(v_1,\ldots,v_k,i)$.) We say $w$ is a younger sibling of $u$ if $w=vj$, $u=vi$ and $j > i$. We will make use of the usual lexicographic order on $\cU$, which is the total order in which each vertex precedes all of its descendants and all of its younger siblings. Also, for $v \in \N^n \subset \cU$ we write $|v|=n$ for the depth of $v$ in $\cU$.
\begin{figure}
\centering
\hspace{-3.5cm}
\begin{minipage}
    [b]{0.1\textwidth}
    \centering
    \tikz[grow =up,  nodes={circle,draw}, minimum size = 
    1.2cm, inner sep = 0pt, scale = 0.78, transform shape]
  \node {	$\emptyset$}
    child {node {	3}
    child {node [right] {	 31}
    child {node {	 311}
    child {node {	 3112}}
    child {node {	 3111}
    child {node {	 31111}}}
    }}}
    child {node {	 2}}
    child {node {	 1}
    child {node {	 11}
    child {node {	 111}}}}
    ;
\end{minipage}
\hspace{2.8cm}
\begin{minipage}
    [b]{0.1\textwidth}
    \centering
    \tikz[grow = up, nodes = {circle, draw},minimum size = 
    1.1cm, inner sep = 0pt, scale = 0.78, transform shape]
    \node {	4}
    child {node {	3} 
    child {node [right]{	10}
    child {node {	1} 
    child {node {	11}}
    child {node {	9}
    child {node {	7}}}node [rectangle, inner sep= 0pt,draw = none, above left] {\hspace{-2.5cm}\color{black!30!orange} $(\sigma_1(9), \sigma_1(11))$}
    node [rectangle, inner sep= 0pt,draw = none,  left] {\hspace{-1.8cm}\color{black!30!orange} $ = (1,2)$}
    }}}
    child {node {	2}}
    child {node {	8}
    child {node {	5}
    child {node {	6}}}}
    node [rectangle, inner sep= 0pt,draw = none, above right] {\hspace{1cm}\color{black!30!orange}  $(\sigma_4(8), \sigma_4(2), \sigma_4(3))$} 
    node [rectangle, inner sep= 0pt,draw = none, right] {\hspace{1.2cm}\color{black!30!orange} $ = (1,2,3)$} 
    ;
\end{minipage}
\hspace{3.3cm}
\begin{minipage}
    [b]{0.1\textwidth}
    \centering
     \begin{tikzpicture}[grow = up, nodes = {circle, draw}, minimum size = 1.1cm, inner sep = 0pt, scale = 0.78, transform shape]
    \node (4) { 	4} 
    child {node[above right] (3) {	3} 
    child {node [above right](10) {	10}
    child {node[above] (1) { 	1}
    child {node [above right] (11) { 	11} edge from parent node [rectangle, inner sep = 1pt, draw=none,fill=white, minimum size = 0.1cm] {\color{black!50!green}   $(1,2)$}}
    child {node [above left] (9) { 	9} 
    child {node [above] (7) { 	7} edge from parent node [rectangle, inner sep = 1pt, draw=none,fill=white, minimum size = 0.1cm] {\color{black!50!green}   $(9,1)$}} edge from parent node [rectangle, inner sep = 1pt, draw=none,fill=white, minimum size = 0.1cm] {\color{black!50!green}   $(1,1)$}} 
    edge from parent node [rectangle, inner sep = 1pt, draw=none,fill=white, minimum size = 0.1cm] {\color{black!50!green}   $(10,1)$}} edge from parent node [rectangle, inner sep = 1pt, draw=none,fill=white, minimum size = 0.1cm] {\color{black!50!green}   $(3,1)$}}edge from parent node [rectangle, inner sep = 1pt, draw=none,fill=white, minimum size = 0.1cm] {\color{black!50!green}   $(4,3)$}}
    child {node[above] (2) {	2} edge from parent node [rectangle, inner sep = 1pt, draw=none,fill=white, minimum size = 0.1cm] {\color{black!50!green}   $(4,2)$}}
    child {node[above left] (8) {	8}
    child {node [above] (5) { 	5}
    child {node [above] (6) {	6} edge from parent node [rectangle, inner sep = 1pt, draw=none,fill=white, minimum size = 0.1cm] {\color{black!50!green}   $(5,1)$}} edge from parent node [rectangle, inner sep = 1pt, draw=none,fill=white, minimum size = 0.1cm] {\color{black!50!green}   $(8,1)$}}edge from parent node [rectangle, inner sep = 1pt, draw=none,fill=white, minimum size = 0.1cm] {\color{black!50!green}   $(4,1)$}}
    ;
    \end{tikzpicture}
\end{minipage}
\caption{Left: an ordered rooted tree. Center: a labeled ordered rooted tree, with the functions $\sigma_v$ indicated for $v \in \{1,4\}$. Right: the edge labeling of $T$, introduced in Section \ref{sec:encodings}.}
    \label{fig:lort}
\end{figure}
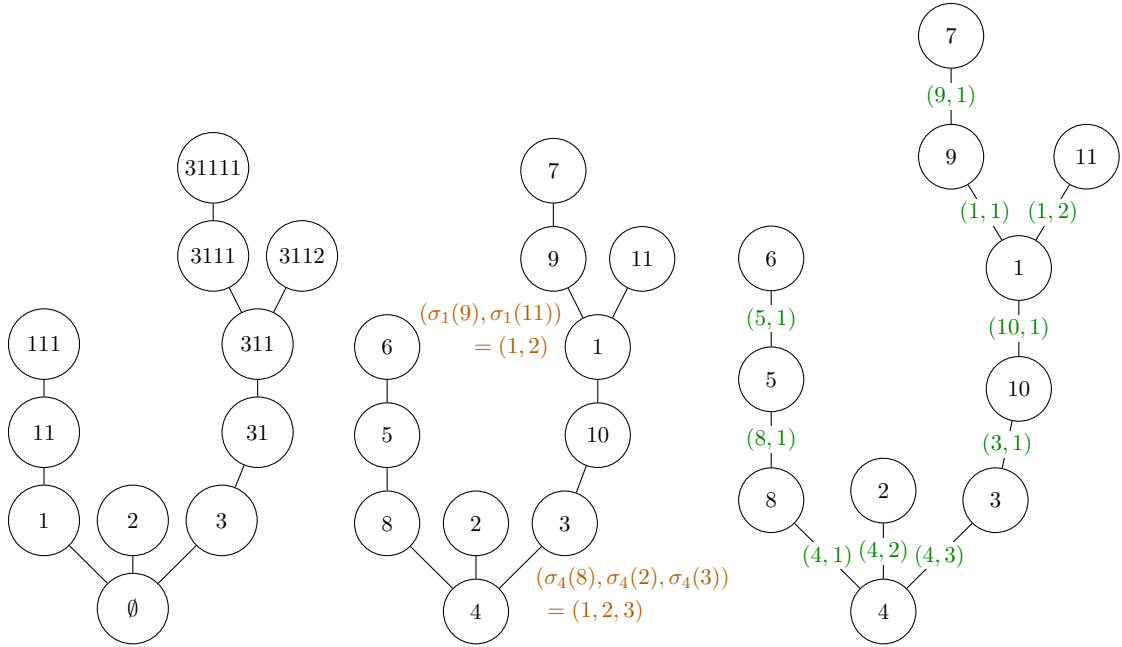

The definitions of the coming paragraph are illustrated in Figure~\ref{fig:lort}. An \emph{ordered rooted tree} is a tree $T$ with $v(T)\subset \cU$ and the following properties: (i) $\emptyset \in v(T)$; (ii) if $v \in v(T)$ then $p(v,\cU) \in v(T)$; (iii) if $vi \in v(T)$ then $vj \in v(T)$ for all $1 \le j \le i$. Note that the edge set of an ordered rooted tree may be recovered from its vertex set, and we will often identify ordered rooted trees with their vertex sets. The lexicographic order on $v(T)$ is simply the restriction of the lexicographic order on $\cU$ to~$v(T)$. 

A {\em labeled ordered rooted tree} is a finite rooted tree $T=(v(T),e(T))$ with $v(T)=[n]$ in which, for each non-leaf vertex of $T$, the set of children is endowed with a total order $\sigma_v=\sigma_{v,T}:C(v,T) \to [c(v,T)]$. We will sometimes abuse notation by writing $vi=\sigma_v^{-1}(i)$ for the $i$-th child of $v$ under this total order. This abuse of notation is justified by the observation that the ordering of the children of each non-leaf induces an injection $\varphi:v(T)\to \cU$ defined inductively by $\varphi(\rho(T))=\emptyset$ and $\varphi(vi)=\varphi(\sigma_v^{-1}(i))=\varphi(v)i$ for $i \in [c(v,T)]$; and $\varphi(v(T))$ is indeed (the vertex set of) an ordered rooted tree. As such, a labeled ordered rooted tree could equivalently be represented as a pair $(T,f)$ where $T\subset \cU$ is a finite ordered rooted tree and $f:T \to [n]$ is a bijection (so $n=|T|$). However, the first representation is more natural in the context of the methods we shall shortly use for constructing random labeled ordered rooted trees. Moreover, the second representation would be confusing, as it is very similar to our representations of {\em branching random walks} and of {\em spatial trees}, which we now describe.

\subsection{Branching random walks, {\L}ukasiewicz path, contour and height processes}\label{sec:bwcontheight}

A {\em branching random walk} is a pair $\rT=(T,Y)$, where $T$ is an ordered rooted tree (possibly labeled) and $Y=(Y^{(v)},v \in v(T)\setminus \partial T)$, where $Y^{(v)}=(Y^{(v)}_j,j \in [c(v,T)]) \in \R^{c(v,T)}$. We think of $Y^{(v)}$ as a set of spatial displacements from vertex $v$ to its children, so $Y^{(v)}_j$ is the difference in the spatial locations of vertices $v$ and $vj$.
The spatial {\em location} of $u \in v(T)$ is then given by the sum of displacements along $u$'s ancestral path:
\[
\ell(u)=\ell(u,\rT):= \sum_{\{(v,vj)\in e(T):vj \preceq u\}} Y^{(v)}_j. 
\]
We refer to the pair $(T,\ell)$ as a {\em spatial tree}. The branching random walk $(T,Y)$ can clearly be recovered from the spatial tree $(T,\ell)$, and vice versa.

Let $T=(v(T),e(T))$ be a finite ordered rooted tree and write $n=|T|$. The {\em \L ukasiewicz path} of $T$ is the function $W_T:[0,n] \to \R$ defined as follows. List the elements of $v(T)$ in lexicographic order as $v_1,\ldots,v_n$. Set $W_T(0) = 0$. For $1 \le i \le n$, set $W_T(i) = \sum_{j=1}^i (c(v_i, T) - 1)$, and then extend the domain of $W_T$ to $[0,n]$ by linear interpolation.
 
The {\em height process} of $T$ is the function $H_T:[0,n] \to \R_{\ge 0}$ defined as follows. For $0 \le i < n$ set $H_T(i)=|v_{i+1}|$ and set $H_T(n) = 0$; then extend the domain of $H_T$ to $[0,n]$ by linear interpolation.

The {\em contour order} of $v(T)$ is the sequence
$w_0,\ldots,w_{2(n-1)}$ of elements of $v(T)$ defined as follows. First, $w_0=\emptyset$ is the root of $T$. Inductively, for each $0 \le i < 2(n-1)$, if~$w_{i}$ has at least one child in $T$ which does not appear in the sequence $w_0,\ldots,w_{i-1}$, then let $w_{i+1}$ be the lexicographically least such child. Otherwise, let $w_{i+1}=p(w_i,T)$. It is straightforward to verify that each vertex $v$ of $T$ appears in the resulting sequence exactly $1+c(v,T)$ times. 
The {\em contour process} of $T$ is the function $\widetilde H_T:[0,2(n-1)] \to \R_{\ge 0}$ defined by setting $\widetilde H_T(i)=|w_i|$ for integers $i$ with $0 \le i \le 2(n-1)$, letting $\widetilde H_T(2n) = 0$, and extending to $[0,2n]$ by linear interpolation. 

If $\rT=(T,Y)$ is a branching random walk with underlying tree $T$ then we encode the spatial locations by a function $R_{\rT}:[0,n]\to \R$ given by setting $R_{\rT}(i)=\ell(v_{i+1},\rT)$,  for $i\in \{0,\dots, n-1\}$, $R_{\rT}(n) = 0$, and extending to $[0,n]$ by linear interpolation. We also define a process $\widetilde{R}_{\rT}:[0,2n] \to \R$ by setting $\widetilde{R}_{\rT}(i)=\ell(w_i,\rT)$ for integers $i$ with $0 \le i \le 2(n-1)$, $\widetilde{R}_{\rT}(2n) = 0,$ and extending to $[0,2n]$ by linear interpolation. 

The following result appears somewhat implicitly in Section 3 of \cite{prescribed_degrees}. For completeness we give a proof.

\begin{lem}
    \label{lem:diff_height_luka} Fix $\alpha_1, \alpha_2 \neq 0$. Let $\rT = (T, Y)$ be the branching random walk with $Y = (Y^{(v)}, v\in v(T) \setminus \partial T)$ such that $Y^{(v)} = (\alpha_1 - \frac{2}{\alpha_2}(c(v, T) - j), j\in [c(v,T)])$, $v\in v(T)$. Let $R_\rT$ be the function encoding the spatial locations of $\rT$. Then for all $t\in [0,n]$,
    \[R_\rT(t) = \alpha_1 H_T(t) - \frac{2}{\alpha_2} W_T(t).\]
\end{lem}
\begin{proof} It is sufficient to prove that 
    \[R_\rT(i) = \alpha_1 H_T(i) - \frac{2}{\alpha_2} W_T(i)\]
for $i\in \{0,1,\dots,n-1\}$. Let $i\in \{0,\dots,n-1\}$.  Then $H_T(i)$ is the number of ancestors of~$v_{i+1}$ in $T$. Further, $W_T(i)$ is the number of younger siblings of ancestors of $v_{i+1}$. It follows that  
\begin{equation*}
  \begin{aligned}
    \alpha_1 H_T(i) - \frac{2}{\alpha_2} W_T(i) &= \alpha_1 \cdot\left(\sum_{(u, uj)\in e(T): uj \preceq v_{i+1}} 1\right) -  \frac{2}{\alpha_2}\sum_{(u, uj)\in e(T): uj \preceq v_{i+1}}\left(c(u, T) - j\right)\\
    &=\sum_{(u, uj)\in e(T): uj \preceq v_{i+1}} \left(\alpha_1 -  \frac{2}{\alpha_2}\left(c(u, T) - j\right)\right)\\
    &= \sum_{(u, uj)\in e(T): uj \preceq v_{i+1}} Y_j^{(u)} =R_\rT(i).\qedhere
  \end{aligned}
  \end{equation*}
\end{proof}

\subsection{Sequential encodings of labeled ordered rooted trees}\label{sec:encodings}
Given a labeled ordered rooted tree $T = ([n], e(T))$, we assign labels to the \emph{edges} of $T$ as follows.
For $v \in [n]$ and $i \in [c(v,T)]$, assign label $(v,i)$ to the edge $\{v,vi\}=\{v,\sigma_v^{-1}(i)\}$. The set of all edge labels is then $L(T)=\{(v,i):v \in v(T),i \in [c(v,T)]\}$. 
Given any path $P=v_0v_1\ldots v_k$ from a vertex $v_0$ of $T$ to one of its descendants, let $\pi_P$ be the sequence of edge labels along the path from $v_0$ to $v_k$: formally, $\pi_P=\pi_P(T)=((v_0,c_0),\ldots,(v_{k-1},c_{k-1}))$, where $c_0,\ldots,c_{k-1}$ are such that $v_{j}=v_{j-1}c_{j-1}$ for each $j \in [k]$. 

We say a sequence $\bd=(d_1,\ldots,d_n)$ of non-negative integers is a {\em degree sequence} if $\sum_{v \in [n]} d_v=n-1$. We say a labeled tree $T$ with $v(T)=[n]$ has degree sequence $\bd$ if $c(v,T)=d_v$ for all $v \in [n]$. Write $\cL_\bd$ for the set of labeled ordered rooted trees with degree sequence $\bd$. For any tree $T \in \cL_{\bd}$, it is the case that $L(T)=\{(v,c): v \in [n],c \in [d_v]\}$. Write $\cP_{\bd}$ for the set of permutations of $\{(v,c): v \in [n],c \in [d_v]\}$; this set has size $(n-1)!$. For a fixed degree sequence $\bd$, we will make extensive use of a bijection $B:\cP_\bd \to \cL_\bd$ for $\bd=(d_1,\ldots,d_n)$ a degree sequence of length $n \ge 2$, which we give below. We first describe $B^{-1}$, as it is slightly simpler. 

\begin{tcolorbox}[title = The bijection $B^{-1}:\cL_\bd\to \cP_{\bd}$. {\bf Input}: $T \in \cL_\bd$.]
\begin{itemize}
    \item Let $T^{(0)}$ be the subtree of $T$ consisting of the root alone. 
    \item For $\ell \ge 1$, if $T^{(\ell-1)}\ne T$ then let $y^{(\ell)}$ be the smallest label of a vertex in $T$ which is not in $T^{(\ell-1)}$, let $P^{(\ell)}$ be the path in $T$ from $T^{(\ell-1)}$ to $y^{(\ell)}$, and let $T^{(\ell)}$ be the subtree of $T$ spanned by $\{P^{(\ell)}, y^{(1)},\dots, y^{(\ell-1)}\}.$ 
    \item Let $\ell^*$ be the first value for which $T^{(\ell^*)}=T$. 
    \item Let $\pi_T$ be the concatenation of the sequences $\pi_{P^{(1)}},\ldots,\pi_{P^{(\ell^*)}}$, and set $B^{-1}(T)=\pi_T$. 
\end{itemize}
\end{tcolorbox}

In the example of Figure~\ref{fig:lort}, $\ell^*=6$ and the paths are $P^{(1)}=4,3,10$, $P^{(2)}=4$, $P^{(3)}=4,8$, $P^{(4)}=5$, $P^{(5)}=1,9$ and $P^{(6)}=1$, so 
\begin{equation}\label{eq:pi}
\pi_T=((4,3),(3,1),(10,1),(4,2),(4,1),(8,1),(5,1),(1,1),(9,1),(1,2))\, .
\end{equation}

We next describe $B$; for this we make use of the fact that to specify a labeled ordered rooted tree $T$ with vertex set $[n]$ it suffices to specify the set $C(v,t)$ and the total orderings $\sigma_v:C(v,T) \to [c(v,T)]$ for each $v \in [n]$. 

Informally, this construction can be thought of as a discrete analog of the continuous line-breaking construction from the second paragraph of Section \ref{sec:proof_overview}. More specifically, given $\pi = ((v_1, c_1),\dots, (v_{n-1}, c_{n-1}))\in \cP_\bd$, certain substrings of $\pi$ will correspond to paths in the tree $B(\pi)$. We will list these paths as $P^{(1)},\dots, P^{(\ell^*)}$. As in the continuous line-breaking construction, for each $i \ge 2$ we will identify one endpoint of the path $P^{(i)}$,  with a vertex in $(P^{(j)})_{j < i}$. In the following formal description we denote the $i$-th identified vertex by $v_{j_i}.$ When we identify the endpoint of path $P^{(i)}$ with vertex $v_{j_i}$, we use the second coordinate of the pair $(v_{j_i},c_{j_i})$ to determine the position of the unique child of $v_{j_i}$ belonging to $P^{(i)}$ among the children of $v_{j_i}$.

\begin{tcolorbox}[title = {The bijection $B:\cP_{\bd} \to \cL_\bd$. {\bf Input}: $\pi=((v_1,c_1),\ldots,(v_{n-1},c_{n-1})) \in\cP_{\bd}$.}]
    \begin{itemize}
        \item Set $m_1=\min\{m \in \N: m \ne v_{1}\}$ and let
        \[
        j_1= \inf\{j > 1: v_j \in \{m_1,v_1,\ldots,v_{j-1}\}\} \wedge n\, .
        \]  
        \item For $i \ge 1$ , if $j_{i} < n$ then:
        \begin{itemize}
            \item set $m_{i+1}=\min\{m>m_{i}: m\not\in \{v_1,\dots,v_{j_{i}}\}\}$;
            \item let \[ j_{i+1}= \inf\{ j > j_i: v_j \in \{m_1,\ldots,m_{i+1},v_1,\dots, v_{j-1}\}\} \wedge n.\]
            \end{itemize}
        \item Let $\ell^*=\min\{i\ge 1:j_i=n\}$. 
        \item Define a labeled ordered rooted tree $T \in \cL_{\bd}$ as follows. For $1 \le i \le n-1$, if $i+1 \not\in \{j_1,\ldots,j_{\ell^*}\}$ then set $v_ic_i=\sigma_{v_i}^{-1}(c_i):=v_{i+1}$. If $i+1=j_k$ for some $1 \le k \le \ell^*$ then set $\sigma_{v_i}^{-1}(c_i)=m_k$.
        \item Set $B(\pi)=T$.\end{itemize}\end{tcolorbox}
The rightmost tree in Figure \ref{fig:lort} is the tree $B(\pi)$ where $\pi$ is equal to $\pi_T$ from (\ref{eq:pi}).

When needed, we will emphasise the dependence of the quantities $m_i$, $j_i$ and $\ell^{*}$ on $\pi$ by writing $m_i(\pi)$, $j_i(\pi)$ and $\ell^{*}(\pi)$. Setting $j_0=1$ for convenience, we may think of $T=B(\pi)$ as the union of the paths $P^{(1)},\ldots,P^{(\ell^*)}$, where $P_i=v_{j_{i-1}}\ldots v_{j_i-1}m_i$ is the path in $T$ from $v_{j_{i-1}}$ to $m_i$. Note that since $m_i \ge i$ for all $i \in [\ell^*]$, vertices $1,\dots,k$ are contained within the union of paths $P^{(1)},\dots,P^{(k \wedge \ell^*)}$ for all $k \in [n]$.

Recall from Section \ref{sec:main_result} that $\rT_n$ denotes a Bienaymé tree with offspring distribution~$\mu$ conditioned to have $n$ vertices. Suppose now that $D^{n} = (D^{n}_1,\dots,D^{n}_n)$ is a sequence of \iid $\mu$-distributed random variables conditioned to have total sum $\sum_{i=1}^n D^{n}_i = n-1$, and let $\Pi_{D^n} \in_{\cU}\cP_{D^n}$. Then the tree $\rT_n$ has the same law as $ B(\Pi_{D^n})$. Furthermore, $\bT_n = (\rT_n, Y)$, which we refer to as a $(\mu, \nu)$-branching random walk (conditioned to have size $n$), has the same law as $(B(\Pi_{D^n}), Y)$. (Here, conditionally on the underlying tree $T$, $Y = (Y^{(v)}, v\in v(T)\setminus \partial T)$ are independent random vectors such that if $c(v,T) = k$ then $Y^{(v)}$ has distribution $\nu_k$). The associated spatial tree $(B(\Pi_{D^n}), \ell)$ is such that $\ell(v_1) = 0$, and for $0 \le i \le n-1$ if $i+1 \not\in \{j_1,\ldots,j_{\ell^*}\}$,
\[ \ell(v_{i+1}) = \ell(v_i) + Y^{(v_i)}_{c_i},
   \]
and if $i+1=j_k \text{ for some } 1 \le k \le \ell^*,$ then 
\[\ell(m_k) =\ell(v_i) + Y^{(v_i)}_{c_i}.
\]

In Section \ref{sec:sample} we study the above bijective construction of uniform trees with a given deterministic degree sequence $\bd$; that is, for $T = B(\Pi_\bd)$ for $\Pi_\bd\in_\cU\cP_\bd$. We note however, that by conditioning on $D^n$, all results in Section \ref{sec:sample} also apply to $\rT_n$ and, consequently, to the underlying tree of $\bT_n = (\rT_n, Y)$.

\section{Sampling from $\cL_{\bd}$}\label{sec:sample}
Fix a degree sequence $\bd = (d_1,\dots,d_n)$. 
The bijection $B$ applied to a uniform element $\Pi_\bd\in_{\cU}\cP_\bd$ yields a uniform element $T = B(\Pi_\bd)$ of $\cL_{\bd}$. We can think of the bijection as constructing $T$ from $\Pi_\bd$ by adding vertices one at a time in order of their first appearance in a pair $(V,C)$ of $\Pi_\bd$. Below, we use this perspective to study properties of $T$, in particular the law of the sequence of vertices ordered by first appearance in a pair $(V,C)$ of $\Pi_\bd$, and the law of the number of vertices contained in the union of the paths $P^{(1)},.., P^{(k)}$, for given $k \ge 1$. 

\subsection{Size-biased random re-ordering}\label{sec:size_biased}
For $n \ge 1$ let $\cS_n$ denote the set of permutations of $[n]$. For $(k_1,\ldots,k_n) \in \N^n$, let $\Sigma = \Sigma_{(k_1,\dots,k_n)}$ be the random permutation of $[n]$ with law given by 
\[\p{\Sigma = \sigma} = \prod_{i=1}^n \frac{k_{\sigma(i)}}{\sum_{j=i}^n k_{\sigma(j)}},\quad \mbox{ for }\sigma\in \cS_n.\]
We call $(k_{\Sigma(1)},\dots, k_{\Sigma(n)})$ the \emph{size-biased random re-ordering} of $(k_1,\dots, k_n).$

For a degree sequence $\bd$, let $N_\bd = |\{i\in [n] : d_i > 0\}|$. For $\pi = ((v_1, c_1),\dots, (v_n,c_n))\in \cP_\bd$ we let $\hat{v}_1(\pi),\dots, \hat{v}_{N_\bd}(\pi)$ denote the internal vertices in $T = B(\pi)$ ordered by their first appearance in a pair $(v,c)$ in $\pi$. When $\pi = \Pi_\bd = ((V_1, C_1),\dots, (V_{n-1}, C_{n-1}))\in_\cU\cP_\bd$ is random, we write $\widehat{V}_i(\Pi_\bd) = \hat{v}_i(\Pi_\bd)$ to reinforce the fact that the order of the vertices is random. The next lemma states that $(\widehat{V}_1(\Pi_\bd),\dots,\widehat{V}_{N_\bd}(\Pi_\bd))$ are the vertices corresponding to a size-biased random reordering of $\{d_i ~:~ d_i > 0, i\in [n]\}$.  

\begin{lem}\label{lem:sb_order} Fix a degree sequence $\bd = (d_1,\dots,d_n)$ and let $\Pi_\bd\in_{\cU}\cP_\bd$. Then for any permutation $(i_1,\dots, i_{N_\bd})$ of $\{i \in [n]: d_i > 0\}$,
    \begin{align*}
     \p{(\widehat{V}_1(\Pi_\bd),\dots,\widehat{V}_{N_\bd}(\Pi_\bd)) = (i_1,\dots, i_{N_\bd})} &= \frac{d_{i_1}}{n-1}\frac{d_{i_2}}{n-1 - d_{i_1}}\dots \frac{d_{i_{N_\bd}}}{n-1-\sum_{j=1}^{N_\bd-1}d_{i_j}}.
    \end{align*}
    Consequently, the size-biased random reordering of the positive entries of $\bd$ is equal in distribution to $(d_{\widehat V_1(\Pi_\bd)},\dots, d_{\widehat V_{N_\bd}(\Pi_\bd)}).$
\end{lem}

\begin{proof}
We show the statement by induction on $N_\bd$. For $N_\bd=1$, the statement is immediate for all $n$ and for all degree sequences of length $n$ with $|\{i: d_i > 0\}|=1$ since if $N_\bd = 1$ there is a single vertex of positive degree and $\widehat{V}_1(\Pi_\bd) = i_1$. 

Next, fix $\ell \in \N$ and suppose the statement holds for all degree sequences $\bd$ with $N_\bd \le \ell$. Then fix any degree sequence $\bd = (d_1,\dots,d_n)$ with $N_\bd = \ell + 1$, and any permutation $(i_1,\dots,i_{N_\bd})$ of $\{i \in [n] : d_i > 0\}.$  To specify an element of $\{\pi\in \cP_\bd:(\hat{v}_1(\pi),\dots, \hat{v}_{\ell+1}(\pi))=(i_1,\dots,i_{\ell+1})\} $, it is necessary and sufficient to specify

\begin{enumerate}
 \item $\pi_1 =  (v_1, c_1)\in \{(i_1,c):c\in [d_{i_1}]\}$; 
    \item The $d_{i_1}-1$ values $j\in \{2,3,\dots,n-1\}$ for which $\pi_j = (i_1, c)$ for some $1\le c \le d_{i_1}$;
    \item The order of the $d_{i_1}-1$ elements of $\{(i_1,c),  1\le c \le d_{i_j} \}\backslash\{\pi_1\}$ in $\pi$;
    \item The order of the elements of $\{(i_j,c), 2\le j\le {\ell+1}, 1\le c \le d_{i_j} \}$ in $\pi$, which must ensure that $(\hat{v}_2(\pi),\dots, \hat{v}_{\ell+1}(\pi))=(i_2, \dots, i_{\ell+1})$. 
\end{enumerate}

By the induction hypothesis applied to the degree sequence $(d_{i_2},\dots, d_{i_{\ell+1}},0,\dots,0)\in$ $\Z_{\ge 0}^{n-d_{i_1}}$ this implies that 
\begin{align} &\left|\{\pi\in \cP_\bd:(\hat{v}_1(\pi),\dots, \hat{v}_{\ell+1}(\pi))=(i_1,\dots,i_{\ell+1})\}\right|\notag\\&\quad=d_{i_1} \binom{n-2}{d_{i_1}-1}(d_{i_1}-1)! (n-1-d_{i_1})!\frac{d_{i_2}}{n-1 - d_{i_1}}\dots \frac{d_{i_{\ell+1}}}{n-1-\sum_{j=1}^{\ell}d_{i_j}}\notag\\
&\quad= (n-1)!\frac{d_{i_1}}{n-1}\frac{d_{i_2}}{n-1 - d_{i_1}}\dots \frac{d_{i_{\ell+1}}}{n-1-\sum_{j=1}^{\ell}d_{i_j}}\, ;\label{eq:pdcount}
\end{align}
since $|\cP_{\bd}|=(n-1)!$, the claim follows. 
\end{proof}

\subsection{Repeats in $\Pi_\bd$} 
Let $\Pi_{\bd} = ((V_1,C_1),\dots, (V_{n-1}, C_{n-1}))\in_\cU \cP_{\bd}$. Recall that $\widehat V_1(\Pi_{\bd}),$ $ \dots,$ $ \widehat V_{N_\bd}(\Pi_\bd)$ are the internal vertices in $B(\Pi_{\bd})$ ordered by their first appearance in a pair $(V,C)$ in $\Pi_\bd$.

For $i \in [\ell^*(\Pi_\bd)]$ let $M^\bd_i = m^\bd_i(\Pi_\bd)$ and $J^\bd_i = j^\bd_i(\Pi_\bd)$. We introduce this notation to emphasise that $M^\bd_1,\dots, M^\bd_{\ell^*(\Pi_\bd)}$ and $J^\bd_1,\dots, J^\bd_{\ell^*(\Pi_\bd)}$ are random variables. We will see later that for the random degree sequences $D^n = (D_1^n,\dots, D_n^n)$ arising in this paper, for $k \ge 1$ fixed and for $n$ large, $\{V_1,\dots, V_{J_k^{D^n}}\}\cap [k] = \emptyset$ with high probability. In this case, for each $i\in [k]$ the first coordinate of the pair $(V_{J_i^{D^n}}, C_{J_i^{D^n}})\in \Pi_{D^n}$, corresponds to a \emph{repeated} first coordinate of $\Pi_{D^n}$. It is therefore convenient to define a second set of indices which correspond to the indices of $\Pi_\bd$ for which the first coordinate is a repeat. Specifically, let $\widetilde{J}^\bd_1 =  \inf \{ j > 1: V_j \in \{V_1,\dots,V_{j-1}\}\}\wedge n$, and for $i \ge 1$, let 
 \[\widetilde{J}^\bd_{i+1} =  \inf \{ j > \widetilde{J}^\bd_{i}: V_j \in \{V_1,\dots,V_{j-1}\}\}\wedge n.\] 

The next two lemmas describe the laws of $\widetilde{J}^\bd_1$ and $(\widetilde{J}^\bd_i, i \ge 2)$, respectively. 

\begin{lem} \label{lem:cuts}Fix an integer $n \ge 2$ and a degree sequence $\bd = (d_1,\dots,d_n)$ and let $\Pi_\bd \in_\cU\cP_{\bd}$. Then for $1 \le k \le N_\bd$, 
\[
\p{\widetilde{J}^\bd_1 > k~\Big|~ \widehat{V}_1(\Pi_\bd),\dots, \widehat{V}_{N_d}(\Pi_\bd)} = \prod_{j=1}^{k} \left(1 - \frac{\sum_{i=1}^j (d_{\widehat{V}_i(\Pi_\bd)}+1)}{n - 1 -j}\right), 
\]
and, for $k>N_\bd$,
\[
\p{\widetilde{J}^\bd_1 > k~\Big|~ \widehat{V}_1(\Pi_\bd),\dots, \widehat{V}_{N_d}(\Pi_\bd)} = 0. 
\]
\end{lem}

\begin{proof}
Observe that $\widetilde{J}_1^\bd \le N_\bd+1$ deterministically, so the statement for $k>N_\bd$ is immediate. To prove the statement for $1 \le k \le N_{\bd}$, fix any ordering $i_1,\ldots,i_{N_{\bd}}$ of $\{i \in [n]:d_i>0\}$. Then using Bayes' formula and the fact that $|\cP_{\bd}|=(n-1)!$, the probability 
\[
\p{\widetilde{J}^\bd_1 > k~\Big|~ (\widehat{V}_1(\Pi_\bd),\dots, \widehat{V}_{N_d}(\Pi_\bd))=(i_1,\dots, i_{N_\bd})}
\]
may be expressed as a ratio with denominator 
\[
|\{\pi=((v_i,c_i),i \in [n]) \in \cP_{\bd}: (\hat{v}_1(\pi),\ldots,\hat{v}_{N_{\bd}}(\pi))=(i_1,\ldots,i_{N_d})\}|
\]
and numerator
\[
|\{\pi\!=\!((v_i,c_i),i \in [n]) \in \cP_{\bd}\hspace{-1pt}:\hspace{-1pt} (\hat{v}_1(\pi),...,\hat{v}_{N_{\bd}}(\pi))\hspace{-1pt}=\hspace{-1pt}(i_1,...,i_{N_d}),(v_1,...,v_k)\hspace{-1pt}=\hspace{-1pt}(i_1,...,i_k)\}|.
\]
Equation \eqref{eq:pdcount} directly yields a formula for the denominator. 
Also, letting
$\bd'$ be the degree sequence $(d_{i_k+1},\ldots,d_{N_d},0,\ldots,0) \in \Z_{\ge 0}^{n-d_{i_1}-\ldots-d_{i_k}}$, then the numerator is 
\[
\prod_{j=1}^k d_{i_j} \cdot
(n-1-k)_{d_{i_1}+\ldots+d_{i_k}-k}
\cdot|
\{\pi'\in \cP_{\bd'}: (\hat{v}_1(\pi'),\ldots,\hat{v}_{N_{\bd}-k}(\pi'))=(i_{k+1},\ldots,i_{N_d})\}
|\, .
\]
The first term selects $c_j \in [d_{i_j}]$ for each  $j \in [k]$; 
the second, falling factorial term selects the locations of the remaining entries of $\pi$ whose first coordinate belongs to $\{i_1,\ldots,i_k\}$; and the third term specifies the order of the remaining entries of $\pi \in \cP_\bd$. Equation \eqref{eq:pdcount} also gives a formula for this final term, and the lemma then follows by routine algebra.
\end{proof}
    \begin{lem}
        \label{lem:later_cuts}
        Fix a degree sequence $\bd = (d_1,\dots,d_n)$ and let $\Pi_\bd\in_\cU\cP_\bd$. Let $i 
    \ge 1$. Then for $n \ge 2$ and $k$ such that $\widetilde{J}^\bd_i+k\in [N_\bd]$,
        \begin{align*}&\p{\widetilde{J}^\bd_{i+1} > \widetilde{J}^\bd_i + k~\Big|~ \widetilde{J}^\bd_1,\dots, \widetilde{J}^\bd_i, \widehat{V}_1(\Pi_\bd),\dots, \widehat{V}_{N_\bd}(\Pi_\bd)}\\
        &\quad= \prod_{j=\widetilde{J}_i^\bd}^{\widetilde{J}_i^\bd + k}\left(1 - \frac{\sum_{\ell = 1}^j (d_{\widehat{V}_\ell(\Pi_\bd)} - 1) - i }{n - j}\right)\end{align*}
        and, for $k>N_\bd$,
    \[\p{\widetilde{J}^\bd_{i+1} > k~\Big|~ \widehat{V}_1(\Pi_\bd),\dots, \widehat{V}_{N_\bd}(\Pi_\bd)} = 0. \]
    \end{lem}
    The proof of Lemma \ref{lem:later_cuts} is analogous to that of Lemma \ref{lem:cuts} and is therefore omitted.
Finally, a bound we will need in Section \ref{sec:tightness}, whose proof relies on the bijective construction of $\rT_n$, is the following; its proof is postponed to the appendix.

\begin{lem}\label{lem:no_ganging_up_d_tree}
Let $\bd=(d_1,\dots,d_{n})$ be a degree sequence and let $\cB\subset[n]$ be a set of vertices. Suppose that $|\cB|\leq K$ and suppose that $\max_{1 \le i \le n} d_i\le \Delta$. Let $B_{\bd}$ be the smallest distance between two vertices in $\cB$ that are ancestrally related in $T_\bd = B(\Pi_\bd)$ (with $B_{\bd}=\infty$ if no vertices in $\cB$ are ancestrally related). Then, for any $b\ge 0$ 
\[
\p{B_{\bd}\leq b}\leq K\left(1-\left(1-\frac{K\Delta}{n-1-b\Delta}\right)^b\right).
\]
\end{lem}

\section{Random finite-dimensional distributions}\label{sec:fdds}

In this section we use the bijection $B$ to prove the convergence of the random finite-dimensional distributions of the head of the discrete snake $(H_n, R_n)$.  We assume throughout this section that $\mu$ is critical and has variance $\sigma^2\in (0,\infty)$, and that assumption [\ref{a1}] holds. 
 
Recall that $\rT_n$ is a Bienaymé tree with offspring distribution $\mu$ conditioned to have $n$ vertices, and that $\bT_n = (\rT_n,Y)$ denotes the conditioned $(\mu,\nu)$-branching random walk. By Section \ref{sec:encodings}, $\bT_n$ has the same distribution as $(B(\Pi_{D^n}),Y)$, where $D^n = (D_1^n,\dots, D_n^n)$ is a sequence of \iid $\mu$-distributed random variables conditioned to have total sum $\sum_{i=1}^nD_i^n = n-1$ and, conditionally on $D^n$, $\Pi_{D^n} = ((V_1,C_1),\dots, (V_{n-1}, C_{n-1}))\in_{\cU}\cP_{D^n}$.
 
Fix $k \ge 1$. Let $U^n_1,\dots,U^n_k$ be a uniformly random $k$-set of indices chosen from $[n]$. Let $\rT_n(U^n_1,\dots,U^n_k)$ be the subtree of $\rT_n$ spanned by the root of $\rT_n$ and the vertices $v_{U^n_1},\dots,$ $ v_{U^n_k}$, where for $i\in [n]$, we recall that $v_i$ is the $i$-th vertex in the lexicographical order of $\rT_n$. (For fixed $k$, as $n \to \infty$, a collection of $k$ \iid Uniform$([n])$ random variables will be distinct with probability tending to 1, so we can treat $U_1^n, \ldots, U_k^n$ as indistinguishable from independent uniform picks from the vertices.) We immediately observe that $\rT_n(U_1^n, \ldots,U_k^n)$ has the same distribution as $\rT_n^k$, the subtree of $B(\Pi_{D^n})$ spanned by the root and the vertices $1,\dots, k$. Since $\rT_n^k$ is more convenient for our analysis, we will work with it instead. Note that $\rT_n^k$ is a labeled ordered rooted tree whose leaves are labeled by $1,2,\ldots,k$. Write $\ell^k_n$ for the map from $\rT_n^k$ into $\R$ which gives the spatial locations of the vertices, so that $(\rT_n^k, \ell^k_n)$ is the spatial tree $(\rT_n,\ell)$ restricted to the subtree spanned by the root and the vertices $1,\ldots,k$.
 
Let $\cT_{2\mathbf{e}}$ denote the Brownian tree encoded by the excursion $2\mathbf{e}$, and let $U_1, \ldots, U_k$ be \iid Uniform($[0,1]$) random variables, independent of $\mathbf{e}$. Recall that $\cT_{2\mathbf{e}}^{k} = \cT_{2\mathbf{e}}(U_1,\dots,U_k)$ denotes the subtree of $\cT_{2\mathbf{e}}$ spanned by the images of $0$ and $U_1,\dots, U_k$ in $\cT_{2\mathbf{e}}$, thought of as an ordered rooted tree with leaves labeled by $1,2,\ldots,k$ and with real-valued edge lengths. Recall that $\cT_{2\mathbf{e}}^k$ has the same distribution as the tree $\cT^k$ built by Aldous' line-breaking construction. We now introduce a version of the line-breaking construction which incorporates spatial locations.

\begin{tcolorbox}[title = Line-breaking construction of the Brownian tree with spatial locations] We construct a sequence $(\cT^k)_{k \ge 1}$ of trees along with two functions $\mathbf{h}: [0,\infty) \to [0, \infty)$ and $\mathbf{l}: [0,\infty) \to \R$ recursively. Let $J_1, J_2, \dots$ be the jump times of a Poisson point process on $[0,\infty)$ with intensity $tdt$ at time $t$, listed in increasing order. Independently, let $(B_t)_{t \ge 0}$ be a standard Brownian motion. Start from the tree $\cT^1$ which consists of the line-segment $[0,J_1]$. Define $\mathbf{h}(t) = t$ and $\mathbf{l}(t) = B_t$ for $0 \le t \le J_1$. Recursively, for $k \ge 2$, conditionally on $J_{k-1}$, sample an attachment point $A_{k-1}\sim \mathrm{Uniform}([0, J_{k-1}])$, independent of $(A_j)_{j < k-1}$. Take the completion of the line segment $(J_{k-1}, J_{k}]$, and let $J_{k-1}^*$ denote the limit point as $x\downarrow J_{k-1}$. Identify the points $J_{k-1}^*$ and $A_{k-1}$. This has the effect of gluing the line-segment $(J_{k-1},J_{k}]$ onto $\cT^{k-1}$. We do this with probability $1/2$ to the left side and with probability $1/2$ to the right side. This yields $\cT^{k}$. Define $\mathbf{h}(t) = \mathbf{h}(A_{k-1}) + t-J_{k-1}$ and $\mathbf{l}(t) = \mathbf{l}(A_{k-1}) + B_t -  B_{J_{k-1}}$ for $t \in (J_{k-1},J_{k}]$ to determine the height and location processes on the new line-segment. 
\end{tcolorbox}

The planar embedding of $\cT^k$ is captured by a permutation $\tau^k: [k] \to [k]$ which is such that $\tau^k(1), \ldots, \tau^k(k)$ is the order in which we observe the leaves when exploring the tree from left to right. 
Using the notation $U^k_{(1)},\ldots,U^k_{(k)}$ for the increasing ordering of $U_1,\ldots,U_k$ 
as in Proposition~\ref{prop:rfdds+rtightness}, we then have  
\begin{align}
& \left(\mathbf{h}(J_{\tau^k(1)}), \ldots, \mathbf{h}(J_{\tau^k(k)}), \mathbf{l}(J_{\tau^k(1)}), \ldots, \mathbf{l}(J_{\tau^k(k)})\right) \notag \\
& \qquad \eqdist \left(2\mathbf{e}_{U_{(1)}^k}, \ldots, 2 \mathbf{e}_{U_{(k)}^k}, \sqrt{2}\mathbf{r}_{U_{(1)}^k}, \ldots, \sqrt{2} \mathbf{r}_{U_{(k)}^k}\right), \label{eq:BSrfdds}
\end{align}
where the equality in distribution of the first $k$ co-ordinates on the two sides is a consequence of 
Corollary 22 of Aldous~\cite{continuum_3}, and that of the final $k$ co-ordinates is a consequence of the definition of the Brownian snake given at (\ref{eq:BSBE}). So the line-breaking construction indeed realises the random finite-dimensional distributions of the head of the Brownian snake.

We show that the scaling limit of $(\rT_n^k,\ell_n^k)$ is $(\cT^k,\mathbf{l}|_{[0,J_k]})$ in an appropriate sense, which will allow us to prove the convergence of the random finite-dimensional distributions, along with a certain amount of extra information which will be useful to us in Section \ref{sec:tightness} where we prove tightness.

Recall that the tree $\rT^k_n$ necessarily sits within the first $k$ paths, $P^{(1)}, \ldots, P^{(k)}$, in the discrete line-breaking construction. We need to understand the lengths of these paths, and the positions at which the paths are glued onto one another. It is convenient to use the \emph{indices} of the vertices in $\Pi_{D^n}$ for this purpose rather than the vertex labels themselves. Recall that $J_1^{D^n}, J_2^{D^n}, \ldots, J_k^{D^n}$ are the first $k$ indices at which we see either a repeat or an element of $\{1,2,\ldots,k\}$. Let us henceforth write $J_i^n = J_i^{D^n}$ (and also $\widetilde{J}_i^n = \widetilde{J}_i^{D^n}$) for $i \ge 1$. Then the lengths of the paths $P^{(1)}, \ldots, P^{(k)}$ are given by $J_1^n, J_2^n - J_1^n, \ldots, J_k^n - J_{k-1}^n$.  For $1 \le m \le k-1$, the index at which the path $P^{(m+1)}$ attaches onto the subtree constructed from the first $m$ paths is given by the value $i$ such that $V_{J_{m}^n} = \widehat V_i(\Pi_{D^n})$ (i.e.\ we find the index of the vertex $V_{J_{m}^n}$ within the vector $(\widehat{V}_1(\Pi_{D^n}),\ldots,\widehat{V}_{N_n}(\Pi_{D^n}))$). We write $A_{m}^n$ for this value $i$ and call this the $m$-th \emph{attachment point}. See Figure \ref{fig:christina_fig}. 

\begin{figure}[h!]
\begin{center}
\includegraphics[width=0.95\textwidth]{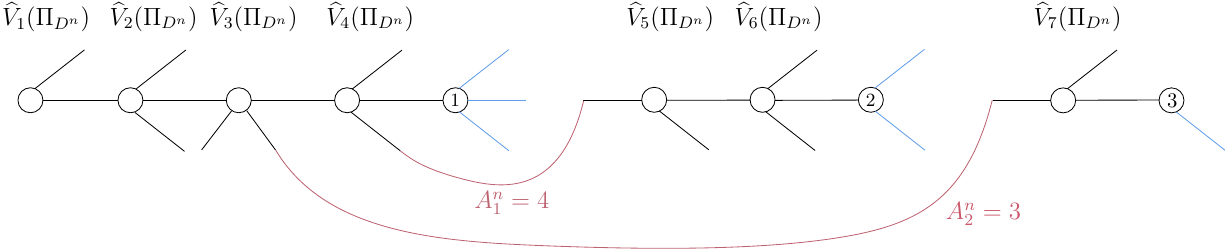}
\end{center}
\caption{Illustration of the first and second attachment points, $A_1^n$ and $A_2^n$.}\label{fig:christina_fig}
\end{figure} 

Since $\rT_n^k$ is an ordered tree, we will need to understand where the paths $P^{(2)}, \ldots, P^{(k)}$ attach relative to the pre-existing children of their attachment points. If we are looking to attach to a vertex which has only one pre-existing child (i.e.\ for which there has been no previous repeat) then that vertex must have degree $d \ge 2$, and then whether we attach to the left or to the right of the pre-existing child is simply determined by the relative ordering of the corresponding second coordinates in the sequence $\Pi_{D^n}$. If there has been no previous repeat at this vertex then this pair of second coordinates is chosen uniformly at random without replacement from $[d]$ and, in particular, we attach to the left and right sides each with probability $1/2$. This ceases to be true after the first repeat (not least because then there are three or more children whose relative ordering we need to understand), but as we shall show below, we observe a second repeat of any vertex in $\rT_n^k$ with vanishing probability as $n \to \infty$. Let $F^n_1, \ldots, F^n_k$ be random variables taking values in $\{0,1,2\}$ such that $F^n_i = 1$ if $P^{(i+1)}$ attaches at a first repeat and to the left-hand side, $F^n_i = 2$ if $P^{(i+1)}$ attaches at a first repeat and to the right-hand side and $F^n_{i+1} = 0$ otherwise, for $1 \le i \le k$.

Finally, recall that vertex $\widehat{V}_i(\Pi_{D^n})$ has degree $D^n_{\widehat{V}_i(\Pi_{D^n})}$ for $i \le N_{D^n} = |\{i \in [n]: D_i^n > 0\}|$. Let $L^n(0) = 0$ and let $L^n(i)$ be the spatial location of the $C_i$-th child of vertex $V_i$ in line-breaking construction $B(\Pi_{D^n})$, for $1 \le i \le n-1$. 

The following proposition shows that, on rescaling, these quantities converge in distribution to their analogues in the line-breaking construction of the Brownian tree with spatial locations.

\begin{prop} \label{prop:fdds}
Fix $k \ge 1$. Then  
\begin{equation}\label{eq:conv_attach-label}\frac{\sigma}{\sqrt{n}}(J_1^n, J_2^n,\dots, J_k^n, A_1^n,\dots, A_k^n) \convdist (J_1, J_2,\dots, J_k, A_1,\dots, A_k)
\end{equation}
as $n\rightarrow \infty$. Jointly with this convergence, we have that
\begin{equation} \label{eq:flagsconv}
(F_1^n, F_2^n, \ldots, F_k^n) \convdist (F_1, F_2, \ldots, F_k),
\end{equation}
where $F_1, F_2, \ldots, F_k$ are \iid random variables, independent of everything else, such that $\p{F_i = 1} = \p{F_i = 2} = 1/2$ and 
\begin{equation} \label{eq:RWtoBM}
\begin{split}
\left(\frac{L^n(\fl{tn^{1/2}}\wedge (J_1^n-1))}{n^{1/4}}\right)_{t \ge 0} & \hspace{-0.5cm}\convdist \beta (B_{t \wedge (J_1/\sigma)})_{t \ge 0}, \\
\left(\hspace{-3pt}\frac{L^n((J_i^n \hspace{-1pt}+\hspace{-1pt} \fl{t n^{1/2}})\hspace{-1pt} \wedge\hspace{-1pt} (J_{i+1}^n\hspace{-2pt}-\hspace{-2pt}1))}{n^{1/4}}\hspace{-3pt}\right)_{t \ge 0} & \hspace{-0.5cm}\convdist  \beta (B_{A_i/\sigma}\hspace{-1pt}+\hspace{-1pt} B_{((J_i/\sigma) \hspace{-1pt}+\hspace{-1pt} t) \wedge (J_{i+1}/\sigma)} \hspace{-1pt}-\hspace{-1pt}B_{(J_i/\sigma)})_{t \ge 0}
\end{split}
\end{equation}
for $1 \le i \le k-1$, in each case for the uniform norm.
\end{prop} 

As a corollary, we obtain the convergence of the random finite-dimensional distributions in (\ref{eq:suffice}).

\begin{cor}\label{cor:4.2}%confident label! -lab
For any $k \ge 1$, as $n\rightarrow \infty$
\begin{align*}
& \left(\frac{H_n(nU_{(1)}^k)}{\sqrt{n}}, \ldots, \frac{H_n(nU_{(k)}^k)}{\sqrt{n}}, \frac{R_n(nU_{(1)}^k)}{n^{1/4}}, \ldots, \frac{R_n(nU_{(k)}^k)}{n^{1/4}}\right) \\
& \convdist \left(\frac{2}{\sigma}\mathbf{e}_{U_{(1)}^k}, \ldots, \frac{2}{\sigma}\mathbf{e}_{U_{(k)}^k}, \beta \sqrt{\frac{2}{\sigma}} \mathbf{r}_{U_{(1)}^k},\ldots,\beta \sqrt{\frac{2}{\sigma}}\mathbf{r}_{U_{(k)}^k}\right),
\end{align*}
where $U_{(1)}^k, \ldots, U_{(k)}^k$ are the order statistics of $k$ \iid Uniform$([0,1])$ random variables.
\end{cor}

\begin{proof}
Let $(U_1^n,\dots, U_k^n)$ be a uniformly random $k$-set chosen from $[n]$, and $(U^{n,k}_{(1)},\dots,$ $ U_{(k)}^{n,k})$ be the order statistics of $(U_1^n,\dots, U_k^n)$. As argued above, we may straightforwardly replace $nU_{(1)}^k, \ldots, nU_{(k)}^k$ by $(U_{(1)}^{n,k}-1, \ldots, U_{(k)}^{n,k}-1)$ at no asymptotic cost. Recall that $H_n(i)$ gives the distance from the root of the $(i+1)$-th vertex visited in a depth-first exploration of the tree. The random variables
\[
(H_n(U_{1}^{n}-1), \ldots, H_n(U_{k}^{n}-1))
\]
have the joint law of the distances from the root to the leaves labeled $1, 2, \ldots, k$ in $\rT^k_n$ (these may be expressed in terms of sums and differences of elements of $(J_1^n, \ldots, J_k^n, A_1^n, \ldots, A_k^n)$ analogously to the definition of $\mathbf{h}$ in the line-breaking construction of the Brownian tree with spatial locations), and
\[
(R_n(U_1^n-1), \ldots, R_n(U_k^n-1)) = (L^n(J_1^n-1), \ldots, L^n(J_k^n-1)).
\]
The effect of ordering the uniforms is simply to apply the same permutation of the entries to each of $(H_n(U_1^n-1),\ldots,H_n(U_k^n-1))$ and $(R_n(U_1^n-1), \ldots, R_n(U_k^n-1))$. This permutation is straightforwardly induced by the choices $(F_1^n, \ldots, F_{k-1}^n)$. By (\ref{eq:flagsconv}), this permutation then converges in distribution to $\tau^k$. But then the claimed convergence follows from Proposition~\ref{prop:fdds} using the scaling property of Brownian motion and (\ref{eq:BSrfdds}).
\end{proof} 

 We begin by studying the vertex degrees at the start of the bijective construction, and show that, on the timescale of $\sqrt{n}$, the degrees that we observe are asymptotically indistinguishable from \iid copies of $\bar\xi$. We show further that the subtree $\rT_n^k$ is constructed on a timescale of order $\sqrt{n}$. This allows us to prove (\ref{eq:conv_attach-label}) in Proposition~\ref{prop:conv_cuts_attach}. To get the convergence of the spatial locations, we observe that, with the exception of branch points, the displacements along the ancestral lineages in $\rT_n^k$ are asymptotically indistinguishable from \iid copies of $Y_{\bar\xi, U_{\bar\xi}}$. Combining this with the convergence of the tree allows us to obtain the convergence of the spatial locations along the branches of the subtree.

\subsection{A discrete change of measure}
In this subsection, we show that the size-biased random re-ordering of the positive entries of $D^n$ may be viewed as a vector of \iid copies of the size-biased offspring random variable $\bar{\xi}$ up to a change of measure. We study the behaviour of the Radon--Nikodym derivative and show that its effect is trivial on the first $O(\sqrt{n})$ entries of the vector. Recall that $N_{D^n} = |\{i\in[n] : D_i^n > 0\}|$. To ease the notation, we write $N_n = N_{D^n}$. Let \[\widehat D^{n} = \left(\widehat D_1^n,\dots, \widehat D^n_{N_n}\right)\] be the size-biased random re-ordering of the positive entries of $D^{n}$. We note that 
\[\widehat D^n \eqdist \left(D^n_{\widehat{V}_1(\Pi_{D^n})},\dots, D^n_{\widehat{V}_{N_n}(\Pi_{D^n})}\right).\]
Later we will often somewhat abuse notation and write $(\widehat D_1^n,\dots, \widehat D_{N_n}^n)$ in place of \[(D_{\widehat V_1(\Pi_{D^n})}^n,\dots, D_{\widehat V_{N_n}(\Pi_{D^n})}),\] for example in the proof of Proposition \ref{prop:conv_cuts_attach}.

\begin{prop} \label{prop:measure_change_basic} 
Let $\xi_1,\xi_2,\dots, \xi_n$ be \iid random variables with distribution $\mu$. Further, let $\bar\xi_1, \bar\xi_2,\dots$ be \iid samples from the size-biased distribution of $\xi_1$. Then for $1 \le m < n$, and any non-negative measurable function $f:\Z^m \rightarrow \R_+$,
\[
\E{f(\widehat D_1^{n},\dots, \widehat D^{n}_m)\I{N_n \ge m}} = \E{f(\bar\xi_1,\dots,\bar\xi_m)\Theta^n(\bar\xi_1,\dots,\bar\xi_m)},
\]
where for $k_1,\dots,k_m \in \N$,
\begin{equation}\label{eq:basic_meas_change}
\Theta^n(k_1,\dots,k_m) = \frac{\p{\sum_{i = m+1}^n \xi_i = n - 1 - \sum_{i=1}^m k_i}}{\p{\sum_{i=1}^n\xi_i = n-1}}\prod_{i=1}^m\left(\frac{n-i+1}{n-1-\sum_{j=1}^{i-1}k_j}\right),
\end{equation}
if $k_1+\ldots+k_m\le n-1$, and $\Theta^n(k_1,\ldots,k_m)=0$ otherwise. 
\end{prop}

Proposition \ref{prop:measure_change_basic} is a special case of Proposition \ref{prop:sizebias_swole} that we state and prove in the appendix, and use in full generality to prove Theorems \ref{thm:hairy_4} and \ref{thm:hairy_2} in Section \ref{sec:hairy_tour}. We state only the special case here as the more general formulation is much more technical and requires definitions that are only relevant in settings where assumption [\ref{a3}] holds.

The next lemma shows that the change of measure $\Theta^n$ appearing in Proposition \ref{prop:measure_change_basic} is asymptotically unimportant provided that $m = \Theta(\sqrt{n})$.

\begin{lem}\label{lem:conv_measure_change_basic} Let $\mu$ be a critical offspring distribution with variance $\sigma^2\in (0,\infty)$, and let $(\bar\xi_i)_{i \ge 1}$ be \iid samples from the size-biased distribution of $\mu$. Suppose that $m = m(n) =
\Theta(\sqrt{n})$. Then as $n\rightarrow\infty$
\[
\Theta^n(\bar\xi_1,\dots,\bar\xi_m) \convprob 1,
\]
and $(\Theta^n(\bar\xi_1,\dots,\bar\xi_m))_{n\ge 1}$ is a uniformly integrable sequence of random variables.
\end{lem}
 
\begin{proof} 
By a subsubsequence argument we may assume that $m/\sqrt{n}\rightarrow t$ as $n\rightarrow \infty$ for some $t > 0$. Let $\xi_1,\dots, \xi_n$ be \iid random variables with distribution $\mu$. We deal with the ratio of probabilities in the definition of $\Theta^{n}$ using the local central limit theorem. Specifically, since $\E{\xi_1} = 1$ and $\V{\xi_1} = \sigma^2$, we have that
\[\sup_{k\in \Z}\left|\sqrt{n - m}\cdot\p{\sum_{i=m+1}^n\xi_i = n - 1 - m + k} - \frac{1}{\sqrt{2\pi\sigma^2}}\exp\left(-\frac{k^2}{2\sigma^2(n-m)}\right)\right| \rightarrow 0\]
as $n\rightarrow \infty$, so for $k_1,\dots, k_m \in \N$,
\begin{align*}
&\p{\sum_{i=m+1}^n\xi_i = n - 1 - \sum_{i=1}^m k_i}\\
&\quad = \p{\sum_{i=m+1}^n\xi_i - (n - m) = -1 - m\sigma^2- \sum_{i=1}^m(k_i - 1-\sigma^2)}\\
& \quad = \frac{\exp\left(-\frac{1}{2\sigma^2(n-m)}\bigg(1+ m\sigma^2 + \sum_{i=1}^m(k_i - 1-\sigma^2)\right)^2\bigg)}{\sqrt{2\pi\sigma^2(n-m)}} + o(n^{-1/2}).
\end{align*}
Similarly, we have that 
\begin{align*}
\p{\sum_{i=1}^n\xi_i = n - 1}&= \frac{1}{\sqrt{2\pi\sigma^2  n }} + o(n^{-1/2}).
\end{align*}
Therefore,  
\begin{align}\label{eq:basic_meas_need}
&\frac{\p{\sum_{i=m+1}^n\xi_i = n - 1- \sum_{i=1}^m k_i}}{\p{\sum_{i=1}^n\xi_i = n - 1}}\nonumber\\
& \qquad =\exp\left(-\left(\frac{1 + m\sigma^2 + \sum_{i=1}^m(k_i - 1-\sigma^2)}{\sqrt{2\sigma^2(n- m)}}\right)^2\right)+ o(1). 
\end{align}
Since the random variables $\bar\xi_1,\dots, \bar\xi_n$ are \iid with mean $\sigma^2 + 1$, by the functional strong law of large numbers (as stated in Lemma \ref{lem:serflln}), as $n\rightarrow\infty$,
\begin{equation}\label{eq:slln}\frac{1}{\sqrt{n}}\max_{1\le i \le \lfloor t\sqrt{n}\rfloor}~\left|~ \sum_{j = 1}^i (\bar\xi_j -1 - \sigma^2)\right| \convas 0.\end{equation}
Since $m = (1+o(1))t\sqrt{n}$ this, in particular, yields that
\begin{equation}\label{eq:conv_gen}
\exp\left(-\left(\frac{1 + m\sigma^2 + \sum_{i=1}^m(\bar\xi_i - (1+\sigma^2))}{\sqrt{2\sigma^2(n - m)}}\right)^2 \right) \convprob \exp\left(-\frac{t^2\sigma^2}{2}\right).
\end{equation}
We claim that, as $n\rightarrow \infty$,
\begin{equation}
\label{eq:not_hairy_prod}\prod_{i=1}^m \left(\frac{n - i + 1}{n - 1 - \sum_{j=1}^{i-1}\bar\xi_j}\right) \convprob \exp\left(\frac{t^2\sigma^2}{2}\right).
\end{equation}
Indeed,
\begin{align*}
\prod_{i=1}^m \left(\frac{n - i + 1}{n - 1 - \sum_{j=1}^{i-1}\bar\xi_j}\right) &= \exp\left(-\sum_{i=1}^m\log\left(1 - \frac{\sum_{j=1}^{i-1}(\bar\xi_j - 1- \sigma^2) + \sigma^2(i-1)}{n  - i + 1}\right)\right).
\end{align*}
It follows by Taylor's theorem and (\ref{eq:slln}) that the last expression is equal to 
\begin{align}\label{eq:use_in_ms}
& \exp\left(\sum_{i = 1}^{m}\frac{\sum_{j = 1}^{i-1}(\bar\xi_j - 1 - \sigma^2)+ \sigma^2 (i-1)}{n-i+1}+ o_\bP(1)\right)\nonumber\\
& \qquad =\exp\left(\frac{\sigma^2 \lfloor t\sqrt{n}\rfloor(\lfloor t\sqrt{n}\rfloor -1)}{2n} + o_\bP(1)\right) \convprob \exp\left(\frac{t^2\sigma^2}{2}\right),
\end{align}
establishing (\ref{eq:not_hairy_prod}).
Combining this with (\ref{eq:basic_meas_need}) and (\ref{eq:conv_gen}) yields that 
\[
\Theta^{n}(\bar\xi_1,\dots, \bar\xi_m)\convprob 1.
\]

To prove uniform integrability, notice that, by applying Proposition \ref{prop:measure_change_basic} with $f \equiv 1$,
\[
\E{\Theta^{n}(\bar\xi_1, \dots,\bar\xi_m )}= \p{N_n \ge m}.
\]
We claim that this tends to $1$ as $n\rightarrow\infty$. To this end, note that \[\#\{i\in [n]~:~\xi_i > 0\} \eqdist \mathrm{Binomial}(n, 1-\mu_0).\] So by conditioning on the event $\{\sum_{i=1}^{n}\xi_i = n - 1 \}$, which occurs with probability $\Theta(n^{-1/2})$, there are $(1+o_\bP(1))n(1-\mu_0)$ non-zero entries of $(\xi_{1},\dots, \xi_n)$. Since $m = (1+o(1))t\sqrt{n}$ it follows that as $n\rightarrow \infty$,
\[ \p{N_n \ge m} = \p{\#\{i\in [n]~:~\xi_i > 0\} \ge m~\bigg|~\sum_{i=1}^n \xi_i = n - 1} \rightarrow 1.\]
Uniform integrability then follows by the generalised Scheffé lemma, see \cite[Theorem 5.12]{kallenberg}. \end{proof}

Lemma~\ref{lem:conv_measure_change_basic} implies that if the offspring distribution has finite variance then, on a timescale of $\sqrt{n}$ in the bijective construction $B(\Pi_{D^n})$ of $\rT_n$, the degrees we observe are asymptotically indistiguishable from \iid copies of $\bar\xi$. To prove Proposition \ref{prop:fdds}, we use this fact in the form of Proposition \ref{prop:conv_degrees_disps} stated below.

\begin{prop} \label{prop:conv_degrees_disps}
Given $\widehat D^n$, let $U_1,\dots, U_n$ be independent random variables such that, for each $i\in [n]$, $U_i$ is uniformly distributed on $[\widehat D_i^n]$. Further, let $Y_{\widehat D^n_1, U_{1}},\dots Y_{\widehat D^n_n, U_{n}}$ be independent random variables such that, for each $i\in [n]$, $Y_{\widehat D_i^n, U_i}$ is a uniform entry of a $\nu_{\widehat D^n_i}$-distributed displacement vector. If [\ref{a1}] holds then as $n\rightarrow \infty$,
\[\left(\frac{1}{\sqrt{n}}\sum_{i=1}^{\lfloor t\sqrt{n}\rfloor}(\widehat D^n_i - 1), \frac{1}{n^{1/4}}\sum_{i=1}^{\lfloor t\sqrt{n}\rfloor} Y_{\widehat D^n_i, U_i}\right)_{t \ge 0} \convdist (\sigma^2 t, \beta B_t)_{t\ge 0},\]
for the topology of uniform convergence on compact time-intervals, where $(B_t)_{t \ge 0}$ is a standard Brownian motion.
\end{prop}

\begin{proof}
Fix $T > 0$ and let $F:\bD([0,T], \R)^2\rightarrow \R$ be a bounded continuous function, where $\bD([0,T], \R)$ is the space of real-valued functions on $[0,T]$ that are right-continuous with left limits equipped with the Skorokhod topology. Let $\bar\xi_1,\bar\xi_2,\dots$ be \iid samples from the size biased distribution of $\xi$.  Further, independently for $i\ge 1$, let $\overline U_i$ be a Uniform$([\bar\xi_i])$ random variable. 
    
By Proposition \ref{prop:measure_change_basic},
\begin{align}\label{eq:to_converge}
    &\E{F\left(\left(\frac{1}{\sqrt{n}}\sum_{i=1}^{\lfloor t\sqrt{n}\rfloor}(\widehat D_i^n - 1), \frac{1}{n^{1/4}}\sum_{i = 1}^{\lfloor t\sqrt{n}\rfloor} Y_{\widehat D_i^n, U_i}\right)_{0 \le t \le T}\right)\I{N_n \ge \lfloor T\sqrt{n}\rfloor}}\nonumber\\
    &\quad = \E{F\left(\left(\frac{1}{\sqrt{n}}\sum_{i=1}^{\lfloor t\sqrt{n}\rfloor}(\bar \xi_i - 1), \frac{1}{n^{1/4}}\sum_{i = 1}^{\lfloor t\sqrt{n}\rfloor} Y_{\bar\xi, \overline U_i}\right)_{0 \le t \le T}\right)\Theta^n\left(\bar\xi_1,\dots,\bar\xi_{\lfloor T\sqrt{n} \rfloor }\right)},
\end{align}
where the random variables $(Y_{\bar\xi_i, \overline U_i})_{i\ge 1}$ are independent and, given $\bar\xi_i$, we have that $Y_{\bar\xi_i, \overline U_i}$ is a uniform entry of a $\nu_{\bar\xi_i}$ distributed displacement vector. Since $\E{\bar\xi_1} = \sigma^2 + 1$, by the functional strong law of large numbers (Lemma \ref{lem:serflln}), as $n\rightarrow \infty$, 
\begin{equation*}
\left(\frac{1}{\sqrt{n}}\sum_{i=1}^{\lfloor t\sqrt{n}\rfloor}(\bar \xi_i - 1)\right)_{t \ge 0} \convprob (\sigma^2 t)_{t\ge 0}
\end{equation*}
in $\mathbf{D}((0,T), \R)$. 

Furthermore, the random variables $(Y_{\bar\xi, \overline U_i})_{i\ge 1}$ are \iid with mean and variance given by 
\[
\E{ Y_{\bar\xi_1, \overline U_1}} = \sum_{k=1}^\infty \mu_k \sum_{j=1}^k \E{Y_{k,j}} = 0,\quad \V{Y_{\bar\xi_1, \overline U_1}} =\sum_{k=1}^\infty \mu_k\sum_{j=1}^k\E{Y_{k,j}^2} = \beta^2.
\]
It then follows from Donsker's theorem that as $n\rightarrow\infty$
\[
\left(\frac{1}{n^{1/4}}\sum_{i = 1}^{\lfloor t\sqrt{n}\rfloor}Y_{\bar\xi_i, \overline U_i}\right)_{t \ge 0} \convdist (\beta B_t)_{t \ge 0}
\]
in $\mathbf{D}([0,T], \R)$. Therefore, by the continuity of $F$, as $n\rightarrow \infty$,
\[
\E{F\left(\left(\frac{1}{\sqrt{n}}\sum_{i=1}^{\lfloor t\sqrt{n}\rfloor}(\bar \xi_i - 1), \frac{1}{n^{1/4}}\sum_{i = 1}^{\lfloor t\sqrt{n}\rfloor}Y_{\bar\xi_i, \overline U_i}\right)_{0 \le t \le T}\right)} \rightarrow \E{F\left((\sigma^2 t, \beta B_t)_{0 \le t \le T}\right)}.
\]
Combining this with Lemma~\ref{lem:conv_measure_change_basic} and the boundedness of $F$ yields that (\ref{eq:to_converge}) converges to
\[
\E{F\left((\sigma^2 t, \beta B_t)_{0 \le t \le T}\right)},
\]
as $n\rightarrow \infty$, and the result follows.
\end{proof}

\subsection{Bijective construction on the timescale $\sqrt{n}$} \label{sec:bijection_sqrtn}
In this subsection we show that the subtree $\rT_n^k$ is constructed on a timescale of order $\sqrt{n}$ with high probability. We then prove that the lengths of the paths which are glued together to form $\rT_n^k$ converge on rescaling, as do the positions at which they attach to one another.

We begin by showing that, with high probability, the vertices $1,\dots, k$ do not appear in the first $\Theta(\sqrt{n})$ entries of $\Pi_{D^n}$. 
\begin{lem} \label{lem:conv_good_event} 
Fix $T > 0$ and $k \ge 1$, and let 
\[
\cG_{n,k}(T) = \left\{\left\{\widehat{V}_1(\Pi_{D^n}),\dots, \widehat{V}_{\lfloor T\sqrt{n}\rfloor}(\Pi_{D^n})\right\}\cap \{1,\dots,k\} = \emptyset, N_n \ge \lfloor T\sqrt{n}\rfloor\right\}.
\]
Then $\p{\cG_{n,k}(T)} \rightarrow 1$ as $n\rightarrow \infty.$
\end{lem}

Notice that on the \emph{good event} $\cG_{n,k}(T)$, if $J_k^n \le \lfloor T\sqrt{n}\rfloor$ then $J_i^{D^n} = \widetilde{J}_i^{D^n}$ for all $i\in [k]$, and $\rT_n^k$ is precisely the tree spanned by the root and the paths $P^{(1)},\dots, P^{(k)}$ in the bijective construction $B(\Pi_{D^n})$ of $\rT_n$. 

\begin{proof}
We have
\begin{align*}
    \p{\cG_{n,k}(T)} &= \p{\{\widehat{V}_1(\Pi_{D^n}),\dots,\widehat{V}_{\lfloor T\sqrt{n}\rfloor}(\Pi_{D^n}))\}\cap \{1,\dots,k\} = \emptyset, N_n \ge \lfloor T\sqrt{n}\rfloor}\\
    &\ge \E{\left(1 - \frac{D^n_1 + \dots + D^n_k}{n - 1 - T\sqrt{n}\max_{1\le i \le n}D_i^n}\right)^{\lfloor T\sqrt{n}\rfloor}}-\p{N_n< \lfloor T \sqrt n\rfloor }.
\end{align*}
Let $\varepsilon > 0$ and $(\xi_i)_{i\ge 1}$ be a sequence of \iid  random variables with distribution $\mu$. Then, 
\begin{equation}\label{eq:max_deg}
\begin{split}
    \p{\max_{1\le i \le n}D_i^n > \varepsilon \sqrt{n}} &= \frac{\p{\max_{1\le i\le n}\xi_i > \varepsilon\sqrt{n}, \sum_{i=1}^n \xi_i = n-1}}{\p{\sum_{i=1}^n \xi_i = n-1}}\\
    &\le \frac{n\p{\xi_1 > \varepsilon \sqrt{n}, \sum_{i=1}^n\xi_i = n-1}}{\p{\sum_{i=1}^n\xi_i = n-1}}.
\end{split}
\end{equation}
Since $\E{\xi^2} < \infty$ we have $n\p{\xi_1 > \varepsilon\sqrt{n}} \rightarrow 0$ as $n\rightarrow \infty.$ Hence,  
\begin{align*}
    &\frac{n\p{\xi_1 > \varepsilon \sqrt{n} , \sum_{i=1}^n\xi_i = n-1}}{\p{\sum_{i=1}^n\xi_i = n-1}}\\ &\quad\le 
    \frac{n\p{\xi_1 > \varepsilon\sqrt{n}}\max_{\varepsilon \sqrt{n} < m \le n-1}\p{\sum_{i=2}^n\xi_i = n - 1 - m}}{\p{\sum_{i=1}^n\xi_i = n-1}} \rightarrow 0
\end{align*}
as $n\rightarrow \infty.$ Combining this with (\ref{eq:max_deg}) gives that 
\[
\frac{1}{\sqrt{n}}\max_{1\le i \le n}D_i^n \convprob 0
\]
and so
\[
\frac{D_1^n + \dots + D_k^n}{\sqrt{n}} \convprob 0
\]
as $n\rightarrow \infty.$ Therefore, by the bounded convergence theorem, 
\[
\E{\left(1 - \frac{D_1^n + \dots + D_k^n}{n - 1 - T\sqrt{n}\max_{1\le i \le n}D_i}\right)^{\lfloor T\sqrt{n}\rfloor}} \rightarrow 1
\]
as $n\rightarrow \infty$. The result then follows by noting (as at the end of the proof of Lemma \ref{lem:conv_measure_change_basic}) that as $n\rightarrow \infty$, 
\[
\p{N_n< \lfloor T \sqrt n\rfloor }\le \frac{\p{\operatorname{Binomial}(n,1-\mu_0)< \lfloor T \sqrt n\rfloor }}{\p{\sum_{i=1}^n\xi_i = n-1}}\to 0. \qedhere
\]
\end{proof}

\begin{prop}\label{prop:conv_cuts_attach}
Fix $k \ge 1$. Then
\begin{equation}\label{eq:conv_attach}
\frac{\sigma}{\sqrt{n}}(J_1^n, J_2^n,\dots, J_k^n, A_1^n,\dots, A_k^n) \convdist (J_1, J_2,\dots, J_k, A_1,\dots, A_k)
\end{equation}
as $n\rightarrow \infty$, where $J_1, J_2,\dots, J_k$ are the first $k$ jump-times of an inhomogeneous Poisson process of intensity $t$ with respect to the Lebesgue measure at $t\in \R_+$ and, for $i\in [k],$ conditionally on $J_1,\dots, J_i$, $A_i$ is uniform on $[0, J_i]$, independently of $A_1,\dots, A_{i-1}.$
\end{prop}

\begin{proof}
Fix $T > 0$.  Let $0 \le t_1 \le \dots \le t_k \le T$ and $s_1 < t_1, \ldots, s_k < t_k$. We will prove that
\begin{align}
& \p{J_1^n \le t_1 \sqrt{n}, \ldots, J_k^n \le t_k\sqrt{n}, A_1^n \le s_1 \sqrt{n}, \ldots, A_k^n \le s_k \sqrt{n}} \notag \\
& \qquad \to \sigma^{2k} \left(\prod_{j=1}^k s_j \right) \int_0^{t_1} \cdots \int_{r_{k-1}}^{t_k}\exp(-\sigma^2t_k^2/2) dr_k \ldots dr_1 \notag \\
& \qquad = \p{J_1 \le \sigma t_1, \ldots J_k \le \sigma t_k, A_1 \le \sigma s_1, \ldots, A_k \le \sigma s_k}. \label{eq:distconv}
\end{align}
We will often work conditionally on the random variables $\widehat{D}^n = (\widehat{D}_1^n, \ldots, \widehat{D}_{N_n}^n)$. To make the equations easier to read, we write $\mathbf{P}_{\widehat{D}^n}$ for the conditional probability given $\widehat{D}^n$ and $\mathbf{E}_{\widehat{D}^n}$ for the corresponding expectation.

Fix $T' > T$. By Skorokhod's representation theorem, there exists a probability space on which the uniform convergence 
\begin{equation} \label{eq:forSkorokhod}
\left(\frac{1}{\sqrt{n}} \sum_{i=1}^{\fl{t \sqrt{n}}} (\widehat{D}_i^n - 1)\right)_{0 \le t \le T'} \convdist (\sigma^2t)_{0 \le t \le T'}
\end{equation}
from Proposition~\ref{prop:conv_degrees_disps} occurs in the almost sure sense. We work on this probability space for the rest of the proof. Note, in particular, that if the above convergence occurs almost surely then it is also the case that
\[
\I{N_n \ge \fl{T\sqrt{n}}} \convas 1.
\]

We first show that, as $n \to \infty$,
\begin{equation}\label{eq:jointdensity}
\begin{split}
& n^{k/2}\pD{\widetilde{J}_1^n = \fl{t_1 \sqrt{n}}, \widetilde{J}_2^n = \fl{t_2\sqrt{n}},\dots, \widetilde{J}_k^n = \fl{t_k\sqrt{n}}} \I{N_n \ge \fl{T \sqrt{n}}} \\
& \qquad \convas \sigma^{2k} t_1 t_2 \ldots t_k \exp \left(-\sigma^2 t_k^2/2\right). 
\end{split}
\end{equation}
By Lemma \ref{lem:sb_order}, whenever the bijective construction $B(\Pi_{D^n})$ of $\rT_n$ encounters a new vertex, its degree is distributionally equivalent to the next one on the list $(\widehat D^n_1,\dots, \widehat D^n_{N_n})$. So by Lemma \ref{lem:cuts}, on the event $\{N_n \ge \fl{T\sqrt{n}}\}$, we have
\begin{align*}
& \pD{\widetilde{J}_1^n = \fl{t_1\sqrt{n}}} \\
& = \frac{\sum_{\ell = 1}^{\fl{t_1 \sqrt{n}}-1} (\widehat D_\ell^n - 1)}{n-\fl{t_1\sqrt{n}}} \prod_{j=1}^{\lfloor t_1 \sqrt{n}\rfloor-2}\left(1 - \frac{\sum_{\ell = 1}^j (\widehat D_\ell^n - 1)}{n-1-j}\right) \\
&=\frac{\sum_{\ell = 1}^{\fl{t_1 \sqrt{n}}-1} (\widehat D_\ell^n - 1)}{n-\fl{t_1\sqrt{n}}}  \exp\left(\sum_{j=1}^{\lfloor t_1\sqrt{n}\rfloor-2}\log\left(1 - \frac{\sum_{\ell = 1}^j (\widehat D_\ell^n - 1)}{n-1-j}\right)\right) \\
&=\frac{\sum_{\ell = 1}^{\fl{t_1 \sqrt{n}}-1} (\widehat D_\ell^n - 1)}{n-\fl{t_1\sqrt{n}}} \exp\left(\sum_{j=1}^{\lfloor t_1 \sqrt{n}\rfloor -2}\log\left(1 - \frac{\sum_{\ell = 1}^j (\widehat D_\ell^n - 1 - \sigma^2) + \sigma^2j}{n-1-j}\right)\right).
\end{align*}
By (\ref{eq:forSkorokhod}) and a similar argument to that used in the proof of (\ref{eq:use_in_ms}), we get that as $n\rightarrow \infty$,
 \[
\sqrt{n} \pD{\widetilde{J}_1^n = \fl{t_1\sqrt{n}}} \I{N_n \ge \fl{T \sqrt{n}}} \convas \sigma^2 t_1 \exp\left(-\frac{t_1^2\sigma^2}{2}\right).
 \]

We now proceed to prove the joint convergence of the first $k$ coordinates in (\ref{eq:jointdensity}) by induction. Suppose that the claimed convergence holds for $\widetilde{J}_1^n$,\dots, $\widetilde{J}_{m-1}^n$. By Lemma \ref{lem:later_cuts}, on the event $\{N_n \ge \fl{T \sqrt{n}}\}$,
\begin{align*}
&\pD{\widetilde{J}_m^n - \widetilde{J}_{m-1}^n = \fl{t_m\sqrt{n}} - \fl{t_{m-1}\sqrt{n}} ~\Big|~ \widetilde{J}_1^n = \fl{t_{1}\sqrt{n}}, \ldots, \widetilde{J}_{m-1}^n = \fl{t_{m-1}\sqrt{n}}}\\
& \quad = \frac{\sum_{\ell = 1}^{\fl{t_m\sqrt{n}}-m}(\widehat D_\ell^n - 1) - m + 1}{n-\fl{t_{m} \sqrt{n}}} \prod_{j = \fl{t_{m-1}\sqrt{n}}}^{\fl{t_m \sqrt{n}} -2}\left(1 - \frac{\sum_{\ell = 1}^{j-m+1}(\widehat D_\ell^n - 1) - m + 1}{n-1-j}\right).
\end{align*}

Arguing as above, we obtain
\begin{align*}
\sqrt{n} & \pD{\widetilde{J}_m^n = \fl{t_m\sqrt{n}} ~\Big|~ \widetilde{J}_1^n = \fl{t_{1}\sqrt{n}}, \ldots, \widetilde{J}_{m-1}^n = \fl{t_{m-1}\sqrt{n}}}\I{N_n \ge \fl{t_m \sqrt{n}}} \\
&  \convas
\sigma^2 t_m \exp\left(-\int_{t_{m-1}}^{t_{m}}\sigma^2 rdr\right).
\end{align*}
By induction on $m$, we get this for all $1 \le m \le k$. Taking the product of the conditional probabilities, we obtain (\ref{eq:jointdensity}).

We now wish to add in the random variables $(A_i^n)_{i\in [k]}$. We work conditionally on the event $\mathcal{G}_{n,k}(T)$. Given also $J_1^n = \fl{t_1 \sqrt{n}}, \ldots, J_{m}^n = \fl{t_m \sqrt{n}}$, $\widehat{D}_1^n, \ldots, \widehat{D}_{N_n}^n$ and $A^n_1,\dots, A_{m-1}^n$, since $N_n \ge \fl{T\sqrt{n}}$, at time $J_{m}^n$  there are $\widehat D_i^n - 1 - \sum_{\ell = 1}^{m-1}\I{A_\ell^n = i}$ remaining instances of the vertex $\widehat{V}_i(\Pi_{D^n})$ to appear in the bijective construction. So, the repeated vertex that we see is $\widehat{V}_i(\Pi_{D^n})$, i.e.\ $A_m^n = i$, with probability 
\[
\frac{\widehat D_i^n - 1 - \sum_{\ell = 1}^{m-1}\I{A_\ell^n = i}}{\sum_{j=1}^{\fl{t_m \sqrt{n}} - m}(\widehat D_i^n - 1) - m },
\]
for $1 \le i \le \fl{t_m \sqrt{n}} - m$. Hence, 
\begin{align*}
&\pD{A_m^n \le s_m\sqrt{n} ~|~ \cG_{n,k}(T), J_1^n = \fl{t_1 \sqrt{n}}, \ldots, J_{m}^n = \fl{t_m \sqrt{n}}, A_1^n,\dots, A_{m-1}^n} \\
&\qquad=\frac{\sum_{i=1}^{\lfloor s_m \sqrt{n}\rfloor}(\widehat D_i^n -1) - \sum_{\ell = 1}^{m-1}\I{A_\ell^n \le s_m \sqrt{n}}}{\sum_{j=1}^{\fl{t_m\sqrt{n}} - m}(\widehat D_j^n - 1) - m }.
\end{align*}
This quantity lies in the interval
\[
\left[\frac{\sum_{i=1}^{\lfloor s_m \sqrt{n}\rfloor}(\widehat D_i^n -1) - m+1}{\sum_{j=1}^{\fl{t_m\sqrt{n}} - m}(\widehat D_j^n - 1) - m }, \frac{\sum_{i=1}^{\lfloor s_m \sqrt{n}\rfloor}(\widehat D_i^n -1)}{\sum_{j=1}^{\fl{t_m\sqrt{n}} - m}(\widehat D_j^n - 1) - m }\right]
\]
whose end-points do not depend on $A_1^n, A_2^n, \ldots, A_{m-1}^n$. Iterating, we thus obtain that
\[
\pD{A_1^n \le s_1 \sqrt{n}, \ldots, A_k^n \le s_k \sqrt{m} ~\Big|~\cG_{n,k}(T), J_1^n = \fl{t_1 \sqrt{n}}, \ldots, J_{m}^n = \fl{t_m \sqrt{n}}}
\]
lies in a random interval depending only on $\widehat{D}_1^n, \ldots, \widehat{D}_{\fl{t_k \sqrt{n}}}$, both of whose end-points converge almost surely to $\prod_{m=1}^k (s_m/t_m)$ by (\ref{eq:forSkorokhod}). So the same is true by sandwiching for our conditional probability which lies in that interval.  

Putting everything together, we then have
\begin{align*}
& \p{J_1^n \le t_1\sqrt{n}, \ldots, J_k^n \le t_k \sqrt{n}, A_1^n \le s_1 \sqrt{n}, \ldots, A_k \le s_k \sqrt{n}} \\
& \qquad = \p{J_1^n \le t_1\sqrt{n}, \ldots, J_k^n \le t_k \sqrt{n}, A_1^n \le s_1 \sqrt{n}, \ldots, A_k^n \le s_k \sqrt{n}, \cG_{n,k}(T)^c} \\
& \qquad\hspace{0.5cm} + \E{ \pD{J_1^n \le t_1\sqrt{n}, \ldots, J_k^n \le t_k \sqrt{n}, A_1^n \le s_1 \sqrt{n}, \ldots, A_k^n \le s_k \sqrt{n}, \cG_{n,k}(T)}}.
\end{align*}
The first term on the right-hand side of this equation clearly tends to 0 by Lemma~\ref{lem:conv_good_event}. Since the second is the expectation of a conditional probability, it is sufficient to show that the conditional probability itself tends to $\exp(-\sigma^2 t_k/2)\prod_{m=1}^k s_k$ in distribution. For $1 \le m \le k$ and $n \ge 1$, let us write
\[
t_m^n = \frac{\fl{t_m \sqrt{n}}+1}{\sqrt{n}}.
\]
Then we have
\begin{align*}
& \pD{J_1^n \le t_1\sqrt{n}, \ldots, J_k^n \le t_k \sqrt{n}, A_1^n \le s_1 \sqrt{n}, \ldots, A^n_k \le s_k \sqrt{n}, \cG_{n,k}(T)} \\
& = \! \! \int_0^{t_1^n} \hspace{-10pt}\ldots \hspace{-3pt} \int_{r_{k-1}}^{t_k^n} \! \! \! \pD{A_1^n \le s_1 \sqrt{n}, \ldots, A_k \le s_k \sqrt{n}\Big|\cG_{n,k}(T), J_1^n = \fl{r_1\sqrt{n}}, \ldots, J_k^n = \fl{r_k \sqrt{n}}}  \\
& \hspace{2.2cm} \times n^{k/2} \pD{\widetilde{J}_1^n = \fl{r_1\sqrt{n}}, \ldots, \widetilde{J}_k^n = \fl{r_k \sqrt{n}}, \cG_{n,k}(T)}  dr_k \ldots dr_1 \\
& = \! \! \int_0^{t_1^n} \hspace{-10pt} \ldots \hspace{-3pt} \int_{r_{k-1}}^{t_k^n} \hspace{-5pt}  \pD{A_1^n \le s_1 \sqrt{n}, \ldots, A_k \le s_k \sqrt{n}\Big|\cG_{n,k}(T), J_1^n = \fl{r_1\sqrt{n}}, \ldots, J_k^n = \fl{r_k \sqrt{n}}}  \\
& \hspace{2cm} \times n^{k/2} \pD{\widetilde{J}_1^n = \fl{r_1\sqrt{n}}, \ldots, \widetilde{J}_k^n = \fl{r_k \sqrt{n}}} \I{N_n \ge \fl{T \sqrt{n}}}  dr_k \ldots dr_1 \hspace{-1.1pt}-\hspace{-1.1pt} E_n,
\end{align*}
where $E_n$ is an error term with the property that $0 \le E_n \le \pD{\cG_{n,k}(T)^c}$, and so tends to $0$ in distribution as $n \to \infty$. The first term in the product which forms the integrand tends to $\prod_{m=1}^k (s_m/r_m)$ as $n \to \infty$ and the second term tends to $\sigma^{2k}r_1 \ldots r_k \exp(-\sigma^2 r_k^2/2)$, both almost surely. Write $g_n(r_1, \ldots, r_k)$ for the integrand above, considered as a function of $r_1, \ldots, r_k$. Then we have just shown that
\[
g_n(r_1, \ldots, r_k) \convas  \sigma^{2k} \prod_{m=1}^k s_m \exp(-\sigma^2 r_k^2/2), 
\]
It is straightforward to see that this convergence is, in fact, uniform on compacts. Hence,
\begin{align*}
& \int_0^{t_1^n} \! \! \cdots \! \int_{r_{k-1}}^{t_k^n} g_n(r_1, \ldots, r_k) dr_k \ldots dr_1 \\
& \qquad \convas \sigma^{2k} \left(\prod_{m=1}^k s_m \right) \int_0^{t_1} \! \! \cdots \! \int_{r_{k-1}}^{t_k} \exp(-\sigma^2r_k^2/2) dr_k \ldots dr_1,
\end{align*}
which yields (\ref{eq:distconv}). The result follows, since $T > 0$ was arbitrary.
\end{proof}

This completes the proof of (\ref{eq:conv_attach-label}) in Proposition~\ref{prop:fdds}.

\subsection{Displacements at repeats}
As shown above, for fixed $k$ and large $n$, $\rT_n^k$ is with high probability the subtree of $B(\Pi_{D^n})$ composed of the union of the paths $P^{(1)},\dots, P^{(k)}$. Moreover, for $i\in [k]$, under the bijective construction, by Proposition \ref{prop:conv_degrees_disps}, with the exception of the first vertex in each path $P^{(i)}$, the displacements of the vertices in $P^{(i)}$ away from their parents are asymptotically indistinguishable from \iid copies of uniform entries of a $\nu_{\bar\xi}$ distributed displacement vector. On the other hand, the displacement away from of the first vertex in $P^{(i)}$ cannot be compared to a random variable with the same distribution as a uniform entry of a~$\nu_{\bar\xi}$ distributed displacement vector. However, in the following lemma we will prove that such displacements are $O_\bP(1)$ and so negligible on the scale of $n^{1/4}$.

We first introduce some notation. Recall that for $i\in [\ell^*(\Pi_{D^n})]$, vertex $V_{J_i^n}$ is the $i$-th repeated vertex encountered in the bijective construction $(B(\Pi_{D^n}), Y)$ of $\bT_n = (\rT_n, Y)$ (and hence a branchpoint). For $i\in[\ell^*(\Pi_{D^n})]$, let $\Delta_i^n$ be the displacement of $V_{J_i^n + 1}$ away from its parent $V_{J_i^n}$ in  $\bT_n$.

\begin{lem}
    \label{lem:attach_points} For any $\ell \ge 0$, $\max\{|\Delta_1^n|,\dots, |\Delta_\ell^n|\}$ is a tight sequence of random variables for $n \ge 1$.
\end{lem}

\begin{proof} 
We will prove that for all $\varepsilon > 0$ there exists $N > 0$ such that for all $n \ge N$, $\p{|\Delta_1^n| > N} < \varepsilon$. 
    
To prove the result for $|\Delta_2^n|,\dots, |\Delta_\ell^n|$, note that by Proposition \ref{prop:conv_cuts_attach}, since $(A_i)_{i \in [k]}$ are almost surely distinct, we have
\[
\p{(A_i^n)_{i\in [k]} \text{ are distinct}}\rightarrow 1
\]
as $n\rightarrow \infty$. On the event $\{(A_i^n)_{i\in [k]} \text{ are distinct}\}$ the proof for $|\Delta_2^n|,\dots, |\Delta_\ell^n|$ is analogous to that for $|\Delta_1^n|$ and so we omit it. 

Recall from Proposition \ref{prop:conv_cuts_attach} that $\sigma n^{-1/2}J_1^n \convdist J_1$. Recalling also that $A_1^n$ is such that $V_{J_1^n} = \widehat V_{A_1^n}(\Pi_{D^n})$, it follows that conditionally on $\widehat D^n_{A_1^n} = k$, $\Delta_1^n \eqdist Y_{k,U_k}$, where $U_k \eqdist \mathrm{Uniform}([k])$ and $Y_{k, U_k}$ is distributed as a uniform entry of a displacement vector with law $\nu_k$, independent of $D^n$.  Fix $T > 0$ large. We work on the event $\{ J_1^n \le  T\sqrt{n}\}$.  For $N > 0$ and $K \ge 1$,
\begin{align}
    &\p{|\Delta_1^n| > N,~J_1^n \le  T\sqrt{n}}\nonumber\\
    & \le \p{\widehat D^n_{A_1^n} > K,~J_1^n \le  T\sqrt{n}} + \p{|\Delta_1^n| > N,~ \widehat D^n_{A_1^n} 
        \le K,~ J_1^n \le  T\sqrt{n}}\nonumber\\
    & \le \p{\widehat D^n_{A_1^n} > K,~J_1^n \le T\sqrt{n}}+ \sum_{k=2}^K\frac{k(k-1)\mu_k}{\sigma^2}\p{|Y_{k,U_k}| > N}\p{ J_1\le \sigma T}\nonumber\\
    & + \sum_{k=2}^K\left|\p{|\Delta_1^n| \hspace{-1pt}>\hspace{-1pt} N,\widehat D_{A_1^n}^n \hspace{-1pt}=\hspace{-1pt} k,J_1^n \hspace{-1pt}\le\hspace{-1pt} T\sqrt{n}} - \frac{k(k-1)\mu_k}{\sigma^2}\p{|Y_{k,U_k}| > N}\p{J_1 \le \sigma T}\right|. \nonumber
\end{align}
We have
\[
    \p{|\Delta_1^n| > N, \widehat{D}^n_{A_1^n}=k, J_1^n \le T \sqrt{n}} = \p{|Y_{k,U_k}| > N} \p{\widehat{D}^n_{A_1^n}=k, J_1^n \le T \sqrt{n}}
\]
and so
\begin{equation}\label{eq:attach_points_1}
\begin{split}
    &\p{|\Delta_1^n| > N,~J_1^n \le  T\sqrt{n}} \\
    &\le \p{\widehat D^n_{A_1^n} > K,~J_1^n \le  T\sqrt{n}}+ \sum_{k=2}^K\frac{k(k-1)\mu_k}{\sigma^2}\p{|Y_{k, U_k}| > N} \\
    &\quad\quad + \sum_{k=2}^K\left|\p{\widehat D^n_{A_1^n} = k, ~J_1^n \le  T\sqrt{n}} - \frac{k(k-1)\mu_k}{\sigma^2}\p{ J_1\le \sigma T}\right|.
\end{split}
\end{equation}
Since $k\mu_k \le 1$ for all $k$, it follows that (\ref{eq:attach_points_1}) is at most
\begin{equation}\label{eq:attach_points}
\begin{split}
    &\p{\widehat D^n_{A_1^n} > K,~J_1^n \le  T\sqrt{n}} + \frac{K-1}{\sigma^2}\p{|Y_{\bar\xi, U_{\bar\xi}}| > N}\\
    &\qquad  + K\max_{2\le k \le K}\left|\p{\widehat D_{A_1^n}^n = k, ~J_1^n \le  T\sqrt{n}} - \frac{k(k-1)\mu_k}{\sigma^2}\p{ J_1 \le \sigma T}\right|.
\end{split}
\end{equation}

Fix $\varepsilon > 0$. Since $|Y_{\bar\xi, U_{\bar\xi}}|$ is a random variable with support in $[0,\infty)$, we may take $M = M(K) > 0$ large enough so that 
\[
\frac{K-1}{\sigma^2}\p{|Y_{\bar\xi, U_{\bar\xi}}| > N} < \frac{\varepsilon}{4}.
\]

It remains to prove that for sufficiently large $n \ge 1$ and $K \ge 1$ the sum of the first and third terms in (\ref{eq:attach_points}) is at most $3\varepsilon/4$. To this end, observe that for $i \ge 1$, 
\[\pD{A_1^n = i,~J_1^n \le  T\sqrt{n}~\big|~ J_1^n} = \frac{\widehat D_i^n - 1}{\sum_{j=1}^{J_1^n -1}(\widehat D_j^n - 1)}\I{1 \le i \le J_1^n}\I{J_1^n \le  T\sqrt{n}}.
\]
Therefore, for any $k \ge 2$, 
\begin{equation*}
\pD{\widehat D^n_{A_1^n} = k, ~J_1^n \le  T\sqrt{n}~\big|~J_1^n} = (k-1)\frac{\left|\left\{1 \le i \le J_1^n~:~\widehat D_i^n = k \right\}\right|}{\sum_{j=1}^{J_1^n - 1}(\widehat D_j^n - 1)}\I{J_1^n \le  T\sqrt{n}}.
\end{equation*}
It follows that 
\begin{align*}
    &\p{\widehat D_{A_1^n}^n = k, J_1^n \le  T\sqrt{n}}\\ 
    &\quad= (k-1)\E{\frac{\left|\left\{1 \le i < J^n_1~:~ \widehat D_i^n = k\right\}\right|}{\sum_{j=1}^{J_1^n - 1}(\widehat D_j^n - 1)}\I{J_1^n \le  T\sqrt{n}}}\\
        &\quad=(k-1)\E{\frac{\left|\left\{1 \le i < J^n_1~:~ \bar \xi_i = k\right\}\right|}{\sum_{j=1}^{J_1^n - 1}(\bar \xi_j - 1)}\I{J_1^n \le  T\sqrt{n}}\Theta^n\left(\bar\xi_1,\dots,\bar\xi_{\lfloor T\sqrt{n}\rfloor}\right)},
\end{align*}
where the final equality holds by Proposition \ref{prop:measure_change_basic}.
    
By Proposition \ref{prop:conv_cuts_attach}, $J_1^n = \Theta_\bP(\sqrt{n})$, and so by a functional law of large numbers (see Lemma \ref{lem:serflln} in the appendix),
\[
\frac{\left|\left\{1 \le i < J^n_1~:~ \bar \xi_i = k\right\}\right|}{\sum_{j=1}^{J_1^n - 1}(\bar \xi_j - 1)} \convprob \frac{k\mu_k}{\sigma^2}.
\]
Combining this with Lemma \ref{lem:conv_measure_change_basic} we obtain that as $n\rightarrow \infty$
\begin{equation}\label{eq:conv_big_jumps} 
\p{\widehat D_{A_1^n}^n = k, ~J_1^n \le  T\sqrt{n}} \rightarrow \frac{k(k-1)\mu_k}{\sigma^2}\p{J_1 \le \sigma T}.
\end{equation}
Since $\sum_{k=1}^\infty k(k-1)\mu_k/\sigma^2 = 1$, we can take $K \ge 1$ and $T > 0$ large enough so that 
\[
\p{J_1 \le \sigma T}\sum_{k=2}^K\frac{k(k-1)\mu_k}{\sigma^2} > 1 - \frac{\varepsilon}{4}.
\]
Further, by (\ref{eq:conv_big_jumps}) we can take $n \ge 1$ large enough such that 
\[
\max_{2\le k \le K}\left|\p{\widehat D_{A^n_1}^n = k,~J_1^n \le  T\sqrt{n}} - \frac{k(k-1)\mu_k}{\sigma^2}\p{J_1 \le \sigma T}\right| < \frac{\varepsilon}{4K}.
\]
For such $n$ and $T$, we have $\p{\widehat D_{A_1^n}^n > K, ~J_1^n \le  T\sqrt{n}} < \varepsilon/2.$ The result follows.
\end{proof}

\subsection{Convergence to the continuous line-breaking construction}

We are now ready to complete the proof of Proposition~\ref{prop:fdds}.

\begin{proof}[Proof of Proposition~\ref{prop:fdds}]

In view of Proposition~\ref{prop:conv_cuts_attach}, it remains to prove (\ref{eq:flagsconv}) and (\ref{eq:RWtoBM}). 

For (\ref{eq:flagsconv}), we recall from the discussion at the start of Section~\ref{sec:fdds} (where $(F_1^n, \ldots, F_k^n)$ were defined) that at attachment points which are first repeats, the attachment is to the left with probability $1/2$ and to the right with probability $1/2$. By Proposition~\ref{prop:conv_cuts_attach}, the first $k$ attachment points are distinct and are, therefore, all first repeats with probability tending to 1 as $n \to \infty$. The statement (\ref{eq:flagsconv}) follows.

For (\ref{eq:RWtoBM}), we must consider the spatial locations of the vertices along the first $k$ paths in the bijective construction. We work on the event that the paths $P^{(1)}, \ldots, P^{(k)}$ terminate in vertices $1, 2, \ldots, k$ respectively, which we have already shown holds with high probability as $n \to \infty$. For the first path, we have 
\[
L^n(\fl{tn^{1/2}}\wedge (J_1^n-1)) = \sum_{j=1}^{\fl{tn^{1/2}} \wedge (J_1^n-1)} Y_{\widehat{D}^n_j, U_j}
\]
and, for $1 \le i \le k-1$,
\[
L^n((J_i^n + \fl{tn^{1/2}}) \wedge (J_{i+1}^n-1)) = L^n(A_i^n+i-2) + \Delta_{i}^n + \sum_{j=J_i^n+1}^{(J_i^n + \fl{tn^{1/2}})\wedge (J_{i+1}^n-1)} Y_{\widehat{D}_j^n,U_j}.
\]
The desired convergence then follows from Propositions~\ref{prop:conv_degrees_disps} and \ref{prop:conv_cuts_attach} and Lemma~\ref{lem:attach_points}.
\end{proof}

\section{Tightness}\label{sec:tightness}

We assume throughout the section that $\mu$ is critical and has finite variance $\sigma^2\in (0,\infty)$, and that [\ref{a1}] and [\ref{a2}] hold.

Let $k \ge 1$. Recall that $\rT_n^k$ is the subtree of $\rT_n$ spanned by the root and the vertices $v_{U_1^n},\dots v_{U_k^n}\in\rT_n$, where $(U^n_1,\ldots, U^n_k)$ is a uniformly random $k$-set sampled from $[n]$ and, for $i\in [n]$, $v_i$ is the $i$-th vertex in the lexicographical order of $\rT_n$. In what follows we write $(U^{n,k}_{(1)},\ldots,U^{n,k}_{(k)})$ for the increasing rearrangement of $(U^n_1,\ldots, U^n_k)$. Further, recall from Section \ref{sec:background} that $\bT_n = (\rT_n, Y)$ is the $(\mu, \nu)$-branching random walk conditioned to have size $n$. 

\begin{prop}\label{prop:tightness}
    Suppose that [\ref{a1}] holds. Then for all $\gamma > 0$, 
     \begin{equation}
         \label{eq:height_tight}
         \lim_{k \to \infty} \limsup_{n \to \infty} \p{\max_{0 \le i \le k} \sup_{s,t \in [U_{(i)}^{n,k}-1, U_{(i+1)}^{n,k}-1]}|H_n(s) - H_n(t)| > \gamma n^{1/2}} = 0
     \end{equation}
    
    and, additionally, if [\ref{a2}] holds, then
    \begin{equation}
        \label{eq:labels_tight}\lim_{k \to \infty} \limsup_{n \to \infty} \p{\max_{0 \le i \le k} \sup_{s,t \in [U_{(i)}^{n,k}-1, U_{(i+1)}^{n,k}-1]}|R_n(s) - R_n(t)| > \gamma n^{1/4}} = 0.
    \end{equation}
  
\end{prop}

Under [\ref{a1}], we have that
\[
\left(\frac{H_n(nt)}{\sqrt{n}}\right)_{0 \le t \le 1} \convdist \left(\frac{2}{\sigma} \mathbf{e}_t\right)_{0 \le t \le 1}
\]
in $\mathbf{C}([0,1],\R)$ and so (\ref{eq:height_tight}) holds. It follows that we only need to prove (\ref{eq:labels_tight}).

Let us immediately observe that the vertices of the tree $\rT_n$ either belong to $\rT_n^k$ or belong to a subtree hanging off $\rT_n^k$. In Proposition~\ref{prop:fdds}, we showed the convergence of the spatial locations along the subtree $\rT_n^k$ to those given by a Brownian motion indexed by $\cT^k$.
This has the consequence that for values $s,t \in [U_{(i)}^{n,k}-1,U_{(i+1)}^{n,k}-1]$ such that both corresponding vertices lie in $\rT_n^k$, we have that $|R_n(s) - R_n(t)|$ is bounded above by the maximum modulus $\Upsilon_i^{n,k}$ of an increment of the location process along the path from $U_{(i)}^{n,k}$ to $U_{(i+1)}^{n,k}$ in $\rT_n^k$. 
Moreover, this upper bound converges in distribution on rescaling to the analogous quantity in the limit tree, which has the same distribution as the maximum modulus $\Upsilon_i^k$ of an increment of $\beta$ times a Brownian motion run for time $D_i^k$, where $D_i^k$ is the distance between the $i$th and $(i+1)$st leaves of $(2/\sigma) \cT^k$ in planar order. 
We thus have that
\[
\max_{0 \le i \le k} \Upsilon_i^k \eqdist \beta \sqrt{\frac{2}{\sigma}} \max_{0 \le i \le k} \sup_{s,t \in [U_{(i)}^k, U_{(i+1)}^k]} |\mathbf{r}_s - \mathbf{r}_t|.
\]
But
\[
\max_{0 \le i \le k} \left(U_{(i+1)}^k - U_{(i)}^k\right) \convas 0
\]
as $k \to \infty$ and so, since $\mathbf{r}$ is uniformly continuous, we may deduce that for any $\gamma > 0$,
\begin{equation} \label{eq:displ}
\lim_{k \to \infty} \lim_{n \to \infty} \p{\max_{0 \le i \le k} \Upsilon_i^{n,k} > \gamma n^{1/4}} = \lim_{k \to \infty} \p{\max_{0 \le i \le k}\Upsilon_i^k > \gamma} = 0.
\end{equation}

For values $s,t \in [U_{(i)}^{n,k}, U_{(i+1)}^{n,k}]$ for some $0 \le i \le k$ such that at least one of the corresponding vertices does not lie in $T_n^k$, we may bound $|R_n(s) - R_n(t)|$ by $\Upsilon_i^{n,k}$ plus twice the maximum modulus of the difference in spatial location between the parent in $\rT_n^k$ of the root of a pendant subtree and some other vertex inside the tree. We have already dealt with $\Upsilon_i^{n,k}$, and so it remains to deal with the pendant subtrees. Before we can do so, we need to do some truncation of the displacements.

Fix $\gamma > 0$ and $\delta \in (0,1/4)$. We will consider three ``restrictions'' of the branching random walk $\bT_n = (\rT_n, Y)$, which we denote by $\bT_{n,\delta} = (\rT_n, Y_{n,\delta}),$ $\bT_{n,\delta}^\gamma = (\rT_n, Y_{n,\delta}^\gamma)$, and $\bT_n^\gamma = (\rT_n, Y_n^\gamma)$. These branching random walks capture the ``typical'', ``mid-range'', and ``large'' spatial displacements in $\bT_n$.
\begin{enumerate}
\item {\bf (typical displacements):} $Y_{n,\delta} = (Y^{(v)}_{n,\delta}, v\in v(\rT_n)\setminus \partial \rT_n)$ is such that for $v\in v(\rT_n)\setminus \partial \rT_n$,
\[Y_{n,\delta}^{(v)} = Y^{(v)}\I{\|Y^{(v)}\|_\infty \le n^{1/4-\delta}}. \]

\item {\bf (mid-range displacements):} $Y_{n,\delta}^\gamma = (Y^{\gamma, (v)}_{n,\delta}, v\in v(\rT_n)\setminus \partial \rT_n)$ is such that for all $v\in v(\rT_n)\setminus \partial \rT_n$,
\[Y_{n,\delta}^{\gamma, (v)} = Y^{(v)}\I{n^{1/4 - \delta} < \|Y^{(v)}\|_\infty \le  \gamma n^{1/4}}.\]

\item {\bf (large displacements):} $Y_{n}^\gamma = (Y^{\gamma, (v)}_{n}, v\in v(\rT_n)\setminus \partial \rT_n)$ is such that for $v\in v(\rT_n)\setminus \partial \rT_n$,
\[Y_{n}^{\gamma, (v)} = Y^{(v)}\I{\|Y^{(v)}\|_\infty >  \gamma n^{1/4}}. \]
\end{enumerate}
For $v\in v(\rT_n)\setminus \partial \rT_n$, the vectors $Y_{n,\delta}^{(v)}, Y_{n,\delta}^{\gamma, (v)}, Y_n^{\gamma, (v)}$ are all of length $c(v, \rT_n)$; however, in what follows we will not refer to their individual entries.

Let $R_{n,\delta}$, $R_{n,\delta}^\gamma$, and $R_{n}^\gamma$ denote the functions encoding the spatial locations of the branching random walks $\bT_{n,\delta}$, $\bT_{n,\delta}^\gamma$, and $\bT_n^\gamma$, respectively. Then,  for all $n$ large enough so that $n^{1/4 - \delta} \le \gamma n^{1/4}$, 
\[
R_n = R_{n,\delta} + R_{n,\delta}^\gamma + R_n^\gamma.
\]
By the triangle inequality, for all $\gamma > 0$, we then have 
\begin{align}
& \max_{0 \le i \le k} \sup_{s,t \in [U_{(i)}^{n,k}-1, U_{(i+1)}^{n,k}-1]}|R_n(s) - R_n(t)|
\notag \\
& \le \max_{0 \le i \le k} \sup_{s,t \in [U_{(i)}^{n,k}-1, U_{(i+1)}^{n,k}-1]}|R_{n,\delta}(s) - R_{n,\delta}(t)| + 2 \|R_{n,\delta}^{\gamma}\|_{\infty} + 2 \|R_n^{\gamma}\|_{\infty}.\label{eq:triangleineq}
\end{align}
We deal with each of these three terms separately.

\subsection{Large and mid-range displacements}

Under assumption [\ref{a2}], we show that the probability that there is a displacement in $\bT_n$ with modulus exceeding $\gamma n^{1/4}$ goes to zero, so that the contribution of the large displacements is negligible.
\begin{prop}\label{prop:large_tight}
   For all $\gamma > 0$,  as $n\rightarrow \infty$,
\[\p{\|R_{n}^\gamma\|_\infty > \gamma n^{1/4}} = o(1).\] 
\end{prop}

\begin{proof}  Let 
\[
M_n^\gamma := \left|\left\{v\in v(\rT_n)\setminus \partial \rT_n~:~ \|Y^{(v)}\|_\infty > \gamma n^{1/4}\right\}\right|.
\]
It suffices to prove that $\p{M_n^\gamma > 0} \rightarrow 0$ as $n\rightarrow \infty$. To this end, let $\xi_1,\dots, \xi_n$ be \iid random variables with distribution $\mu$. By assumption [\ref{a2}], 
\[\p{\|Y_{\xi_1}\|_\infty > \gamma n^{1/4}} = o(n^{-1}).\]
Fixing $\varepsilon > 0$, this implies that for $n$ large enough, 
\begin{equation}\label{eq:stoch_dom}\widetilde{M}_n^\gamma := \left|\left\{i\in [n]~:~\|Y_{\xi_i}\|_\infty > \gamma n^{1/4}\right\}\right| \preceq_{st} \mathrm{Bin}\left(n, \frac{\varepsilon}{n}\right),\end{equation}
where $\preceq_{st}$ denotes stochastic domination. It follows from a Chernoff bound that there exists $c > 0$ such that for $n$ sufficiently large,
\[
\p{\widetilde{M}^\gamma_n \ge n^\varepsilon} \le \exp\left(-cn^\varepsilon\right).
\]

Since $\p{\sum_{i=1}^n\xi_i = n-1} = \Theta(n^{-1/2}),$ we obtain 
\begin{align*}
\p{M_n^\gamma \ge n^\varepsilon} & = \p{\widetilde M_n^\gamma \ge n^\varepsilon~|~ \sum_{i=1}^n \xi_i = n-1} \\
&\le \frac{\p{\widetilde{M}_n^\gamma \ge n^\varepsilon}}{\p{\sum_{i=1}^n\xi_i = n-1}} = O\left(n^{1/2}\exp(-cn^\varepsilon)\right).
\end{align*}

Let $\widetilde{S}^\gamma_n := \sum_{i=1}^n\xi_i\I{\|Y_{\xi_i}\|_\infty > \gamma n^{1/4}}$ and let $S^\gamma_n := \sum_{v\in v(\rT_n)}c(v,\rT_n)\I{\|Y^{(v)}\|_\infty > \gamma n^{1/4}}$. Since $\E{\xi^3} < \infty$, by \cite[Corollary 19.11]{janson2012simply}, both $\max_{1\le i \le n} \xi_i$ and $\max_{v\in v(\rT_n)}c(v,\rT_n)$ are $O_\bP(n^{1/3})$, and so
\begin{align}\label{eq:hairy_refer_2}
\p{{M}^\gamma_n\ge n^\varepsilon \text{ or }  S^\gamma_n \ge n^{1/3 + \varepsilon}} 
&\le o(1) + \p{S_n^\gamma \ge n^{1/3+\varepsilon} \cap \max_{v\in v(\rT_n)}c(v,\rT_n) \le n^{1/3}}\nonumber\\
&\le o(1) + \p{\sum_{v\in v(\rT_n)}\I{\|Y^{(v)}\|_\infty \ge \gamma n^{1/4}} > n^\varepsilon}\nonumber\\
&\le o(1) + \frac{\p{\mathrm{Bin}\left(n, \frac{\varepsilon}{n}\right) > n^\varepsilon}}{\p{\sum_{i=1}^n \xi_i = n-1}} = o(1),
\end{align}
where the final inequality holds by (\ref{eq:stoch_dom}).
Further, for $\xi_1^n, \xi_2^n,\dots$ independent random variables such that for each $i\ge 1$, $\xi_i^n$ is distributed as $\xi_i$ conditional on $\|Y_{\xi_i}\|_\infty < \gamma n^{1/4}$, we have that
\[
\p{\sum_{i=1}^n \xi_i = n-1~\bigg|~\widetilde{S}^\gamma_n, \widetilde{M}_n^\gamma} = \p{\widetilde{S}^\gamma_n + \sum_{i = 1}^{n-\widetilde{M}^\gamma_n}\xi_i^n = n-1 ~\bigg|~\widetilde{S}^\gamma_n, \widetilde{M}_n^\gamma}.
\]
Therefore,  
\begin{align*}
\p{M_n^\gamma>0} & =\p{0<M_n^\gamma < n^{\varepsilon},~ S^\gamma_n < n^{1/3+ \varepsilon}} + o(1)\\
& = \p{0<\widetilde M_n^\gamma < n^{\varepsilon},~ \widetilde S^\gamma_n < n^{1/3+ \varepsilon}\bigg| \sum_{i=1}^n \xi_i=n-1} + o(1)\\
&= \frac{\p{0<\widetilde M_n^\gamma < n^{\varepsilon},~ \widetilde S^\gamma_n < n^{1/3+ \varepsilon}, \sum_{i=1}^n \xi_i=n-1 }}{\p{\sum_{i=1}^n\xi_i = n-1}} + o(1)\\
&= \E{\frac{\p{\widetilde{S}^\gamma_n \hspace{-1pt}+\hspace{-1pt} \sum_{i = 1}^{n-\widetilde{M}^\gamma_n}\xi_i^n = n\hspace{-1pt}-\hspace{-1pt}1 \bigg|\widetilde{S}^\gamma_n, \widetilde{M}_n^\gamma}}{\p{\sum_{i=1}^n\xi_i = n-1}}\I{0<\widetilde M_n^\gamma < n^{\varepsilon},~ \widetilde S^\gamma_n < n^{1/3+ \varepsilon}}} + o(1).
\end{align*}
By a quantitative local limit theorem (see Lemma \ref{lem:local_move} in the appendix), we obtain that as $n\rightarrow \infty$
\[
\frac{\p{\sum_i^{n- m}\xi_i^n = n- 1 -  s}}{\p{\sum_{i=1}^n\xi_i = n-1}}\rightarrow 1,
\]
uniformly over all $0 < m < n^\varepsilon$ and $0< s < n^{1/3 + \varepsilon}$. It follows that 
\[
\p{M_n^\gamma>0} =\p{0<\widetilde M_n^\gamma < n^{\varepsilon},~ \widetilde S^\gamma_n < n^{1/3+ \varepsilon}} + o(1)\le \p{\widetilde M_n^\gamma >0} + o(1). 
\]
The result follows since for $n$ sufficiently large $\widetilde M_n^\gamma \preceq_{st}\mathrm{Bin}(n,\varepsilon/n)$, and $\varepsilon > 0$ is arbitrary.
 \end{proof}

Similarly to the large displacements, the mid-range displacements are also negligible on the order of $n^{-1/4}$. However, the argument required to prove this is more refined. 

\begin{prop}\label{prop:mid_range} 
Fix $\gamma > 0$. For $\delta > 0$ sufficiently small, as $n\rightarrow \infty$,
\[
\p{\|R^\gamma_{n,\delta}\|_\infty > \gamma n^{1/4}}=o(1).
\]
\end{prop} 

To prove this proposition, we will require some further results pertaining to the positions of non-typical displacements in the branching random walk $\bT_n$. More specifically, we will need to study the law of the \emph{number} and \emph{positions} of the vertices $v\in v(\rT_n)\setminus \partial \rT_n$ such that $\|Y^{(v)}\|_\infty >  n^{1/4 - \delta}$, for fixed, small $\delta > 0$. The next lemma pertains to the number of such vertices. 

\begin{lem}\label{lem:few_hairs}
For $\delta > 0$ sufficiently small,
\begin{equation*}
\left|\left\{ v\in v(\rT_n)\setminus \partial \rT_n \text{ such that } \|Y^{(v)}\|_\infty >  n^{1/4 - \delta}\right\}\right| = o_\bP(n^{1/12}).
\end{equation*}
\end{lem}
\begin{proof}
Let $\xi_1,\dots, \xi_n$ be \iid with distribution $\mu$. By [\ref{a2}]  there exists $C > 0$ such that $\p{\|Y_{\xi_1}\|_\infty >  n^{1/4 - \delta}} \le Cn^{-1+4\delta}$. It follows that 
\[ 
A_n := \left|\left\{i\in [n]~:~ \|Y_{\xi_i}\|_\infty >  n^{1/4 - \delta}\right\}\right| \preceq_{st} \mathrm{Bin}\left(n, Cn^{-1 + 4\delta}\right).
\]    
By a Chernoff bound, this implies that for $\delta \in (0, 1/48)$, and $n \ge 1$ sufficiently large, for any $\varepsilon >0$,
\begin{align*}
\p{ A_n > \varepsilon n^{1/12}}&\le \p{\mathrm{Bin}\left(n, Cn^{-1+4\delta}\right) > \varepsilon n^{1/12}} \\
&= \p{\mathrm{Bin}\left(n, Cn^{-1+4\delta}\right) > Cn^{4\delta}\left(1 + \left(\frac{\varepsilon}{C}n^{1/12 - 4\delta} - 1\right)\right)}\\
&= O\left(\exp(-n^{4\delta})\right),
\end{align*}
and so 
\begin{equation*}
\p{ A_n > \varepsilon n^{1/12}~\bigg|~\sum_{i=1}^n\xi_i = n-1} 
= O\left(n^{1/2}\exp(-n^{4\delta})\right)=o(1).\qedhere
\end{equation*}
\end{proof}

We say that two vertices $u,v\in \cU$ are \emph{ancestrally related} if either $u \prec v$ or $v\prec u$. The following lemma establishes that with high probability there are no ancestrally related vertices $u,v\in v(\rT_n)\setminus \partial \rT_n$ such that $\|Y^{(u)}\|_\infty \wedge \|Y^{(v)}\|_\infty >  n^{1/4 - \delta}$. 

\begin{prop}\label{prop:ancestral_rel}
For $\delta > 0$ sufficiently small, as $n\rightarrow \infty$, 
\[\p{\exists u,v\in \bT_n, u \prec v, \text{ such that } \|Y^{(u)}\|_\infty \wedge  \|Y^{(v)}\|_\infty >  n^{1/4 - \delta}} = o(1).\]
\end{prop}

The proof of this proposition relies on an application of the technical lemma, Lemma \ref{lem:no_ganging_up_d_tree}, which we prove in the appendix.

\begin{proof}
We generate $\rT_n$ using the bijective construction $B(\Pi_{D^n})$ described in Section \ref{sec:encodings}. Sample the displacement vectors $(Y_{D_i^n})_{1\le i \le n}$ with $Y_{D_i^n} = (Y_{D_i^n, 1}, \dots, Y_{D_i^n, D_i^n})$, and let 
\[
\cB = \left\{i\in [n]~:~ \|Y_{D_i^n}\|_\infty > n^{1/4 - \delta}\right\}.
\]
Then
\begin{align}\label{eq:ancest_hairs}
& \p{\exists u,v\in \bT_n, u \prec v, \text{ such that } \|Y^{(u)}\|_\infty \wedge \|Y^{(v)}\|_\infty >  n^{1/4 - \delta}} \nonumber\\
& \quad \le \p{\max_{0 \le i \le n} H_n(i) > t\sqrt{n}}+ \p{|\cB| > sn^{1/12}} + \p{\max_{1\le i \le n}D_i^n > Tn^{1/3}}\nonumber\\
& \quad\quad + \p{\left\{\max_{1\le i \le n}D_i^n \le Tn^{1/3},|\cB| \le sn^{1/12}\right\} \cap \left\{\exists i,j \in \cB: i\prec j, d_n(i,j) \le t\sqrt{n}\right\}}
\end{align}
where, for vertices $i,j\in v(\rT_n)$, $d_n(i,j)$ denotes the length of the shortest path between $i$ and $j$ in $B(\Pi_{D^n}) \eqdist \rT_n$. Take $t$ and $T$ large enough so that $\p{\max_{0 \le i \le n} H_n(i) > t\sqrt{n}} < \varepsilon/4$, and $\p{\max_{1\le i \le n}D_i^n > Tn^{1/3}} < \varepsilon/4$. (The latter inequality is possible by \cite[Corollary 19.11]{janson2012simply} since $\E{\xi^3} < \infty$.) By Lemma \ref{lem:few_hairs}, we may take $n$ large enough so that $\p{|\cB| > sn^{1/12}} < \varepsilon/4$. Therefore, for $t,T$ and $n$ sufficiently large, (\ref{eq:ancest_hairs}) is at most 
\[
\frac{3\varepsilon}{4} + \p{\left\{\max_{1\le i \le n}D_i^n \le Tn^{1/3},|\cB| \le sn^{1/12}\right\} \cap \left\{\exists i,j \in \cB: i\prec j, d_n(i,j) \le t\sqrt{n}\right\}}.
\]
Then by Lemma \ref{lem:no_ganging_up_d_tree} with $\bd = D^n$, $K \le sn^{1/12}$, $\Delta \le Tn^{1/3}$, and $b = t\sqrt{n}$, for $n$ sufficiently large, this is at most
\begin{align*}
&\frac{3\varepsilon}{4} + sn^{1/12}\left(1 - \left(1- \frac{sTn^{-7/12}}{1 - n^{-1} - tTn^{-1/6}}\right)^{t\sqrt{n}}\right).
\end{align*}
The result follows by taking $s > 0$ small enough and $n$ large enough so that 
\[
sn^{1/12}\left(1 - \left(1- \frac{sTn^{-7/12}}{1 - n^{-1} - tTn^{-1/6}}\right)^{t\sqrt{n}}\right) < \frac{\varepsilon}{4},
\]
which is possible since
\[
sn^{1/12}\left(1 - \left(1 - \frac{sTn^{-7/12}}{1 - n^{-1} - tTn^{-1/6}}\right)^{t\sqrt{n}}\right) < s^2Tt\frac{1}{1 - n^{-1} - tTn^{-1/6}},
\]
for $n$ large enough because $(1 - x)^r > 1- rx$ for $x < 1$ and $r > 1$.
\end{proof}

\begin{lem}\label{lem:pendant_small} 
Let $v^*(\rT_n) \subseteq v(\rT_n)\setminus \partial \rT_n$ be the set of vertices $v\in v(\rT_n)\setminus \partial \rT_n$ such that $\|Y^{(v)}\|_\infty \le  n^{1/4 - \delta}$ and there exists an ancestor $u\prec v$ with $\|Y^{(u)}\|_\infty >  n^{1/4 - \delta}$. For $\delta > 0$ sufficiently small, $v^*(\rT_n) = o_\bP(n)$.

\end{lem}
\begin{proof}
The result holds if and only if the probability that a uniformly random vertex in $v\in v(\rT_n)$ is ancestrally related to a vertex $u\in v(\rT_n)\setminus \partial \rT_n$ with $\|Y^{(u)}\|_\infty > n^{1/4-\delta}$ is $o_\bP(1)$. By exchangeability, this holds if and only if the probability that vertex $1$ is ancestrally related to a vertex  $u\in v(\rT_n)\setminus \partial \rT_n$ with $\|Y^{(u)}\|_\infty > n^{1/4-\delta}$ is $o_\bP(1)$. To prove this we may adapt the proof of Proposition \ref{prop:ancestral_rel} by including vertex $1$ in the set $\cB$. Then by Lemma \ref{lem:few_hairs}, $|\cB|=o_\bP(n^{1/12})$ still holds, and so the proof carries over verbatim.
\end{proof}

As an immediate consequence of Lemma \ref{lem:pendant_small}, with probability $1-o(1)$ none of the increments of the branching random walk $\bT_{n,\delta}^\gamma$ are ancestrally related with high probability, and Proposition \ref{prop:mid_range} follows.

\subsection{Typical displacements} 
In this subsection we will prove the following proposition.
\begin{prop}\label{prop:tightness_typ}
For all $\gamma > 0$, 
\[
\lim_{k\rightarrow \infty}\limsup_{n\rightarrow\infty}\p{\max_{0 \le i \le k} \sup_{s,t \in [U_{(i)}^{n,k}-1, U_{(i+1)}^{n,k}-1]}|R_{n,\delta}(s) - R_{n,\delta}(t)| > \gamma n^{1/4}} = 0.
\]
\end{prop}

Notice that $R_{n,\delta}$ is equal in distribution to the function encoding the spatial locations of the branching random walk with underlying tree $\rT_n$ and displacements $Y^{n,\delta} = ( Y^{n,\delta,(v)}, v\in v(\rT_n)\setminus \partial \rT_n)$ such that if $v\in v(\rT_n)\setminus \partial \rT_n$ has $k$ children, then $Y^{n,\delta,(v)}$ has the same distribution as 
\[
Y_k^{n,\delta} = (Y_{k,1}^{n,\delta},\dots Y_{k,k}^{n,\delta}) 
:= \begin{cases}
(Y_{k,1},\dots, Y_{k,k}) & \text{if } \max_{1\le j \le k} |Y_{k,j}| \le n^{1/4 - \delta}, \\
(0,\dots,0) & \text{else.}
\end{cases}
\]
This branching random walk is \emph{not} globally centered, and in particular has ``global'' drift $\E{Y^{n,\delta}_{\bar\xi, U_{\bar\xi}}}$. Thus for all $t\in [0,n]$ we have that
\[
R_{n,\delta}(t) \eqdist \breve{R}_{n,\delta}(t) + \E{Y^{n,\delta}_{\bar\xi, U_{\bar\xi}}}\cdot H_n(t),
\]
where $\breve{R}_{n,\delta}:[0,n]\rightarrow \R$ is the function encoding the spatial locations of the globally centered branching random walk $(\rT_n, \breve{Y}^{n,\delta})$ for which, conditionally on $\rT_n$, $\breve{Y}^{n,\delta}= (\breve{Y}^{n,\delta,(v)}, v\in v(\rT_n)\setminus \partial \rT_n)$ is a vector of independent random variables, such that if $v\in v(\rT_n)\setminus \partial \rT_n$ has $k$ children then $\breve{Y}^{n,\delta,(v)}$ has the same distribution as 
\[
\breve{Y}_k^{n,\delta} := Y_k^{n,\delta} - \E{Y_{\bar\xi, U_{\bar\xi}}^{n,\delta}} = Y_k^{n,\delta} - \E{Y_{\bar\xi, U_{\bar\xi}}\I{\|Y_{\bar\xi}\|_\infty \le n^{1/4 - \delta}}}.
\]
Moreover, by the triangle inequality, for all $\gamma > 0$,
\begin{align*}
& \lim_{k\rightarrow \infty}\limsup_{n\rightarrow\infty}\p{\max_{0 \le i \le k} \sup_{s,t \in [U_{(i)}^{n,k}-1, U_{(i+1)}^{n,k}-1]}|R_{n,\delta}(s) - R_{n,\delta}(t)| > \gamma n^{1/4}} \\
& \le \lim_{k\rightarrow \infty}\limsup_{n\rightarrow\infty}\p{\max_{0 \le i \le k} \sup_{s,t \in [U_{(i)}^{n,k}-1, U_{(i+1)}^{n,k}-1]}|\breve R_{n,\delta}(s) - \breve R_{n,\delta}(t)| > \frac{\gamma}{2} n^{1/4}}\\
& \qquad + \limsup_{n\rightarrow \infty}\p{\left|\E{Y_{\bar\xi, U_{\bar\xi}}^{n,\delta}}\right| \cdot \|H_n\|_\infty > \frac{\gamma}{4} n^{1/4}}.
\end{align*} 

\begin{lem}\label{lem:moments_truncation}
It holds that, as $n\to \infty$, 
\[
\left| \E{Y^{n,\delta}_{\bar\xi, U_{\bar\xi}}}\right| =  O(n^{-5/12 +5\delta/3})\quad\text{ and }\quad\V{Y^{n,\delta}_{\bar\xi, U_{\bar\xi}}}\rightarrow \beta^2.
\]
\end{lem}

This result is a special case of Lemma \ref{lem:moments_truncation_move}, which is stated and proved in the appendix. Since $\|H_n\|_\infty = O_\bP(\sqrt{n}),$ Lemma \ref{lem:moments_truncation} implies that for $\delta$ sufficiently small,
\[
\limsup_{n\rightarrow \infty}\p{\left|\E{Y_{\bar\xi, U_{\bar\xi}}^{n,\delta}}\right| \cdot \|H_n\|_\infty > \gamma n^{1/4}} = 0.
\]
It follows that to prove Proposition \ref{prop:tightness_typ}, it suffices to prove that for all $\gamma > 0$, 
\[
\lim_{k\rightarrow \infty}\limsup_{n\rightarrow\infty}\p{\max_{0 \le i \le k} \sup_{s,t \in [U_{(i)}^{n,k}-1, U_{(i+1)}^{n,k}-1]}|\breve R_{n,\delta}(s) - \breve R_{n,\delta}(t)| > \gamma n^{1/4}} = 0.
\]

As discussed above, we need to deal with the maximum modulus of the difference in spatial location (for the branching random walk $\mathbf{\rT}_{n,\delta}$) between the parent of the root of a pendant subtree and a vertex of that subtree. There are 
\[
c(\rT_n^{k}):= \sum_{v\in V(\rT_n^{k})}(c(v, T_n) - 1) + 1
\]
edges in $\rT_n$ with one endpoint in $\rT_n^{k}$ and another in $\rT_n\setminus \rT_n^{k}$. Conditionally on $\rT_n^{k}$, if we remove all such edges we obtain a Bienaymé($\mu$) forest conditioned to have $n-|V(\rT^{(k)}_n)|$ vertices and $c(\rT_n^{k})$ trees. We denote this forest by $\mathrm{F}_n^{k} = (\rT_{n,j}^{k})_{j \ge 1}$, where the trees are listed in decreasing order of size, and $|\rT_{n,j}^{k}| = 0$ for $j > c(\rT_n^{k})$. Write $\|\breve R_{n,\delta}(\rT_{n,j}^k)\|_{\infty}$ for maximum modulus of the difference in spatial location between the root and any other vertex of $\rT_{n,j}^k$. 

The trees $(\rT_{n,j}^k)_{j \ge 1}$ are independent Bienaymé trees, conditioned on their sizes. Therefore, conditionally on $\mathrm{F}_n^k$, we have $\|\breve R_{n,\delta}(\rT_{n,j}^k)\|_\infty \eqdist \|\breve R_{|\rT^k_{n,j}|,\delta}\|_\infty$. Moreover, displacements on the tree $\rT_{n,j}^k$ (from the branching random walk $\bT_{n,\delta}$) depend on those in other parts of $\rT_n$ only through the displacement  $\breve Z_j^{n,\delta}$ of the root of $\rT_{n,j}^k$ away from its parent in $\rT_n^k$; see Figure~\ref{fig:tree_tight}.

\begin{figure}
    \centering
    \includegraphics[scale = 0.8]{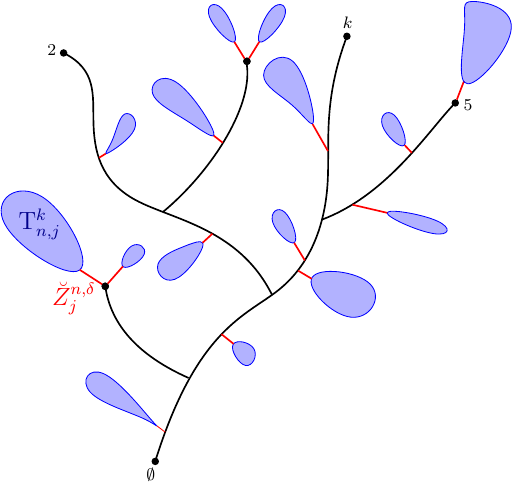}
    \caption{In black, the tree $\rT_n^{k}$. In blue, the forest $\mathrm{F}_n^{k} = (\rT_{n,j}^{k})_{j \ge 1}$. The root of tree $\rT_{n,j}^{k}$ is displaced $\breve Z^{n,\delta}_j$ away from its parent in $\rT_n^{k}$.}
    \label{fig:tree_tight}
\end{figure}

It follows that
\begin{align*}
& \max_{0 \le i \le k} \sup_{s,t \in [U_{(i)}^{n,k}-1, U_{(i+1)}^{n,k}-1]}|\breve R_{n,\delta}(s) - \breve R_{n,\delta}(t)| \\
& \qquad \le \max_{0 \le i \le k} \Upsilon_i^{n,k} + 2 \max_{1 \le j \le c(\mathrm{F}_n^{k})}\left(\left\|\breve R_{n,\delta}(\rT_{n,j}^k)\right\|_\infty + | \breve Z^{n,\delta}_j|\right).
\end{align*}
Consequently, using (\ref{eq:displ}), in order to prove Proposition~\ref{prop:tightness_typ}, it is sufficient to prove that for $\gamma > 0$,
\[
\lim_{k \to \infty} \limsup_{n \to \infty} \p{\max_{1\le j \le c(\mathrm{F}_n^{k})}\left(\left\|\breve R_{n,\delta}(\rT_{n,j}^k)\right\|_\infty + | \breve Z^{n,\delta}_j|\right) \ge \gamma n^{1/4}} = 0.
\]
The proof requires two key ingredients: (1) a scaling limit for the sizes of the trees in $\mathrm{F}_n^{k}$; (2) quantitative control on the tail of $\|\breve R_{n,\delta}\|_\infty$. We begin by establishing (1).

\begin{prop}\label{prop:comp_trees}
As $n\rightarrow \infty$,
\begin{equation}\label{eq:conv_components}
\frac{c(\rT_n^{k})}{\sigma\sqrt{n}} \convdist J_k,
\end{equation}
where $J_k$ is $\mathrm{Gamma}(k,1/2)$ distributed. Jointly with this convergence, we have
\begin{equation} \label{eq:conv_sizes}
\frac{\sigma}{n - |V(\rT_n^{k})|}(|\rT_{n,j}^{k}|, j \ge 1) \convdist (|\gamma^{k}_j|, j \ge 1),
\end{equation}
where, conditionally on $J_k$, $(|\gamma^{k}_j|, j \ge 1)$ lists the sizes of the excursions above the past minimum of a Brownian motion stopped on first hitting $-J_k$, listed in decreasing order.
\end{prop}

\begin{proof}
By Skorokhod's representation theorem we may work in a probability space where the convergence in Proposition \ref{prop:conv_cuts_attach} holds almost surely so that in particular as $n\rightarrow \infty$, $\sigma n^{-1/2} J_k^n \convas J_k$. 

Let $T > 0$ and recall the event \[\cG_{n,k}(T) = \left\{\left\{\widehat V_1(\Pi_{D^n}),\dots, \widehat V_{\lfloor T\sqrt{n}\rfloor}(\Pi_{D^n})\right\}\cap \{1,\dots,k\} = \emptyset, N_n \ge \lfloor T\sqrt{n}\rfloor\right\}\] from Lemma \ref{lem:conv_good_event}. 
On $\cG_{n,k}(T)\cap \{J_k^n \le \lfloor T\sqrt{n}\rfloor\}$, the tree $\rT_n^k$ is precisely the subtree of $\rT_n$ spanned by the root and the vertices $1,\dots,k$. Therefore, on $\cG_{n,k}(T)\cap \{J_k^n \le \lfloor T\sqrt{n}\rfloor\}$,
\[
\frac{V(\rT_n^k)}{\sqrt{n}} = \frac{J_k^n}{\sqrt{n}}.
\]
Since $T>0$ is arbitrary and $\sigma n^{-1/2}J_k^n \convas J_k$ we obtain that $n - |V(\rT_n^{k})| = n - o_\bP(n)$. 
Hence, we are essentially considering a forest of Bienaymé trees conditioned to have $n$ vertices. We now need to show that the number of trees in such a forest is $\sim \sigma \sqrt{n}J_k$. We note that on the event $\cG_{n,k}(T) \cap \{J_k^n \le T\sqrt{n}\}$, there are $\sum_{i=1}^{J_k^n - k}(\widehat D_i^n - 1) + \sum_{i=1}^k D_{J_i^n}^n$ subtrees of $\rT_n$ whose roots have a parent in $\rT_n^k$, and $(k-1)$ branch points in $\rT_n^k$. Therefore, for $s \ge 0$,
\begin{align}\label{eq:num_comps}
&\p{\frac{c(\rT_n^k)}{\sigma \sqrt{n}} \ge s, ~\cG_{n,k}(T), ~J_k^n \le T\sqrt{n}}\nonumber\\ 
&\quad= \p{\frac{1}{\sigma\sqrt{n}}\left(\sum_{i=1}^{J_k^n - k}(\widehat D^n_i - 1) + \sum_{i=1}^k D_{J_i^n}^n - (k-1)\right) \ge s,~\cG_{n,k}(T),~J_k^n \le T\sqrt{n}}.
\end{align}
Since $\sigma n^{-1/2}J_k^n \convas J_k$, by Proposition \ref{prop:conv_degrees_disps},
\begin{align*}\frac{1}{\sigma\sqrt{n}}\sum_{i=1}^{J_k^n - k}(\widehat D_i - 1) \convdist J_k.
\end{align*}
Combining this with Lemma \ref{lem:conv_good_event} and Proposition \ref{prop:conv_cuts_attach}, we obtain that (\ref{eq:num_comps}) converges to 
\[
\p{J_k > s, J_k \le \sigma T}.
\]
Then (\ref{eq:conv_components}) follows as $T>0$ is arbitrary. The scaling limit in (\ref{eq:conv_sizes}) now follows from \cite[Proposition 1.4]{tao_lei} and \cite[Lemma 11]{prescribed_degrees}.
\end{proof}

The control on $\|\breve R_{n,\delta}\|_\infty$ needed to prove Proposition \ref{prop:tightness_typ} is given by the next proposition. 

\begin{prop}\label{lem:max_tight}
There exists $A > 0$ such that for all $\gamma > 0$, $\delta\in (0,1/4)$, and $n \ge 1$,
\[
\p{\|\breve R_{n,\delta}\|_\infty > \gamma n^{1/4}} \le \frac{A}{\gamma^8}.
\]
\end{prop}

The proof of this proposition is long and somewhat technical, so we postpone it until Section \ref{sec:technical}. 

\begin{proof}[Proof of Proposition \ref{prop:tightness_typ} assuming Proposition \ref{lem:max_tight}]
By Skorokhod's representation theorem, we may assume that we are working on a probability space where the convergence in Proposition \ref{prop:conv_cuts_attach} is almost sure. In particular, $\sigma n^{-1/2}J_k^n \convas J_k$ as $n\rightarrow \infty.$ 

As argued above, it remains to show that, for $\gamma > 0$,
\[
\lim_{k \to \infty} \limsup_{n \to \infty} \p{\max_{1\le j \le c(\mathrm{F}_n^{k})}\left(\left\|\breve R_{n,\delta}(\rT_{n,j}^k)\right\|_\infty + | \breve Z^{n,\delta}_j|\right) \ge \gamma n^{1/4}} = 0.
\]
Since $(\rT_{n,j}^k)_{1 \le j \le c(\mathrm{F}_n^k)}$ are independent Bienaym\'e($\mu$) trees conditionally on their sizes, we obtain
\begin{align*}
& \p{\max_{1\le j \le c(\mathrm{F}_n^{k})}\left(\left\|\breve R_{n,\delta}(\rT_{n,j}^k)\right\|_\infty + | \breve Z^{n,\delta}_j|\right) \ge \gamma n^{1/4}} 
    \notag \\
& = \E{\p{\max_{1\le j \le c(\mathrm{F}_n^{k})}\left(\left\|\breve R_{|\rT_{n,j}^{k}|,\delta}\right\|_\infty + | \breve Z^{n,\delta}_j|\right) \ge \gamma n^{1/4}~|~ \mathrm{F}^{k}_n, (\breve Z^{n,\delta}_{j})_{j \ge 1}}}\nonumber\\
&\le \E{\sum_{j = 1}^{c(\mathrm{F}_n^{k})}\p{\left\|\breve R_{|\rT_{n,j}^{k}|,\delta}\right\|_\infty \ge n^{1/4}(\gamma - |\breve Z^{n,\delta}_j|/n^{1/4})~|~  \mathrm{F}^{k}_n, (\breve Z^{n,\delta}_{j})_{j \ge 1}}} \\
& \le \E{\sum_{j = 1}^{\infty}\p{\|\breve R_{|\rT^{k}_{n,j}|}\|_\infty \ge \frac{\gamma n^{1/4}}{2}~\bigg|~\mathrm{F}_n^{k}}},
\end{align*}
for all $n$ sufficiently large, since $|\breve Z_j^{n,\delta}| \le n^{1/4 - \delta}$ for all $1 \le j \le c(\rT_k^n)$ and all $n \ge 1$. Applying Proposition \ref{lem:max_tight} to each of the conditional probabilities in the above sum, we obtain that the right-hand side is at most
\begin{equation}\label{eq:pre_lim_n}
\frac{2^8A}{\gamma^8} \E{\sum_{j = 1}^\infty \left(\frac{|\rT^{k}_{n,j}|}{n}\right)^{2}}
=\frac{2^8A}{\gamma^8}\E{\left(\frac{n - |V(\rT_n^{k})|}{n }\right)^2\cdot\frac{\widehat{| \rT^{k}|}}{n- |V(\rT_n^{k})|}},
\end{equation}
where $\widehat{| \rT^{k}|}$ is a size-biased pick from $( |\rT^{k}_{n,j}|)_{j \ge 1}.$ Clearly,
\[
\frac{n - |V(\rT_n^k)|}{n} \le 1.
\]
By Proposition \ref{prop:comp_trees}, as $n\rightarrow \infty$, $\widehat{| \rT^{k}|}/(n-|V(\rT_n^{k})|) \convdist \sigma^{-1}\widehat{|\gamma^{k}|}$ where $\widehat{|\gamma^{k}|}$ is a size-biased pick from $(|\gamma^{k}_j|, j \ge 1)$. By \cite[Section 8.1]{pitman2003poisson}, conditionally on $J_k$, 
\[
\widehat{|\gamma^{k}|}\eqdist \frac{B^2}{J_k^2 + B^2},
\]
where $B$ is a $N(0,1)$ random variable independent of $J_k$. Combining this with (\ref{eq:pre_lim_n}), we obtain that
\[
\limsup_{n\rightarrow\infty}  \p{\max_{1\le j \le c(\mathrm{F}_n^{k})}\left(\left\|\breve R_{n,\delta}(\rT_{n,j}^k)\right\|_\infty + | \breve Z^{n,\delta}_j|\right) \ge \gamma n^{1/4}} \le \frac{2^8A}{\sigma \gamma^8}\E{\frac{B^2}{J_k^2 + B^2}}.
\]
As $k \rightarrow \infty$, $J_k \convprob \infty$. Therefore, by bounded convergence,
\[
\lim_{k\rightarrow \infty}\limsup_{n\rightarrow \infty} \p{\max_{1\le j \le c(\mathrm{F}_n^{k})}\left(\left\|\breve R_{n,\delta}(\rT_{n,j}^k)\right\|_\infty + | \breve Z^{n,\delta}_j|\right) \ge \gamma n^{1/4}} = 0.\qedhere
\]
\end{proof}

\medskip

Assuming Proposition~\ref{lem:max_tight}, Proposition \ref{prop:tightness} now follows from (\ref{eq:triangleineq}) by taking $\delta \in (0,1/4)$ sufficiently small so that Proposition~\ref{prop:mid_range} holds, and combining that with Propositions~\ref{prop:large_tight} and \ref{prop:tightness_typ}.

\section{The maximum spatial location: proof of Proposition~\ref{lem:max_tight}}\label{sec:technical}

We assume throughout this section that $\mu$ is critical and has finite variance $\sigma^2\in (0,\infty)$, and that [\ref{a1}] and [\ref{a2}] hold.

For $n \ge 1$, let $\Lambda^{(n)}:= (\Lambda_1^{(n)},\Lambda_2^{(n)},\dots, \Lambda^{(n)}_{\widehat D^n_1})$ be the sizes of the subtrees of the root of $\rT_n$, so that $\Lambda_i^{(n)}$ is the size of the subtree rooted at the $i$-th child of the root. We will make extensive use of the fact that, conditionally on $\widehat{D}_1^n$, these are exchangeable random variables (i.e.\ their distribution is invariant under permutations of the labels). To prove Proposition \ref{lem:max_tight} we will make extensive use of the following consequence of Lemma 25 of Haas and Miermont \cite{haas2012scaling} which, roughly speaking, tells us that typically only one subtree of a child of the root is macroscopic and, moreover, the probability of a non-trivial macroscopic split at the root is on the order of $n^{-1/2}$. 

\begin{lem}[Lemma 25 of \cite{haas2012scaling}] \label{lem:haas} 
It holds that
\begin{equation}\label{eq:haas_lem}
\E{1 - \sum_{i=1}^{\widehat D_1^n}\left(\frac{\Lambda_i^{(n)}}{n}\right)^2} = \Theta(n^{-1/2}).
\end{equation}
\end{lem}

In the proof of Proposition \ref{lem:max_tight}, we encounter terms directly related to the global centering and global finite variance conditions, respectively. The latter is more challenging to control, and is the reason for the third moment condition on the offspring distribution. These terms, and the control we will require on them, are stated in the following technical lemma. Recall the definition of $\widehat D^m$, the size-biased ordering of $D^m = (D_1^m, \dots, D_m^m)$, \iid samples random distribution $\mu$ conditioned to sum to $m-1$. 

The proof of Proposition \ref{lem:max_tight} is inductive, and requires that we control the maximum of~$\breve R_{n,\delta}^k$ when restricted to subtrees of $\rT_n^k$. We henceforth use $m \ge 1$ to denote the number of vertices in the underlying tree, $\rT_m$, and $n \ge 1$ to denote the truncation threshold $n^{1/4 - \delta}$ on the displacements. More specifically, in this section, we will consider branching random walks on $\rT_m$ with displacements $\breve Y_k^{n,\delta}$, $k \ge 1$.

\begin{lem} \label{lem:technical} 
Let $n \ge 1$ and $m \le n$. There exists $B > 0$ such that
\begin{equation} \label{eq:var}
    \E{\sum_{i=1}^{\widehat{D}_1^m} \left(\frac{\Lambda_i^{(m)}}{m}\right)^{2}(\breve Y^{n,\delta}_{\widehat{D}_1^m, i})^2} \le B.
\end{equation}
If in addition $(\mu, \nu)$ satisfies [\ref{a1}] and [\ref{a2}], then there exists $B'> 0$ such that
\begin{equation} \label{eq:mean}
    \left|\E{\sum_{i=1}^{\widehat{D}_1^m} \left(\frac{\Lambda_i^{(m)}}{m}\right)^{2}\breve Y^{n,\delta}_{\widehat{D}_1^m, i}}\right|  \le   \frac{B'n^{1/4 - \delta}}{\sqrt{m}} +\frac{B'}{m^{1/4}}.
\end{equation}
\end{lem}

Condition [\ref{a1}] pertains to the mean and variance of the displacement of a uniform child of a vertex with a size-biased number of offspring, $Y_{\bar\xi, U_{\bar\xi}}$. The displacement from the root of $\rT_m$ to a uniform child is distributed as $Y_{\widehat{D}_1^m, U_{\widehat{D}_1^m}}$ and we have $\widehat D_1^m \convdist \bar\xi$ as $m\rightarrow \infty$. However, in order to use the global centering and global finite variance conditions in the proof of Lemma \ref{lem:technical}, we need something stronger, namely an explicit rate of decay for the total variation distance between the laws of $\bar\xi$ and $\widehat{D}_1^m$. This is provided by the next lemma.

\begin{lem} \label{lem:tv_distance} 
As $m\rightarrow \infty$,
\[
d_{\mathrm{TV}}(\widehat D_1^m, \bar\xi) = \frac{1}{2}\sum_{k=1}^\infty \left|\p{\widehat{D}_1^m = k } - \p{\bar\xi = k}\right| = o(m^{-1/2}).
\]
\end{lem}

\begin{proof}
Let $k \ge 1$, and let $(S_m)_{m\ge 1}$ be a random walk with \iid $\mu$-distributed increments. Recall from (\ref{eq:dn_dist}) that 
\[\p{\widehat{D}_1^m = k} = \left(\frac{m}{m-1}\right)\frac{\p{S_{m-1} = m-1-k}}{\p{S_{m} = m-1}}\p{\bar\xi = k}.\]
Since $\E{\xi} = 1$  and $\E{\xi^3}<\infty$, by Theorem \ref{thm:gen_clt},
\begin{align*}
& \sqrt{2\pi (m-1)} \sigma \p{S_{m-1} = m-1-k} \\
& \quad= e^{-k^2/(2\sigma^2 (m-1))}\left(1 + \frac{1}{\sqrt{m-1}}\frac{\gamma_3}{6\sigma^3}\left(\frac{k^3}{\sigma^3 (m-1)^{3/2}} - \frac{3k}{\sigma \sqrt{m-1}}\right) \right) + o(m^{-1/2}).
\end{align*}
If $k = O(m^{1/4}),$ 
\[
\left|\frac{k^3}{\sigma^3 (m-1)^{3/2}} - \frac{3k}{\sigma \sqrt{m-1}}\right| = O(m^{-1/4}),
\]
and
\[
e^{-k^2/(2\sigma^2 (m-1))} = 1 - \frac{k^2}{2\sigma^2 (m-1)} + O(m^{-1}).
\]
Hence for $k = O(m^{1/4}),$
\[
\sqrt{2\pi (m-1)} \sigma \p{S_{m-1} =m-1 -k} = 1 - \frac{k^2}{2\sigma^2 (m-1)} + o(m^{-1/2}).
\]
It follows that for $k = O(m^{1/4}),$
\[
\left(\hspace{-2pt}\frac{m}{m-1}\hspace{-2pt}\right)\hspace{-2pt}\frac{\p{S_{m-1} \hspace{-1pt}=\hspace{-1pt}m\hspace{-1pt}-\hspace{-1pt}1\hspace{-1pt} -\hspace{-1pt}k}}{\p{S_{m} \hspace{-1pt}=\hspace{-1pt} m-1}} = \frac{1 \hspace{-1pt}-\hspace{-1pt} \frac{k^2}{2\sigma^2(m\hspace{-1pt}-\hspace{-1pt}1)} \hspace{-1pt}+\hspace{-1pt} o(m^{-1/2})}{1 + o(m^{-1/2})} = 1 \hspace{-1pt}-\hspace{-1pt} \frac{k^2}{2\sigma^2 (m\hspace{-1pt}-\hspace{-1pt}1)} + o(m^{-1/2}),
\]
and, consequently,
\begin{equation*}\label{eq:finer_control}
\p{\widehat{D}_1^m = k} = \left(1 - \frac{k^2}{2\sigma^2 (m-1)} + o(m^{-1/2})\right)\p{\bar\xi = k}.
\end{equation*}
Therefore, 
\begin{align*}
&\sum_{k=1}^{\infty}\left|\p{\widehat{D}_1^m = k} - \p{\bar\xi = k}\right|\\
&= \sum_{k=1}^{\lfloor m^{1/4}\rfloor}\left|\p{\widehat{D}_1^m = k} - \p{\bar\xi = k}\right| + \sum_{k=\lfloor m^{1/4}\rfloor +1}^{\infty}\left|\p{\widehat{D}_1^m = k} - \p{\bar\xi = k}\right|\\
&=\sum_{k=1}^{\lfloor m^{1/4}\rfloor}\hspace{-5pt}\left(\frac{k^2}{2\sigma^2 (m-1)} + o(m^{-1/2})\right)\p{\bar\xi = k} + \sum_{k=\lfloor m^{1/4}\rfloor +1}^{\infty}\left|\p{\widehat{D}_1^m = k} - \p{\bar\xi = k}\right|\\
&\le \frac{\E{\xi^3}}{2\sigma^2 (m-1)} + o(m^{-1/2}) + \sum_{k=\lfloor m^{1/4}\rfloor +1}^{\infty}\left(\p{\widehat{D}_1^m = k} + \p{\bar\xi = k}\right)\\
&\le o(m^{-1/2}) + (c+1)\cdot \p{\bar\xi > m^{1/4}},
\end{align*}
where the final inequality follows since $\p{\widehat D_1^m = k} \le c \p{\bar\xi = k}$ for all $k\in [m]$. Since $\E{\xi^3} < \infty$, $\bar\xi$ has a finite second moment. Therefore, $\p{\bar\xi > k} = o(k^{-2})$ as $k \rightarrow \infty$ and so $\p{\bar\xi > m^{1/4}} = o(m^{-1/2})$. The result follows.
\end{proof}

The terms (\ref{eq:var}) and (\ref{eq:mean}) relate to the variance and mean (respectively) of the displacement of a uniform child of the root in branching random walk $(\rT_m, \breve Y^{n,\delta})$. Since this branching random walk is globally centered, it is reasonable to expect that the mean will be small and that the second moment will be bounded. A key technical lemma follows.
\begin{lem}\label{lem:mean_no_perm}
There exists a constant $C > 0$ such that for $m \le n$,
\[
\left|\E{\frac{1}{\widehat{D}_1^m}\sum_{i=1}^{\widehat{D}_1^m}\breve{Y}^{n,\delta}_{\widehat{D}_1^m, i}}\right| \le \frac{Cn^{1/4 - \delta}}{\sqrt{m}}.
\]
\end{lem}

\begin{proof} 
Let $(\widehat D_1^m, \bar\xi)$ be a coupling of the degree of the root of $\rT_m$ and the size-biased distribution of $\mu$. We consider the events $\{\bar\xi = \widehat{D}_1^m\}$ and $\{\bar \xi \neq \widehat{D}_1^m\}$ separately:
\begin{align}\label{eq:1_mean_no_perm}
    \left|\E{\frac{1}{\widehat D_1^m}\sum_{i=1}^{\widehat D_1^m} \breve{Y}^{n,\delta}_{\widehat D_1^m, i}}\right|
    &\le \left|\E{\frac{1}{\bar\xi}\sum_{i=1}^{\bar\xi} \breve{Y}^{n,\delta}_{\bar\xi, i}\I{\bar\xi = \widehat D_1^m}}\right|+ \E{\frac{1}{\widehat D_1^m}\sum_{i=1}^{\widehat D_1^m} |\breve{Y}^{n,\delta}_{\widehat D_1^m, i}|\I{ \bar\xi \neq \widehat D_1^m}}\nonumber\\
    &=\left|\E{\frac{1}{\bar\xi}\sum_{i=1}^{\bar\xi} \breve{Y}^{n,\delta}_{\bar\xi, i}\I{\bar\xi \neq \widehat D_1^m}}\right|+ \E{\frac{1}{\widehat D_1^m}\sum_{i=1}^{\widehat D_1^m} |\breve{Y}^{n,\delta}_{\widehat D_1^m, i}|\I{ \bar\xi \neq \widehat D_1^m}},
\end{align}
where the equality holds since 
\[ 
\E{\frac{1}{\bar\xi}\sum_{i=1}^{\bar\xi} \breve Y_{\bar\xi, i}^{n,\delta}} = \E{\breve Y_{\bar\xi, U_{\bar\xi}}^{n,\delta}}= 0.
\] 
Since $|\breve{Y}^{n,\delta}_{k,j}| \le 2n^{1/4-\delta}$ for all $k \ge 1$  and  $ j \in [k]$, it follows that (\ref{eq:1_mean_no_perm}) is at most \[4n^{1/4 - \delta}\p{\bar\xi \neq \widehat D_1^m}.\] The result follows from Lemma \ref{lem:tv_distance} by taking an optimal coupling of $(\widehat{D}_1^m, \bar\xi).$
\end{proof}

We now proceed to prove Lemma \ref{lem:technical}.

\begin{proof}[Proof of Lemma \ref{lem:technical}]
We first prove (\ref{eq:var}). Note that by exchangeability of $(\Delta_1^{(m)}, \ldots, $ $\Delta_{\widehat{D}_m^1}^{(m)})$ and linearity of conditional expectation 
\begin{align}\label{eq:need_var}
    \E{\sum_{i=1}^{\widehat{D}_1^m}\left(\frac{\Lambda_i^{(m)}}{m}\right)^{2}(\breve Y^{n,\delta}_{\widehat{D}_m^1, i})^2} &= \E{\left(\frac{1}{\widehat D_1^m}\sum_{i=1}^{\widehat D_1^m}(\breve{Y}^{n,\delta}_{\widehat D_1^m,i})^2\right)\left(\sum_{i=1}^{\widehat D_1^m}\left(\frac{\Lambda_i^{(m)}}{m}\right)^{2}\right)}\nonumber\\
    &\le \E{\frac{1}{\widehat D_1^m}\sum_{i=1}^{\widehat D_1^m}(\breve{Y}^{n,\delta}_{\widehat D_1^m,i})^2}\nonumber\\
    &\le c \E{\frac{1}{\bar\xi}\sum_{i=1}^{\bar\xi}(\breve{Y}^{n,\delta}_{\bar\xi,i})^2}\nonumber\\
    &= c\E{(\breve{Y}_{\bar\xi, U_{\bar\xi}}^{n,\delta})^2},
\end{align}
where the second inequality follows since $\p{\widehat D_1^m = k} \le c\p{\bar\xi = k}$. By Lemma \ref{lem:moments_truncation}, (\ref{eq:need_var}) tends to $\beta^2$ as $n\rightarrow \infty$ and hence (\ref{eq:var}) holds.

We now proceed to proving (\ref{eq:mean}). By linearity and the triangle inequality we have 
\begin{align}\label{eq:1_technical}
    &\left|\E{\sum_{i=1}^{\widehat{D}_1^m}\left(\frac{\Lambda_i^{(m)}}{m}\right)^{2}\breve Y^{n,\delta}_{\widehat{D}_m^1, i}}\right|\nonumber\\
    & \qquad \le \left|\E{\frac{1}{\widehat D_1^m}\sum_{i=1}^{\widehat D_1^m}\breve{Y}^{n,\delta}_{\widehat D_1^m, i}}\right| + \left|\E{\left(\frac{1}{\widehat D_1^m}\sum_{i=1}^{\widehat D_1^m}\breve{Y}^{n,\delta}_{\widehat D_1^m,i}\right)\left(1 - \sum_{i=1}^{\widehat D_1^m}\left(\frac{\Lambda_i^{(m)}}{m}\right)^{2}\right)}\right|.
\end{align}
By Lemma \ref{lem:mean_no_perm}, (\ref{eq:1_technical}) is at most
\[ 
\frac{Cn^{1/4 - \delta}}{\sqrt{m}} + \left|\E{\left(\frac{1}{\widehat D_1^m}\sum_{i=1}^{\widehat D_1^m}\breve{Y}^{n,\delta}_{\widehat D_1^m,i}\right)\left(1 - \sum_{i=1}^{\widehat D_1^m}\left(\frac{\Lambda_i^{(m)}}{m}\right)^{2}\right)}\right|.
\]
Applying the Cauchy--Schwarz inequality to the second term yields an upper bound of 
\begin{align}\label{eq:2_technical}
& \frac{Cn^{1/4 - \delta}}{\sqrt{m}} + \E{\left(\frac{1}{\widehat D_1^m}\sum_{i=1}^{\widehat D_1^m}|\breve{Y}^{n,\delta}_{\widehat D_1^m,i}|\right)^2}^{1/2}\E{\left(1 - \sum_{i=1}^{\widehat D_1^m}\left(\frac{\Lambda_i^{(m)}}{m}\right)^{2}\right)^2}^{1/2}\nonumber\\
& \le \frac{Cn^{1/4 - \delta}}{\sqrt{m}}+ \E{\frac{1}{\widehat D_1^m}\sum_{i=1}^{\widehat D_1^m}(\breve{Y}^{n,\delta}_{\widehat D_1^m,i})^2}^{1/2}\E{\left(1 - \sum_{i=1}^{\widehat D_1^m}\left(\frac{\Lambda_i^{(m)}}{m}\right)^{2}\right)^2}^{1/2},
\end{align}
where we have again used the Cauchy--Schwarz inequality on the sum inside the expectation to get the second inequality. Since $\p{\widehat D_1^m = k} \le c\p{\bar\xi = k}$, by the same methods used in (\ref{eq:var}), there exists $c'> 0$ such that (\ref{eq:2_technical}) is at most 
\[
\frac{Cn^{1/4 - \delta}}{\sqrt{m}}+ c' \E{\left(1 - \sum_{i=1}^{\widehat D_1^m}\left(\frac{\Lambda_i^{(m)}}{m}\right)^{2}\right)^2}^{1/2}.
\]
Lastly, since $x^2 \le x$ for all $x\in[0,1]$, we obtain the bound
\begin{equation*}
\left|\E{\sum_{i=1}^{\widehat{D}_1^m}\left(\frac{\Lambda_i^{(m)}}{m}\right)^{2}\breve Y^{n,\delta}_{\widehat{D}_1^m,i}}\right| \le  \frac{Cn^{1/4 - \delta}}{\sqrt{m}} + c' \E{1 - \sum_{i=1}^{\widehat D_1^m}\left(\frac{\Lambda_i^{(m)}}{m}\right)^{2}}^{1/2},
\end{equation*}
and the result follows by Lemma \ref{lem:haas}.
\end{proof}

We now present the proof of Proposition \ref{lem:max_tight}.

\begin{proof}[Proof of Proposition \ref{lem:max_tight}]
For $n \ge 1$ and $m \le n$, let $ \breve R_{m,n,\delta}$ be the spatial process of a branching random walk $\breve\bT_{m,n} = (\rT_m, \breve{Y}^{n,\delta})$ where the displacement vector of a vertex $v\in v(\rT_m)\setminus \partial \rT_m$ with $k$ children is distributed as 
\[
\breve Y_k^{n,\delta} = Y_k^{n,\delta} - \E{Y_{\bar\xi,U_{\bar\xi}}^{n,\delta}}.
\]
Furthermore, let
\[
\breve R_{m,n,\delta}^{+} := \max\left\{0, \max_{0 \le i \le m} \breve R_{m,n,\delta}\right\},
\]
and 
\[ 
\breve R_{m,n,\delta}^{-} := -\min\left\{0, \min_{0 \le i \le m} \breve R_{m,n,\delta}\right\}.
\]

It suffices to prove that there exists $A > 0$ such that for all $m \ge 0$, all $n \ge m$ and all $\gamma > 0$,
\[
\p{ \breve R_{m, n,\delta}^+ > \gamma n^{1/4}} \le \frac{A}{\gamma^8} \quad \text{and} \quad \p{\breve R_{m,n,\delta}^- > \gamma n^{1/4}} \le \frac{A}{\gamma^8},
\]
since Proposition \ref{lem:max_tight} then follows by taking $n = m$. We only prove the tail bound for $\breve R_{m,n,\delta}^+$, as the bound for $\breve R_{m,n,\delta}^-$ then follows by symmetry.
 
Notice that $\breve R_{1,n,\delta}^{+} = 0$ for all $n \ge 0$,  and so the claim holds trivially if $m = 1$. Moreover, at the cost of taking $A > 0$ larger, it is sufficient to prove the result for $\gamma > 0$ sufficiently large. We will proceed by induction on $m \ge 2$, and hence assume that for $1 \le k \le m-1$ and $\gamma > 0$, 
\[
\p{ \breve R_{k,n,\delta}^+ > \gamma n^{1/4}} \le \frac{A}{\gamma^8},
\]
for all $n \ge k$. 
 
Observe that conditionally on $\widehat D^{m}_1$ and $\Lambda^{(m)}$,
\[
\breve R_{m,n,\delta}^+ \eqdist \max\left\{0, \max_{1\le i\le \widehat D_1^{m}} \left\{ \breve R_{\Lambda_i^{(m)}, n, \delta}^+ + \breve Y_{\widehat{D}_1^m, i}^{n,\delta}\right\}\right\}.
\]
For the rest of the proof, we write $\breve Y_i^{n, \delta}$ in place of $\breve Y_{\widehat{D}_1^m,i}^{n, \delta}$ to ease the notation.

Take $u_0\in (0,1)$ such that for all $0 < u < u_0$, $(1-u)^{-8} \le 1 + 8 u + 72 u^2$. Then, taking $\gamma > 2/u_0$ (recall this is possible at the cost of taking $A > 0$ larger), it follows that 
\begin{align*}
    &\p{ \breve R_{m,n,\delta}^+ \le \gamma n^{1/4}}\\
    &= \E{\p{\max_{1\le i \le \widehat D_1^m}\left\{ \breve R_{\Lambda^{(m)}, n, \delta}^+ + \breve Y_i^{n,\delta}\right\} \le \gamma n^{1/4}\bigg|\widehat D_1^m, \Lambda^{(m)}}}\\
    &= \E{\prod_{i=1}^{\widehat D_1^m}\p{ \breve R_{\Lambda_i^{(m)}, n, \delta}^+ \le \gamma n^{1/4} - \breve Y_i^{n,\delta}\bigg| \widehat D_1^m, \Lambda^{(m)}, \breve Y_i^{n,\delta}}},
\end{align*}
where in the second equality we have used the tower law and the branching property. We will bound the right-hand side of the above equality by applying induction to each term in the product. More specifically, taking $n = k = \Lambda_i^{(m)}$ and for the $i$-th term of the product, by the induction hypothesis, we obtain
\begin{align*}
&\E{\prod_{i=1}^{\widehat D_1^m}\p{\breve R_{\Lambda_i^{(m)}, n, \delta} \le \gamma n^{1/4} - \breve Y_i^{n,\delta}\bigg| \widehat D_1^m, \Lambda^{(m)}, \breve Y_i^{n,\delta}}} \\
&\quad= \E{\prod_{i=1}^{\widehat D_1^m}\p{\breve R_{\Lambda_i^{(m)}, n, \delta} \le \left(\frac{\gamma n^{1/4} - \breve Y_i^{n,\delta}}{(\Lambda_i^{(m)})^{1/4}}\right)(\Lambda_i^{(m)})^{1/4}\bigg| \widehat D_1^m, \Lambda^{(m)}, \breve Y_i^{n,\delta}}} \\
&\quad\ge\E{\prod_{i=1}^{\widehat D_1^m} \left(1 - \frac{A (\Lambda_i^{(m)})^{2}}{(\gamma n^{1/4} - \breve Y_i^{n,\delta})^8}\right)_+}.
\end{align*}
Furthermore, since $\prod_{i=1}^k(1-x_i)_+\ge 1 - \sum_{i=1}^kx_i$ for any non-negative sequence $(x_i)_{i \ge 1}$, we may lower bound the above as
\begin{align*}
    \E{\prod_{i=1}^{\widehat D_1^m} \left(1 - \frac{A (\Lambda_i^{(m)})^{2}}{(\gamma n^{1/4} - \breve Y_i^{n,\delta})^8}\right)_+}&\ge1-\frac{A}{\gamma ^8}\E{\sum_{i=1}^{\widehat D_1^m}\left(\frac{\Lambda_i^{(m)}}{n}\right)^{2}\left(1 - \frac{\breve Y_i^{n,\delta}}{\gamma n^{1/4}}\right)^{-8}}\\
    &\ge 1 -\frac{A}{\gamma^8}\E{\sum_{i=1}^{\widehat D_1^m}\left(\frac{\Lambda_i^{(m)}}{m}\right)^{2}\left(1 - \frac{\breve Y_i^{n,\delta}}{\gamma n^{1/4}}\right)^{-8}},
\end{align*}
where the final inequality holds since $m \le n$. Moreover, since $\gamma > 2/u_0$, we have that $|\breve Y_i^{n,\delta}|/(\gamma n^{1/4}) < u_0$ for any $n$, and so $(1 - \breve Y_i^{n,\delta}/(\gamma n^{1/4}))^{-8} \le 1 + 8 \breve Y_i^{n,\delta}/(\gamma n^{1/4}) +  72 (\breve Y_i^{n,\delta})^2/(\gamma ^2\sqrt{n})$. Hence,
\begin{align}
\p{ \breve R_{m,n,\delta}^+ \le \gamma n^{1/4}} &\ge 1-\frac{A}{\gamma^8} + \frac{A}{\gamma^8}\E{1 - \sum_{i=1}^{\widehat D_1^m} \left(\frac{\Lambda_i^{(m)}}{m}\right)^{2}}\label{eq:sizes}\\
&\qquad -\frac{8A }{\gamma^{9 }n^{1/4}}\E{\sum_{i=1}^{\widehat D_1^m}\left(\frac{\Lambda_i^{(m)}}{m}\right)^{2}\breve Y_i^{n,\delta}}\label{eq:for_mean}\\
&\qquad - \frac{72 A}{\gamma^{10}\sqrt{m}}\E{\sum_{i=1}^{\widehat D_1^m}\left(\frac{\Lambda_i^{(m)}}{m}\right)^{2}(\breve Y_i^{n,\delta})^2}\label{eq:for_var},
\end{align}
where we may take the denominator of the final term of the above expression to be $\gamma^{10}\sqrt{m}$ rather than $\gamma^{10}\sqrt{n}$ as  $m \le n$ and the expectation in this term is non-negative.  Applying (\ref{eq:haas_lem}), (\ref{eq:var}), and (\ref{eq:mean}) to bound the expectations in (\ref{eq:sizes}), (\ref{eq:for_var}), and (\ref{eq:for_mean}), respectively, we obtain that there exist constants $B, B', B'' > 0$ such that 
\begin{align*}
\p{ \breve R_{m,n,\delta}^+ \le \gamma n^{1/4}} &\ge 1-\frac{A}{\gamma^8} + \frac{A B''}{\gamma^8 \sqrt{m}}- \frac{8A}{\gamma ^{9}n^{1/4}} \left( \frac{B'n^{1/4 - \delta}}{\sqrt{m}}+\frac{B'}{m^{1/4}}\right) - \frac{72 A B}{\gamma^{10}\sqrt{m}}\\
&\ge  1-\frac{A}{\gamma^8} + \frac{A }{\gamma^8 \sqrt{m}}\left(B'' -\frac{8B'}{\gamma n^\delta} - \frac{8B'}{\gamma} - \frac{72 B}{\gamma^2}\right).
\end{align*}
For $\gamma > 0$ large enough, the final term in parentheses is positive so the whole expression is at least $1 - A/\gamma^8$. The result follows by induction on $m$.
\end{proof}

\section{The hairy tour}\label{sec:hairy_tour}

In this section we prove Theorems \ref{thm:hairy_4} and \ref{thm:hairy_2}. In particular, we show that under assumptions [\ref{a1}] and [\ref{a3}] for a given measure $\pi$ with $\eta \in [0,2)$, we have that $(n^{-1/2}H_n, n^{-1/(4-\eta)}R_n)$ converges in distribution to a generalisation of the hairy tour introduced by Janson and Marckert \cite{discrete_snakes} if $\eta = 0$, and to a process whose second coordinate is a pure jump process if $\eta\in (0,2)$. Recall that by [\ref{a3}], $\pi$ is a Borel measure on $\R^2\setminus \{(0,0)\}$ such that for for any $\varepsilon >0$, both $\pi(\R_+\times (\varepsilon, \infty))< \infty$ and $\pi((\varepsilon,\infty)\times \R_+)< \infty$, and that for all Borel sets $A\subset \R_+^2\setminus \{(0,0)\}$ for which $\pi(\partial A)=0$, 
\[
r^{4-\eta}\p{\frac{1}{r}\left(\max_{1\le i \le \xi} Y_{\xi,i}^+, \max_{1\le i \le \xi}Y^-_{\xi,i}\right)\in A)}\rightarrow \pi(A)
\]
as $r\rightarrow \infty$, where $Y_{k,j}^{+} = Y_{k,j}\vee 0$ and $Y_{k,j}^- = (-Y_{k,j})\vee 0$.
The measure $\pi$ will be the intensity measure for a Poisson point process which drives the second coordinate of the limit.

Recall that $\bT_n = (\rT_n, Y)$ is such that given $\rT_n$, $Y = (Y^{(v)}, v\in v(\rT_n)\setminus \partial \rT_n)$ is a collection of independent random vectors, where if $v\in v(\rT_n)\setminus \partial \rT_n$ has $k$ children then $Y^{(v)}$ has distribution $\nu_k$. Observe that, for fixed $\eta \in [0,2)$, by assumption [\ref{a3}], if the measure $\pi$ has non-zero mass then 
\[
\max_{v\in v(\rT_n)}\|Y^{(v)}\|_\infty = \Theta_\bP(n^{1/(4-\eta)}).
\]
Fix $\gamma > 0$, $\delta \in (0,1/(4-\eta))$, and suppose that $n \ge 1$ is sufficiently large so that $n^{1/(4-\eta) - \delta} \le \gamma n^{1/(4-\eta)}$. As in the proof of tightness for Theorem \ref{thm:main}, and more specifically as in Section \ref{sec:tightness}, in order to prove Theorems \ref{thm:hairy_4} and \ref{thm:hairy_2}, we will need to consider three ``restrictions'' of the branching random walk $\bT_n$. These restrictions are a generalisation of those used in Section \ref{sec:tightness} from the case $\eta = 0$ to that of general $\eta \in [0,2)$; the modified definitions are given below.

We denote the restrictions of $\bT_n$ by $\bT_{n,\delta} = (\rT_n, Y_{n,\delta})$, $\bT_{n,\delta}^\gamma = (\rT_n, Y_{n,\delta}^\gamma)$, and $\bT_n^\gamma = (\rT_n, Y_n^\gamma)$. Again, these branching random walks will respectively capture the ``typical'', ``mid-range'', and ``large'' displacements in $\bT_n$, as follows: 
\begin{enumerate}
\item {\bf (typical displacements):} For all $v\in v(\rT_n)\setminus \partial \rT_n$, \[Y_{n,\delta}^{(v)} = Y^{(v)}\I{\|Y^{(v)}\|_\infty \le n^{1/(4-\eta) - \delta}};\] 
\item  {\bf (mid-range displacements):} For all $v\in v(\rT_n)\setminus \partial \rT_n$, \[Y_{n,\delta}^{\gamma, (v)} = Y^{(v)}\I{n^{1/(4-\eta) - \delta }< \|Y^{(v)}\|_\infty \le \gamma n^{1/(4-\eta)}};\] 
\item {\bf (large displacements):} For all $v\in v(\rT_n)\setminus \partial \rT_n$, \[Y_{n}^{\gamma, (v)} = Y^{(v)}\I{\|Y^{(v)}\|_\infty > \gamma n^{1/(4-\eta)}}.\] 
\end{enumerate}
We note that, informally, taking $\gamma\downarrow 0$ in $\bT_n^{\gamma}$ captures all displacements of the largest order. 
We define $R_{n,\delta}, R_{n,\delta}^\gamma$, and $R_n^\gamma$ to be the functions encoding the spatial locations of the vertices of $\bT_{n,\delta}, \bT_{n,\delta}^\gamma,$ and $\bT_n^\gamma$, respectively.

Before studying the convergence of the functions $R_{n,\delta}$, $R_{n,\delta}^\gamma$, and $R_n^\gamma$, we will prove convergence upon rescaling of the \emph{values} of the large displacements.
For  $v\in v(\rT_n) \setminus \partial \rT_n$, let 
\[
Y^{(v,+)}:=0\vee \max_{j \in [c(v,\rT_n)]} Y^{(v)}_j \quad \text{ and } \quad
Y^{(v,-)}:=0\vee \max_{j \in [c(v,\rT_n)]} (-Y^{(v)}_j)
\]
be the largest positive and negative terms (respectively) in the displacement vector $Y^{(v)}$ from $v$ to its children and, for $v\in \partial \rT_n$, set  $Y^{(v,+)}=Y^{(v,-)}=0$. 
For a finite multiset $S \subset \R^2$, by ``the decreasing ordering of $S$'' we mean the vector $(s_1,\ldots,s_m)$ which lists the elements of $S$ in decreasing order of their largest coordinate, breaking ties in decreasing order of their smallest coordinate.
Let $L_n^{\eta, \gamma}$ be the decreasing ordering of the multiset
\begin{equation} \label{eq:Lndef}
\left\{(Y^{(v,+)}, Y^{(v,-)})\I{\|Y^{(v)}\|_\infty > \gamma n^{1/(4-\eta)}}, ~v\in v(\rT_n)\right\},
\end{equation}
concatenated with an infinite sequence with all entries $(0,0)$. 

\begin{lem}\label{lem:big_jumps_conv}
Fix $\gamma > 0$ and suppose that [\ref{a1}] holds and [\ref{a3}] holds for a given measure $\pi$ with $\eta \in [0,2)$. Then as $n\rightarrow \infty$, 
\[
\frac{L_n^{\eta, \gamma}}{n^{1/(4-\eta)}}\convdist L^{\eta, \gamma}
\]
in $\ell_\infty$, where $L^{\eta, \gamma}$ is the decreasing ordering of the points of a Poisson process on $ \R_{\ge 0}^2$ with intensity $\pi(dx, dy)\I{(x\vee y) > \gamma}$ concatenated with an infinite sequence with all entries $(0,0)$. 
\end{lem}

\begin{proof}
Let $(\xi_i, i\ge 1) $ be \iid samples from the offspring distribution $\mu$.
Further, for $i \ge 1$, sample $Y_{\xi_i}$ independently and let 
\[
Y^{+}_{\xi_i}:=0\vee \max_{j \in [\xi_i]}  Y_{\xi_i,j} \quad \text{ and } \quad
 Y_{\xi_i}^{-}:=0\vee \max_{j \in [\xi_i]} (- Y_{\xi_i, j}).
\]
By definition, the multiset $\{(Y^{(v,+)}, Y^{(v,-)}), v\in v(\rT_n)\}$ is distributed as $\{( Y^+_{\xi_i},  Y^-_{\xi_i}), i\in [n]\}$ conditioned on the event that $\sum_{i=1}^n\xi_i = n-1$. 

For $n \ge 1$, let $\widetilde L_n^{\eta, \gamma}$ be the decreasing ordering of
\[
\left\{(Y_{\xi_i}^+, Y_{\xi_i}^-)\I{\|Y_{\xi_i}\|_\infty > \gamma n^{1/(4-\eta)}},~i\in[n]\right\},
\]
concatenated with an infinite sequence with all entries $(0,0)$. We will first show that 
\begin{equation}\label{eq:large_jumps_tilde} 
n^{-1/(4-\eta)}\widetilde{L}_n^{\eta, \gamma}\convdist L^{\eta, \gamma}
\end{equation}
in $\ell_\infty$ as $n\rightarrow \infty$. To this end, note that by [\ref{a3}], for any $x,y\ge 0$ such that $x\vee y > \gamma$ and such that $\pi((\{x\}\times [y,\infty)) \cup ([x,\infty)\times \{y\})) =0$, 
\[
n\p{Y_{\xi_i}^+ > xn^{1/(4-\eta)},~Y_{\xi_i}^- > yn^{1/(4-\eta)}} \rightarrow \pi((x,\infty)\times (y,\infty)),
\] as $n\rightarrow \infty$ and, moreover, $\pi((x,\infty)\times(y,\infty)) < \infty$. Therefore, 
\begin{align}\label{eq:tight_num_non_zero}
&\left|\left\{i\in [n]~:~Y_{\xi_i}^+ > xn^{1/(4-\eta)},~ Y_{\xi_i} ^-> yn^{1/(4-\eta)}\right\}\right| \nonumber\\
&\quad\quad\quad\eqdist \mathrm{Binomial}\left(n\p{Y_{\xi_i}^+ > xn^{1/(4-\eta)},~Y_{\xi_i}^-> yn^{1/(4-\eta)}}\right)\nonumber\\
&\quad\quad\quad\convdist \mathrm{Poisson}(\pi((x,\infty)\times (y,\infty))),
\end{align}
and (\ref{eq:large_jumps_tilde}) follows from the fact that a Poisson process on $\R^2$ is determined by its distribution on half-infinite rectangles and the continuity of the function $x,y \mapsto x\vee y, x\wedge y$ that we use to order the multisets.

We now show that the convergence in (\ref{eq:large_jumps_tilde}) still holds when we condition on $\sum_{i=1}^n \xi_i = n-1$. We note that the remainder of this proof is similar to the end of the proof of Proposition \ref{prop:large_tight}.

Let $\widetilde{M}^\gamma_n$ be the number of elements in $\widetilde{L}_n^{\eta, \gamma}$ which are not equal to $(0,0)$. Note that by (\ref{eq:tight_num_non_zero}), the sequence $(\widetilde{M}^\gamma_n)_{n \ge 1}$ is tight. Further, let $\widetilde S^\gamma_n := \sum_{i\in [n]}\xi_i\I{\|Y_{\xi_i}\|_\infty > \gamma n^{1/(4-\eta)}}$. Since $\xi_1,\dots \xi_n$ are \iids, the law of $\sum_{i=1}^n\xi_i$ depends on $\widetilde L^{\eta, \gamma}_n$ solely through $\widetilde M^\gamma_n$ and $\widetilde S^\gamma_n$. To be precise, let $\xi_1^n, \xi_2^n,\dots$ be independent random variables such that for each $i\ge 1$, $\xi_i^n$ is distributed as $\xi_i$ conditional on $\|Y_{\xi_i}\|_\infty < \gamma n^{1/(4-\eta)}$. Then,
\begin{equation}\label{eq:need_to_decouple}
\p{\sum_{i=1}^n \xi_i = k~\bigg|~\widetilde{L}^{\eta,\gamma}_n} = \p{\widetilde{S}^\gamma_n + \sum_{i = 1}^{n-\widetilde{M}^\gamma_n}\xi_i^n = k~\bigg|~\widetilde{S}^\gamma_n,~\widetilde{M}^\gamma_n}.
\end{equation}

Let $F:\ell_\infty \rightarrow \R$ be a bounded measurable function. Then, by analogous arguments to those used to prove (\ref{eq:hairy_refer_2}),
\begin{align*}
\E{F(L_n^{\eta, \gamma})} & =\E{F(\widetilde L_n^{\eta, \gamma})\I{\widetilde M^\gamma_n < n^{\varepsilon},~\widetilde S^\gamma_n < n^{1/3+ \varepsilon}}~\bigg|~\sum_{i=1}^n\xi_i = n-1} + o(1)\\
&= \frac{\E{F(\widetilde L_n^{\eta, \gamma})\I{\sum_{i=1}^n \xi_i = n-1,~\widetilde M^\gamma_n < n^{\varepsilon},~\widetilde S^\gamma_n < n^{1/3+\varepsilon}}}}{\p{\sum_{i=1}^n\xi_i = n-1}} + o(1)\\
&= \frac{\E{\E{F(\widetilde L_n^{\eta, \gamma})\I{\sum_{i=1}^n\xi_i = n-1, ~\widetilde M_n^\gamma < n^\varepsilon, ~\widetilde S_n^\gamma < n^{1/3 + \varepsilon}}~\bigg|~ \widetilde L_n^{\eta, \gamma}}}}{\p{\sum_{i=1}^n \xi_i = n-1}} + o(1)\\
&=\E{F(\widetilde L_n^{\eta,\gamma})}\frac{\p{\sum_i^{n-\widetilde M^\gamma_n}\xi_i^n = n\hspace{-1pt}-\hspace{-1pt} 1\hspace{-1pt} -\hspace{-1pt} \widetilde S^\gamma_n~\bigg|~\widetilde M^\gamma_n \hspace{-1pt}<\hspace{-1pt} n^\varepsilon,\widetilde S^\gamma_n<n^{1/3+\varepsilon}}}{\p{\sum_{i=1}^n\xi_i = n-1}} + o(1),
\end{align*} 
where the last equality holds by (\ref{eq:need_to_decouple}). By a quantitative local limit theorem (see Lemma \ref{lem:local_move} in the appendix), we obtain that as $n\rightarrow \infty$
\[
\frac{\p{\sum_i^{n- m}\xi_i^n = n- 1 -  s}}{\p{\sum_{i=1}^n\xi_i = n-1}}\rightarrow 1,
\]
uniformly over all $m < n^\varepsilon$ and $s < n^{1/3 + \varepsilon}$. It follows that 
\[
\E{F(L_n^{\eta, \gamma})} =\E{F(\widetilde L_n^{\eta,\gamma})} + o(1). 
\]
The result then follows by (\ref{eq:large_jumps_tilde}).
\end{proof}

In the remainder of the section, we continue to use $L^{n,\gamma}$ to refer to a random vector with the distribution given in Lemma~\ref{lem:big_jumps_conv}.

To prove Theorems \ref{thm:hairy_4} and \ref{thm:hairy_2}, we use similar methods to those used to prove Theorem \ref{thm:main}. First, we will prove convergence of the branching random walk restricted to the subtree spanned by $k$ uniform vertices, by showing that the convergence from Proposition \ref{prop:fdds} holds jointly with that in Lemma \ref{lem:big_jumps_conv}, and that the limits are independent. This, in particular, implies the convergence of the random finite-dimensional distributions in Theorems \ref{thm:hairy_4} and \ref{thm:hairy_2}.  The independence is the key issue here, and in order to obtain it, we require adaptations of Proposition~\ref{prop:measure_change_basic} and Lemma~\ref{lem:conv_measure_change_basic} to the setting of $n$-dependent offspring distributions. The required technical results may be found in the appendix.

Following this, using similar techniques to those used in Sections~\ref{sec:tightness} and~\ref{sec:technical} to prove tightness for the discrete snake in Theorem \ref{thm:main}, and applying the aforementioned joint convergence, we will show that a discrete snake comprised solely of the ``typical'' displacements converges to the head of the BSBE on rescaling by $n^{-1/4}$ if $\eta = 0$, and to $0$ on rescaling by $n^{-1/(4-\eta)}$ if $\eta \in (0,2)$. Furthermore, this discrete snake is asymptotically independent of the large displacements. In Section~\ref{sec:large} we show that for $\eta\in [0,2)$ the mid-range displacements make only a vanishing contribution to the head of the discrete snake on the scale of $n^{1/(4-\eta)}$. Next, by a small variant of Lemma \ref{lem:pendant_small}, we deduce that the large displacements appear near the leaves. We apply this result to prove Lemma \ref{lem:compare}, which states that the discrete snake associated with the branching random walk $\bT_n'$ obtained by pruning sub-branching random walks rooted at vertices with large displacements in $\bT_n$ converges upon rescaling by $n^{-1/(4-\eta)}$ to the same limit as that of the ``typical displacement'' discrete snake (with the limit depending on whether $\eta = 0$ or $\eta\in (0,2)$). Theorems \ref{thm:hairy_4} and \ref{thm:hairy_2} then follow by showing that the branching random walk obtained by regrafting these pruned sub-branching random walks to uniform leaves of $\bT_n’$ has the same law as $\bT_n$. 

The following proposition establishes the convergence of the branching random walk restricted to the subtree spanned by $k$ uniform vertices, as well as the its asymptotic independence from the large displacements.

\begin{prop}\label{prop:rffds_hairy}
Fix $\gamma > 0$ and suppose that [\ref{a1}] holds and [\ref{a3}] holds for a given measure $\pi$ with $\eta \in [0,2)$. Fix $k \ge 1$. Then  
\begin{equation*}
\frac{\sigma}{\sqrt{n}}(J_1^n, J_2^n,\dots, J_k^n, A_1^n,\dots, A_k^n) \convdist (J_1, J_2,\dots, J_k, A_1,\dots, A_k)
\end{equation*}
as $n\rightarrow \infty$. Jointly with this convergence, we have that
\begin{equation*} 
(F_1^n, F_2^n, \ldots, F_k^n) \convdist (F_1, F_2, \ldots, F_k),
\end{equation*}
where $F_1, F_2, \ldots, F_k$ are \iid random variables, independent of everything else, such that $\p{F_i = 1} = \p{F_i = 2} = 1/2$ and 
\begin{align*}
\left(\frac{L^n(\fl{tn^{1/2}}\wedge (J_1^n-1))}{n^{1/4}}\right)_{t \ge 0} & \convdist \beta (B_{t \wedge (J_1/\sigma)})_{t \ge 0}, \notag \\
\left(\frac{L^n((J_i^n + \fl{t n^{1/2}}) \wedge (J_{i+1}^n-1))}{n^{1/4}}\right)_{t \ge 0} & \convdist  \beta (B_{A_i/\sigma}+ B_{((J_i/\sigma) + t) \wedge (J_{i+1}/\sigma)} - B_{(J_i/\sigma)})_{t \ge 0} 
\end{align*}
for $1 \le i \le k-1$, in each case for the uniform norm. Moreover, jointly with this convergence,
\[
\frac{L_n^{\eta, \gamma}}{n^{1/(4-\eta)}}\convdist L^{\eta, \gamma},
\]
in $\ell_\infty$, where $L^{\eta, \gamma}$ is independent of all the other limiting random variables. 
\end{prop}

\begin{proof}
Fix $k\ge 1$ and $\gamma > 0$ and write 
\begin{align*}
V_n & =  (J_1^n, J_2^n,\dots, J_k^n, A_1^n,\dots, A_k^n, F_1^n, F_2^n, \ldots, F_k^n,\\
& \qquad (L^n(\fl{tn^{1/2}}\wedge (J_1^n-1)))_{t \ge 0}),  (L^n((J_1^n + \fl{t n^{1/2}}) \wedge (J_{2}^n-1)))_{t \ge 0}), \dots, \\
& \hspace{7cm} (L^n((J_{k-1}^n + \fl{t n^{1/2}}) \wedge (J_{k}^n-1)))_{t \ge 0})  
\end{align*}
for the vector containing all variables that, in Proposition~\ref{prop:fdds}, have already been shown to converge jointly under rescaling when we equip the first $3k$ entries with the  Euclidean topology on $\R$, the last $k$ entries with the topology of uniform convergence, and the whole vector with the product topology.  Then, let $g$ be an $\R$-valued bounded continuous function (for this topology), and $h:\ell_\infty\rightarrow \R$ be another bounded continuous function. By Proposition~\ref{prop:fdds} and Lemma~\ref{lem:big_jumps_conv}, it suffices to prove that
\begin{equation}\label{eq:indep_V_n} 
\Big|\E{g(V_n)h(L_n^{\eta, \gamma})}-\E{g(V_n)}\E{h(L_n^{\eta, \gamma})}\Big| \to 0
\end{equation}
as $n\to \infty$. 

Let $(M_n^{\eta, \gamma}, S_n^{\eta, \gamma})$ have the joint distribution of the number of vertices with a large displacement, $\sum_{v\in \rT_n}\I{\|Y^{(v)}\| > \gamma n^{1/(4-\eta)}}$, and the total number of children of such vertices, $\sum_{v\in \rT_n}c(v,\rT_n)\I{\|Y^{(v)}\| > \gamma n^{1/(4-\eta)}}$. Fix $\varepsilon\in (0,1/6)$ and define the good event 
\[
\cG_1=\{M_n^{\eta,\gamma} \le n^\varepsilon,S_n^{\eta, \gamma} \le n^{1/3 + \varepsilon}\}.
\]
By analogous arguments to those used to prove (\ref{eq:hairy_refer_2}), $\cG_1$ occurs with high probability. Now recall that $\sigma n^{-1/2}J^n_k \convdist J_k$ as $n\rightarrow \infty$. Fix $T > 0$ and let $\cG_2$ be the (good) event that $J^n_k\le T\sqrt{n}$. (We observe that by choosing $T$ large we may make $\p{\cG_2}$ as close to 1 as we like, uniformly in $n$ sufficiently large.)  Then, 
\[
\E{g(V_n)h(L_n^{\eta, \gamma})}=\E{g(V_n)h(L_n^{\eta, \gamma})\I{\cG_1\cap \cG_2}} +o(1),
\]
where $o(1)$ is to be understood as an error that tends to $0$ as $n\to \infty$ and then $T\to \infty$.
 
Let $\cF_n^{\eta,\gamma}$ denote the $\sigma$-algebra generated by the degrees and displacement vectors of the vertices $v$ with $\|Y^{(v)}\| > \gamma n^{1/(4-\eta)}$. We see that $\cG_1$ and $L_n^{\eta,\gamma}$ are measurable with respect to $\cF_n^{\eta,\gamma}$, and so 
\[ 
\E{g(V_n)h(L_n^{\eta, \gamma})\I{\cG_1\cap \cG_2 }}=\E{\E{g(V_n)\I{ \cG_2 }\mid \cF_n^{\eta,\gamma}}h(L_n^{\eta, \gamma})\I{\cG_1}}.
\]
Therefore, since $g$ and $h$ are bounded, to prove (\ref{eq:indep_V_n}) it suffices to show that as $n\rightarrow \infty$ and $T \rightarrow \infty$,
\begin{equation} \label{eq:final_conv_V_n}
\Big|\E{g(V_n)\I{ \cG_2}\mid \cF_n^{\eta,\gamma}}\I{\cG_1}-\E{g(V_n)}\Big|\convprob 0.
\end{equation}

To prove (\ref{eq:final_conv_V_n}), we will use the measure change between a size-biased random array and a vector of \iid size-biased random variables which may be found in Proposition~\ref{prop:sizebias_swole} below. To this end, let $\xi^n$ denote a random variable with distribution $\mu$, conditioned not to yield a large displacement vector (i.e.\ conditioned on $\max_{1\le i \le \xi^n} |Y_{\xi^n,i}|\le \gamma n^{1/(4-\eta)}$), and let $\mu^n$ denote the distribution of $\xi^n$. Using similar notation to that in Proposition~\ref{prop:sizebias_swole}, write $r_n$ for the value of  $M_n^{\eta,\gamma}$, $s_n$ for the value of $S_n^{\eta, \gamma}$ and $d_1,\dots, d_{r_n}$ for the degrees of the vertices $v$ with $\|Y^{(v)}\| > \gamma n^{1/(4-\eta)}$. Then, let $\xi^n_{r_n+1},\dots,\xi^n_{n}$ be \iid samples from $\mu^n$ and write $\vec{Z}=(Z_1,\dots, Z_n)=(d_1,\dots, d_{r_n},\xi^n_{r_n+1},\dots,\xi^n_{n})$. Further, conditionally given $\vec{Z}$, let $\Sigma=\Sigma_{\vec{Z}}$ be the random permutation in \eqref{eq:sizebias_def}, so that $(Z_{\Sigma(1)},\dots, Z_{\Sigma(n)})$ is a size-biased random re-ordering of $\vec{Z}$. Also define $\tau_{r_n}(\Sigma)=\min\{j\in [n]:\Sigma(j)\in [r_n]\}$.  Finally, write $N=N_{n,r_n}=|\{i\in \{r_n+1,\dots, n\} :\xi^n_i>0\}|$.  
 
Note that conditionally on $\cF_n^{\eta,\gamma}$, the remaining vertex degrees are distributed as $\xi^n_{r_n+1},\dots,$ $\xi^n_{n}$ conditioned on $\xi^n_{r_n+1}+\dots+\xi^n_{n}=n-1-s_n$. Therefore, 
\begin{align*}
&\E{g(V_n)\I{ \cG_2}\mid \cF_n^{\eta,\gamma}}\\
&\quad=\E{\E{g(V_n)\I{ \cG_2} ~\bigg|~ \xi^n_{r_n+1},\dots,\xi^n_{n},\sum_{i=r_n+1}^n \xi^n_{i}=n-1-s_n, \cF_n^{\eta,\gamma}} \bigg|~ \cF_n^{\eta,\gamma}}.
\end{align*}

By (\ref{eq:christina_lemma}), $\tau_{r_n}(\Sigma)>T\sqrt{n}$ with high probability. Furthermore, by a Chernoff bound, $N\ge T\sqrt{n}$ with high probability. It follows that
\begin{align*}
&\E{\E{g(V_n)\I{ \cG_2} ~\bigg|~ \xi^n_{r_n+1},\dots,\xi^n_{n},\sum_{i=r_n+1}^n \xi^n_{i}=n-1-s_n, \cF_n^{\eta,\gamma}} \bigg|~ \cF_n^{\eta,\gamma}}\\
&=\E{\E{g(V_n)\I{ \cG_2} \I{N\ge T\sqrt{n},\tau_{r_n}(\Sigma)>T\sqrt{n}} \bigg| \xi^n_{r_n+1},...,\xi^n_{n}, \hspace{-5pt} \sum_{i=r_n+1}^n \hspace{-6.3pt}\xi^n_{i}=n-1-s_n, \cF_n^{\eta,\gamma}} \hspace{-2pt}\bigg| \cF_n^{\eta,\gamma}}\\
&\quad+o_\bP(1).
\end{align*} 
Now observe that, on the event $\cG_2$, $\mathrm{T}^k_n$ contains at most $T\sqrt{n}$ vertices, and further on the event $\tau_{r_n}(\Sigma)>T\sqrt{n}$, none of these vertices have a displacement exceeding $\gamma n^{1/(4-\eta)}$. This implies that $g(V_n)\I{ \cG_2} \I{N\ge T\sqrt{n},\tau_{r_n}(\Sigma)>T\sqrt{n}}$ only depends on $\xi^n_{r_n+1},\dots,\xi^n_{n}$ and $ \cF_n^{\eta,\gamma}$ through $Z_{\Sigma(1)},\dots,Z_{\Sigma(\lfloor T\sqrt{n}\rfloor)}$ and $\Sigma(1),\dots,\Sigma(\lfloor T\sqrt{n}\rfloor)$. Therefore, 
\begin{align*}
&\E{g(V_n)\I{ \cG_2}\bigg|~ \cF_n^{\eta,\gamma}}\\
&= \E{\E{g(V_n)\I{ \cG_2} \I{N\ge T\sqrt{n},\tau_{r_n}(\Sigma)>T\sqrt{n}} \bigg| \left(Z_{\Sigma(i)}\right)_{ i \in [\lfloor T\sqrt{N}\rfloor]},\left(\Sigma(i)\right)_{i\in [\lfloor T\sqrt{N}\rfloor]}}\hspace{-2pt}\bigg| \cF_n^{\eta,\gamma}}\\
&\quad +o_\bP(1). \\
&=\E{\E{g(V_n)\I{ \cG_2}  \bigg| \left(Z_{\Sigma(i)}\right)_{ i \in [\lfloor T\sqrt{N}\rfloor]},\left(\Sigma(i)\right)_{i\in [\lfloor T\sqrt{N}\rfloor]}}\I{N\ge T\sqrt{n},\tau_{r_n}(\Sigma)>T\sqrt{n}}\bigg|  \cF_n^{\eta,\gamma}}\\
&\quad +o_\bP(1),
\end{align*}
where the last equality is implied by the fact that the events $N\ge T\sqrt{n}$ and $\tau_{r_n}(\Sigma)>T\sqrt{n}$ are measurable with respect to $Z_{\Sigma(1)},\dots,Z_{\Sigma(\lfloor T\sqrt{n}\rfloor)},\Sigma(1),\dots,\Sigma(\lfloor T\sqrt{n}\rfloor)$. However, observe that $g(V_n)\I{ \cG_2}$ is independent of $\Sigma(1),\dots,\Sigma(\lfloor T\sqrt{n}\rfloor)$ given $Z_{\Sigma(1)},\dots,Z_{\Sigma(\lfloor T\sqrt{n}\rfloor)}$, and so 
\begin{align}\label{eq:to_apply_sb}
&\E{\E{g(V_n)\I{ \cG_2}  \bigg| \left(Z_{\Sigma(i)}\right)_{ i \in [\lfloor T\sqrt{N}\rfloor]},\left(\Sigma(i)\right)_{i\in [\lfloor T\sqrt{N}\rfloor]}}\I{N\ge T\sqrt{n}, \tau_{r_n}(\Sigma)>T\sqrt{n}}\bigg|~  \cF_n^{\eta,\gamma}}\nonumber\\
&\quad = \E{\E{g(V_n)\I{ \cG_2}  \bigg| \left(Z_{\Sigma(i)}\right)_{ i \in [\lfloor T\sqrt{N}\rfloor]}}\I{N\ge T\sqrt{n},\tau_{r_n}(\Sigma)>T\sqrt{n}}\bigg|~
  \cF_n^{\eta,\gamma}}.
\end{align}

We now apply the measure change from Proposition~\ref{prop:sizebias_swole} to obtain that, for $\bar{\xi}^n_1,\bar{\xi}^n_2,\dots$ \iid samples from the size-biased law of $\mu^n$, (\ref{eq:to_apply_sb}) is equal to
\begin{align}\label{eq:applied_mc}
&\E{\E{g(V_n)\I{ \cG_2}  \bigg| \bar{\xi}^n_1,\dots,\bar{\xi}^n_{\lfloor T\sqrt{n}\rfloor}} \Theta_{\mu^n}^{n,r_n,s_n}(\bar{\xi}_1^n, \dots, \bar{\xi}_{\lfloor T\sqrt{n}\rfloor}^n) \mid \cF_n^{\eta,\gamma} },
\end{align}
where the inner conditional expectation of $g(V_n)\I{\mathcal{G}_2}$ is now thought of as a measurable functional of the \iid random variables $\bar{\xi}^n_1,\dots,\bar{\xi}^n_{\lfloor T\sqrt{n}\rfloor}$ in place of $Z_{\Sigma(1)},\dots,Z_{\Sigma(\lfloor T\sqrt{n}\rfloor)}$. 
This implies that 
\begin{align*}
&\E{g(V_n)\I{\cG_2}\mid\cF_n^{\eta,\gamma}}\I{\cG_1}\\
&\quad = \E{\E{g(V_n)\I{\cG_2}\bigg| \bar{\xi}_1^n,\dots, \bar{\xi}^n_{\lfloor T\sqrt{n}\rfloor}}\Theta_{\mu^n}^{n, r_n, s_n}(\bar{\xi}^n_1,\dots, \bar{\xi}_{\lfloor T\sqrt{n}\rfloor}^n)\bigg| \cF_n^{\eta, \gamma}}\I{\cG_1} +o_\bP(1). 
\end{align*} 
By applying Lemma \ref{lem:conv_measure_change_hairy} on $\cG_1$ (which occurs with high probability), 
\[
\Theta_{\mu^n}^{n,r_n,s_n}(\bar{\xi}_1^n, \dots, \bar{\xi}_{\lfloor T\sqrt{n}\rfloor}^n) \convprob 1
\] 
as $n\rightarrow \infty$ and $(\Theta_{\mu^n}^{n,r_n,s_n}(\bar{\xi}_1^n, \dots, \bar{\xi}_{\lfloor T\sqrt{n}\rfloor}^n))_{n\ge 0}$ is uniformly integrable, so (\ref{eq:applied_mc}) is equal to 
\[
\E{\E{g(V_n)\I{ \cG_2}  \bigg| \bar{\xi}^n_1,\dots,\bar{\xi}^n_{\lfloor T\sqrt{n}\rfloor}} \bigg| \cF_n^{\eta,\gamma} } +o_\bP(1). 
\]
Since $\E{g(V_n)\I{ \cG_2}  \bigg| \bar{\xi}^n_1,\dots,\bar{\xi}^n_{\lfloor T\sqrt{n}\rfloor}}$ does not depend on $\cF_n^{\eta,\gamma}$, it follows that
\[
\E{\E{g(V_n)\I{ \cG_2}  \bigg| \bar{\xi}^n_1,\dots,\bar{\xi}^n_{\lfloor T\sqrt{n}\rfloor}} \bigg| \cF_n^{\eta,\gamma} }=\E{\E{g(V_n)\I{ \cG_2}  \bigg| \bar{\xi}^n_1,\dots,\bar{\xi}^n_{\lfloor T\sqrt{n}\rfloor}} }.
\]
By Corollary~\ref{cor:dtv_cond}, the total variation distance between $(\bar{\xi}^n_1,\dots,\bar{\xi}^n_{\lfloor T\sqrt{n} \rfloor})$ and \iid size-biased samples from $\mu$, henceforth denoted by $(\bar{\xi}_1,\dots,\bar{\xi}_{\lfloor T\sqrt{n}\rfloor})$, tends to $0$ as $n\to\infty$. Therefore, since $g$ is bounded,
\[
\E{\E{g(V_n)\I{ \cG_2}  \bigg| \bar{\xi}^n_1,\dots,\bar{\xi}^n_{\lfloor T\sqrt{n}\rfloor}} }=\E{\E{g(V_n)\I{ \cG_2}  \bigg| \bar{\xi}_1,\dots,\bar{\xi}_{\lfloor T\sqrt{n}\rfloor}} }+o(1).
\]
Finally, by Lemma~\ref{lem:conv_measure_change_basic} and Proposition~\ref{prop:measure_change_basic}, this is in turn equal to 
\begin{align*}
\E{\E{g(V_n)\I{ \cG_2}  \bigg| \widehat{D}_1^n,\dots,\widehat{D}_{\lfloor T\sqrt{n}\rfloor}^n}\I{N_n\ge  T\sqrt{n}\rfloor}} +o(1),
\end{align*}
where we recall that $N_{n} = |\{i\in [n]~:~D_i^n > 0\}|$. Again, since the probability of $\cG_2$ and $N_{n}\ge \lfloor T\sqrt{n}\rfloor$ occurring tends to $1$ as $n\to \infty$ and subsequently $T\to \infty$, we see that 
\begin{align*}
\E{\E{g(V_n)\I{ \cG_2}  \bigg| \widehat{D}_1^n,\dots,\widehat{D}_{\lfloor T\sqrt{n}\rfloor}^n}\I{N_n\ge  T\sqrt{n}\rfloor}} 
= \E{g(V_n)}+o(1),
\end{align*}
which proves (\ref{eq:final_conv_V_n}). The result follows.
\end{proof}

\subsection{Typical displacements}\label{sec:trimmed_snake}
Fix $\eta\in [0,2)$ and  $\delta\in(0,1/(10-4\eta))\subset (0,1/(4-\eta))$. In this section we will study the function encoding the spatial locations of the branching random walk $\bT_{n,\delta} = (\rT_n, Y_{n,\delta})$, namely $R_{n,\delta}:[0,n] \rightarrow \R$. 

\begin{prop}\label{prop:trimmed_snake} 
Fix $\gamma > 0$ and suppose that [\ref{a1}] holds and [\ref{a3}] holds for a given measure $\pi$ with $\eta \in [0,2)$. Let $\delta \in (0,1/(10-4\eta))$. If $\eta = 0,$ then 
\[\left(\left(\frac{H_n(nt)}{\sqrt{n}}, \frac{R_{n,\delta}(nt)}{n^{1/4}}\right)_{0 \le t \le 1}, \frac{L^{0, \gamma}_n}{n^{1/4}}\right)\convdist \left(\left(\frac{2}{\sigma}\mathbf{e}_t,\beta\sqrt{\frac{2}{\sigma}} \mathbf{r}_t\right)_{0 \le t\le 1}, L^{0, \gamma}\right),\]
as $n\rightarrow \infty$, in $\bC([0,1],\R^2)\times \ell_\infty$. Furthermore, $L^{0,\gamma}$ is independent of $((\mathbf{e}_t, \mathbf{r}_t))_{0 \le t \le 1}$.

If $\eta\in (0,2)$, then 
\[
\left(\left(\frac{H_n(nt)}{\sqrt{n}}, \frac{R_{n,\delta}(nt)}{n^{1/(4-\eta)}}\right)_{0 \le t \le 1}, \frac{L^{\eta, \gamma}_n}{n^{1/(4-\eta)}}\right)\convdist \left(\left(\frac{2}{\sigma}\mathbf{e}_t,0\right)_{0 \le t\le 1}, L^{\eta, \gamma}\right),
\]
in $\bC([0,1],\R^2)\times\ell_\infty$, where $L^{\eta,\gamma}$ is independent of $(\mathbf{e}_t)_{0 \le t \le 1}$. 
\end{prop}

\begin{proof} 
The convergence of the random finite-dimensional distributions follows from Proposition~\ref{prop:rffds_hairy} exactly as Corollary~\ref{cor:4.2} follows from Proposition~\ref{prop:fdds}, but now with the additional independence from $L^{\eta,\gamma}$. 

We will obtain tightness (now on the scale of $n^{1/(4 - \eta)}$) via arguments very similar to those in Section~\ref{sec:tightness}, where we replace the truncations with those defined in Section \ref{sec:hairy_tour}. In particular, the key point is that we must show the analogue of Proposition~\ref{prop:tightness_typ}, which states that
\[
\lim_{k \to \infty} \limsup_{n \to \infty} \p{\max_{0 \le i \le k} \sup_{s,t \in [U_{(i)}^{n,k}-1, U_{(i+1)}^{n,k}-1]} |R_{n,\delta}(s) - R_{n,\delta}(t)| > \gamma n^{1/(4-\eta)}} = 0.
\]

Fix $\delta \in (0,1/(10 - 4\eta)).$ For all $n\ge 1$ and $k \ge 1$ let $Y_k^{n,\delta}\in \R^k$ be such that 
\[
Y_k^{n,\delta} = (Y_{k,1}^{n,\delta},\dots, Y_{k,k}^{n,\delta}) 
:= \begin{cases}
(Y_{k,1},\dots, Y_{k,k}) & \text{if } \max_{1\le j \le k}|Y_{k,j}| \le n^{1/(4-\eta) - \delta},\\
(0,\dots, 0) & \text{else.}
\end{cases}
\]
As discussed in Section \ref{sec:technical}, the displacements of the branching random walk $\bT_{n,\delta}$ are not necessarily globally centered and so may not satisfy [\ref{a1}]. Thus to prove the result, we will need to instead consider the re-centered branching random walk $(\rT_n, Y_{n,\delta}^*)$ where conditionally on $\rT_n$, the entries of $Y_{n,\delta}^* = (Y_{n,\delta}^{*,(v)}, ~v\in v(\rT_n)\setminus \partial \rT_n)$ are independent random vectors, such that if $v\in v(\rT_n)\setminus \partial \rT_n$ has $k$ children then $Y^{*,(v)}_{n,\delta}$ has the same distribution as \[ Y_k^{n,\delta} - \E{Y_{\bar\xi, U_{\bar\xi}}^{n,\delta}}.\] The function $R_{n,\delta}^*:[0,n]\rightarrow \R$ encoding the spatial locations of $(\rT_n, Y_{n,\delta}^*)$ is such that for all $t\in [0,n]$, \begin{equation}\label{eq:rec_trim}
R_{n,\delta}^*(t) \eqdist R_{n,\delta}(t) - \E{Y_{\bar\xi, U_{\bar\xi}}^{n,\delta}}\cdot H_n(t).
\end{equation}

By Lemma~\ref{lem:moments_truncation_move},
\[
\left|\E{Y_{\bar\xi, U_{\bar\xi}}^{n,\delta}}\right|=O\left((n^{1/(4-\eta) - \delta})^{1-2(4-\eta)/3}\right).
\]
Since $(n^{-1/2}H_n(nt))_{0 \le t \le 1} \convdist \frac{2}{\sigma}(\mathbf{e}_t)_{0 \le t\le 1}$ as $n\rightarrow \infty$ in $\mathbf{C}([0,1], \R)$, it then follows that 
\begin{equation}\label{eq:drift_conv}
\frac{\|H_n\|_\infty}{n^{1/(4-\eta)}}\E{Y_{\bar\xi, U_{\bar\xi}}^{n,\delta}} \convprob 0
\end{equation}
as long as $\delta > 0$ satisfies
\[
\left(\frac{1}{4-\eta} - \delta\right)\left(1 - \frac{2(4-\eta)}{3}\right) < \frac{1}{4-\eta} - \frac{1}{2}.
\]
Rearranging, this is equivalent to requiring that $\delta < (10 - 4\eta)^{-1}.$ For these values of $\delta$, we then have
\[
\sup_{t \in [0,1]}|R_{n,\delta}^*(t) - R_{n,\delta}(t)| \convdist 0,
\]
and so there is no asymptotic cost in doing this re-centering. Arguing again exactly as in Section~\ref{sec:tightness}, it is sufficient to prove the analogue of Lemma~\ref{lem:max_tight}, which states that there exists $A > 0$ such that for any $\gamma > 0$, $\delta \in (0,1/(4-\eta))$ and $n \ge 1$ we have
\[
\p{ \|R_{n,\delta}^*\|_{\infty} > \gamma n^{1/(4-\eta)}} \le \frac{A}{\gamma^8}.
\]
It is straightforward to verify that the proof of Lemma~\ref{lem:max_tight} given in Section~\ref{sec:technical} generalises immediately to this setting, on replacing $n^{1/4}$ by $n^{1/(4-\eta)}$.
\end{proof}

\subsection{Mid-range and large displacements}\label{sec:large}
We will adapt the proof of Proposition \ref{prop:mid_range} to the case where [\ref{a3}] holds instead of [\ref{a2}]. The proof of Proposition \ref{prop:mid_range} uses Lemma \ref{lem:a_no_ganging_up_d_tree}  to show that, with high probability, there are no vertices with a mid-range or large displacement that are ancestrally related. To apply that lemma, it is sufficient to bound both the maximum degree in the tree and the number of vertices with a mid-range or large displacement, with high probability. The required bound on the maximal degree follows from the assumption that $\E{\xi^3}<\infty$. Therefore, for the adaptation, we need to obtain the same control on the number of mid-range displacements under [\ref{a3}] as we obtained under [\ref{a2}] in Lemma~\ref{lem:few_hairs}. 

\begin{lem}
Suppose that [\ref{a3}] holds for a given measure $\pi$ and $\eta\in [0,2)$. For $\delta > 0$ sufficiently small,
\begin{equation*}
\left|\left\{ v\in v(\rT_n)\setminus \partial \rT_n \text{ such that } \|Y^{(v)}\|_\infty >  n^{1/(4-\eta) - \delta}\right\}\right| = o_\bP(n^{1/12}).
\end{equation*}
\end{lem}

\begin{proof}
Let $\xi_1,\dots, \xi_n$ be \iid with distribution $\mu$. Let $x\in (0,1)$ be such that $\pi(\{x\}\times \R_+)=\pi(\R_+\times \{x\})=0$. Then, by [\ref{a3}],
\begin{align*} 
n^{1-(4-\eta)\delta} \p{\|Y_{\xi_1}\|_\infty >  n^{1/(4-\eta) - \delta}}
&\le n^{1-(4-\eta)\delta} \p{\|Y_{\xi_1}\|_\infty >  x n^{1/(4-\eta) - \delta}}\\
&\to \pi\big(((x,\infty)\times \R_+) \cup( \R_+\times (x,\infty))\big) <\infty. 
\end{align*}
This in particular implies that there exists $C > 0$ such that $\p{\|Y_{\xi_1}\|_\infty >  n^{1/(4-\eta) - \delta}}\le Cn^{-1+(4-\eta)\delta}$ for all $n\ge 1$. It follows that 
\[ 
A_n := \left|\left\{i\in [n]~:~ \|Y_{\xi_i}\|_\infty >  n^{1/(4-\eta) - \delta}\right\}\right| \preceq_{st} \mathrm{Bin}\left(n, Cn^{-1 + (4-\eta)\delta}\right).
\]
By a Chernoff bound, this implies that for $\delta \in (0, (12(4-\eta))^{-1})$, and $n \ge 1$ sufficiently large, for any $\varepsilon >0$,
\begin{align*}
\p{ A_n > \varepsilon n^{1/12}}&\le \p{\mathrm{Bin}\left(n, Cn^{-1+(4-\eta)\delta}\right) > \varepsilon n^{1/12}} \\
& = \p{\mathrm{Bin}\left(n, Cn^{-1+(4-\eta)\delta}\right) \hspace{-1pt}>\hspace{-1pt} Cn^{(4-\eta)\delta}\left(1 + \left(\frac{\varepsilon}{C}n^{1/12 - (4-\eta)\delta} - 1\right)\right)}\\
& = O\left(\exp(-n^{(4-\eta)\delta})\right),
\end{align*}
so 
\begin{equation*}
\p{ A_n > \varepsilon n^{1/12}~\bigg|~\sum_{i=1}^n\xi_i = n-1} 
= O\left(n^{1/2}\exp(-n^{(4-\eta)\delta})\right)=o(1). \qedhere
\end{equation*}
\end{proof}

We now obtain the following result on the mid-range displacements under [\ref{a3}] with a proof that is analogous to that of Proposition~\ref{prop:mid_range}; we omit the details.

\begin{prop}\label{prop:mid_range_hairy}
Fix $\gamma > 0$ and suppose that [\ref{a3}] holds for a given measure $\pi$ and $\eta\in [0,2)$. For $\delta > 0$ sufficiently small, as $n\rightarrow \infty$,
\[
\p{\|R_{n,\delta}^\gamma\|_\infty > \gamma n^{1/(4-\eta)}} = o(1).
\]
\end{prop}

In the remainder of this section we will study the function encoding the spatial locations of the ``large-displacement'' branching random walk $\bT_n^\gamma = (\rT_n, Y_n^\gamma)$, namely $R_{n}^\gamma:[0,n] \rightarrow \R$.

Let $\Xi$ be a Poisson process on $[0,1]\times \R_+^2\setminus \{(0,0)\}$ with intensity $dt\otimes \pi (dx,dy)$, and let $\Xi^\gamma$ be the restriction of $\Xi$ to $[0,1]\times (\R_+^2\setminus([0,\gamma]\times [0,\gamma]))$. Also, recall the definition of the function $U$ from just before Theorem~\ref{thm:hairy_4}.
 
\begin{prop}\label{prop:conv_spiky_snake} 
Fix $\gamma >0$ and suppose that [\ref{a1}] holds and that [\ref{a3}] holds for a given measure $\pi$ and $\eta \in [0,2)$. Let $\delta \in \left(0, \frac{1}{64}\wedge \frac{1}{10 - 4\eta}\right)$. 

If $\eta = 0$ then as $n\rightarrow \infty$, 
\begin{equation*}
\left(\left(\frac{H_n(nt)}{\sqrt{n}}, \frac{R_{n,\delta}(nt)}{n^{1/4}}\right)_{0 \le t \le 1}, U\left(\frac{R_n^\gamma}{n^{1/4}}, \emptyset\right)\right)\convdist \left(\left(\frac{2}{\sigma}\mathbf{e}_t, \beta\sqrt{\frac{2}{\sigma}} \mathbf{r}_t\right)_{0 \le t \le 1}, U(0, \Xi^\gamma)\right)
\end{equation*}
with convergence in the first coordinate in $\bC([0,1],\R^2)$, and convergence in the second coordinate with respect to the Hausdorff topology on non-empty compact subsets. Furthermore, $U(0,\Xi^\gamma)$ is independent of $(\mathbf{e}_t, \mathbf{r}_t,~ 0 \le t \le 1)$. 

If $\eta \in (0,2)$ then as $n \to \infty$,
\begin{equation*}
\left(\left(\frac{H_n(nt)}{\sqrt{n}}, \frac{R_{n,\delta}(nt)}{n^{1/(4-\eta)}}\right)_{0 \le t \le 1}, U\left(\frac{R_n^\gamma}{n^{1/(4-\eta)}}, \emptyset\right)\right)\convdist \left(\left(\frac{2}{\sigma}\mathbf{e}_t, 0\right)_{0 \le t \le 1}, U(0, \Xi^\gamma)\right)
\end{equation*}
with convergence in the first coordinate in in $\bC([0,1],\R^2)$, and convergence in the second coordinate with respect to the Hausdorff topology on non-empty compact subsets. Furthermore, $U(0,\Xi^\gamma)$ is independent of $(\mathbf{e}_t, 0 \le t \le 1)$. 
\end{prop}

We first prove Theorems \ref{thm:hairy_4} and \ref{thm:hairy_2} assuming Proposition \ref{prop:conv_spiky_snake}.

\begin{proof}[Proof of Theorems \ref{thm:hairy_4} and \ref{thm:hairy_2} assuming Proposition \ref{prop:conv_spiky_snake}]

For $\gamma$ and $\delta$ as in Proposition~\ref{prop:conv_spiky_snake},
\begin{equation*}
\left( \frac{R_n(nt)}{n^{1/(4-\eta)}} \right)_{0 \le t \le 1} = \left(\frac{R_{n,\delta}(nt)}{n^{1/(4-\eta)}} + \frac{R_n^\gamma(nt)}{n^{1/(4-\eta)}} \right)_{0 \le t \le 1}+ \left(\frac{R_{n,\delta}^\gamma(nt)}{n^{1/(4-\eta)}}\right)_{0 \le t \le 1}.
\end{equation*}
By Proposition \ref{prop:conv_spiky_snake}, as $n\rightarrow \infty$, $U(n^{-1/(4-\eta)}R_n^\gamma(n\ \!\cdot), \emptyset) \convdist U(0, \Xi^\gamma)$ with respect to the Hausdorff topology on non-empty compact subsets, jointly with convergence 
\[
\left(\frac{R_{n,\delta}(nt)}{n^{1/(4-\eta)}}\right)_{0 \le t \le 1} \convdist 
\begin{cases}
\beta\sqrt{\frac{2}{\sigma}}\mathbf{r} & \text{ if } \eta = 0, \\
0 & \text{ if } \eta \in (0,2)
\end{cases}
\]
in $\mathbf{C}([0,1],\R^2)$ where, for $\eta = 0$, $U(0, \Xi^\gamma)$ and $(\mathbf{r}_t)_{0 \le t \le 1}$ are independent. Therefore,
\[
U\left(\frac{R_{n,\delta}(n\cdot)}{n^{1/(4-\eta)}} + \frac{R_n^\gamma(n\cdot)}{n^{1/(4-\eta)}}, \emptyset\right)
\convdist 
\begin{cases}
U\left(\beta\sqrt{\frac{2}{\sigma}}\mathbf{r}, \Xi^{\gamma}\right) & \text{ if } \eta = 0,\\
U(0, \Xi^{\gamma}) & \text{ if } \eta\in (0,2).
\end{cases}
\]
Note that $U(0,\Xi)$ is a compact set by our assumptions on $\pi$, and that $U(0,\Xi^\gamma) \convas U(0,\Xi)$ in the Hausdorff sense as $\gamma \downarrow 0$. We have
\[
d_{\mathrm{H}}\left(U \left(\frac{R_n(n\cdot)}{n^{1/(4-\eta)}}, \emptyset\right), U\left(\frac{R_{n,\delta}(n\cdot)}{n^{1/(4-\eta)}} + \frac{R_n^\gamma(n\cdot)}{n^{1/(4-\eta)}}, \emptyset\right) \right) \le n^{-1/(4-\eta)} \|R_{n,\delta}^{\gamma}\|_{\infty}
\]
and, by Proposition~\ref{prop:mid_range_hairy},
\[
\lim_{\gamma \to 0} \limsup_{n \to \infty} \p{\|R_{n,\delta}^{\gamma}\|_{\infty} > \gamma n^{1/(4-\eta)}} = 0.
\]
We may now apply the principle of accompanying laws \cite[Theorem 3.2]{billingsley2013convergence} in order to obtain that
\[
U\left(\frac{R_{n}(n\cdot)}{n^{1/(4-\eta)}}, \emptyset\right)
\convdist 
\begin{cases}
U\left(\beta\sqrt{\frac{2}{\sigma}}\mathbf{r}, \Xi\right) & \text{ if } \eta = 0,\\
U(0, \Xi) & \text{ if } \eta\in (0,2),
\end{cases}
\]
which yields Theorems \ref{thm:hairy_4} and \ref{thm:hairy_2}.
\end{proof}

\medskip

The remainder of this section is devoted to the proof of Proposition \ref{prop:conv_spiky_snake}. We will need a notion of pruning and grafting of branching random walks. We refer to Figure \ref{fig:prune_graft} as a visual aid in understanding the following three definitions.

\begin{dfn}[Pruning branching random walks]\label{dfn:extract}
Let $\rT = (T, Y)$ be a branching random walk with displacements $Y = (Y^{(v)}, v\in v(T)\setminus \partial T)$. Let $v\in v(T)$ and $T^{(v)}$ be the subtree of $T$ rooted at $v$. The sub-branching random walk of $\rT$ rooted at $v$ is the branching random walk $\rT^{(v)} = (T^{(v)}, Y')$ with displacements $Y' = (Y'^{(u)}, u\in v(T^{(v)})\setminus \partial T^{(v)})$.
Also, $\rT^{\uparrow v}$ is  the branching random walk obtained by removing all descendants of $v$ from $T$. More generally, for $\mathrm{v} = (v_1,\dots,v_k)$ a sequence of distinct vertices in $v(T)$ such that no two vertices in $\mathrm{v}$ are ancestrally related, we set $\rT^\mathrm{v} = (\rT^{(v)}, v\in \mathrm{v}),$ and define $\rT^{\uparrow \mathrm{v}}$ inductively as $\rT^{\uparrow\mathrm{v}} = (\rT^{\uparrow(v_1,\dots, v_{k-1})})^{\uparrow v_k}.$
\end{dfn}  

\begin{dfn}[Grafting branching random walks]\label{dfn:graft}
For branching random walks $\rT=(T,Y)$ and $\rT'=(T', Y')$, and for a leaf $l\in\partial T$, let $\rT\oplus_l\rT'=(T\oplus_l T, Y\oplus_l Y')$ be the branching random walk defined by setting $T\oplus_l T'=T\cup lT'$ and, for $v\in v(T\oplus_l T')\setminus \partial ( T\oplus_l T')$, setting 
\[
(Y \oplus_l Y')^{(v)}=
\begin{cases} 
Y^{(v)} &\text{ if }v\in v(T)\setminus \partial T, \\
Y'^{(u)} &\text{ if }v = lu \text{ for some } u\in  v(T')\setminus \partial T'.
\end{cases} 
\]
More generally, for branching random walks $\rT, \rT^1, \dots, \rT^k$  and distinct leaves $l_1,\dots, l_k \in \partial T$, define $\rT\oplus_{l_1,\dots,l_k}(\rT^1,\dots,\rT^k)$ recursively over $k$ as 
\[
\rT\oplus_{l_1,\dots,l_k}(\rT^1,\dots,\rT^k)=(\rT\oplus_{l_1,\dots,l_{k-1}}(\rT^1,\dots,\rT^{k-1}))\oplus_{l_k} \rT^k.
\]
\end{dfn}

The previous definitions imply that for a branching random walk $\rT = (T,Y)$ and $v\in v(T)$, 
\[\rT^{\uparrow v}\oplus_v \rT^{(v)} =\rT,\] and, more generally, for a sequence of distinct vertices $\mathrm{v}=(v_1,\dots,v_k)$ of $T$ such that no two vertices in $\mathrm{v}$ are ancestrally related in $T$, that $\rT^{\uparrow\mathrm{v}}\oplus_\mathrm{v} \rT^\mathrm{v} =\rT.$

We next use the above definitions to define a map that prunes the sub-branching random walks of branching random walks that are rooted at ancestrally minimal vertices $v$ with $\|Y^{(v)}\|_\infty \ge \tau$. See Figure \ref{fig:prune_graft} for an illustration of the coming definition.
 
\begin{dfn}\label{def:map}
For a branching random walk $\rT = (T, Y)$ and for $\tau>0$, let $\mathrm{v}_\tau=( v_1, \dots, v_m)$ be the set of vertices $v\in v(T)$ such that $\|Y^{(v)}\|_\infty > \tau$ and for all ancestors $u \preceq v$, $\|Y^{(u)}\| \le \tau$, listed in depth-first order. Define a map $f_\tau$ by 
\[
\rT \stackrel{f_{\tau}}{\longmapsto} (\rT^{\uparrow\mathrm{v}_\tau}, \{\rT^{(v_1)},\dots,\rT^{(v_m)}\}),
\]
where the second coordinate is a multiset with elements $\rT^{(v_1)},\dots,\rT^{(v_m)}$ which are the branching random walks rooted at the vertices $v_1,\dots, v_m$.
\end{dfn}

For $\tau \ge 0$, let 
\[
v^{\tau}(\bT_n) := \left\{v\in v(\rT_n)\setminus \partial \rT_n ~:~\|Y^{(v)}\|_\infty > \tau, \text{ and } \|Y^{(u)}\|_\infty \le \tau ~\forall u\prec v\right\}.
\]
We will apply $f_\tau$ to $\bT_n$, and then study the law of $\bT_n$ conditional on $f_\tau(\bT_n)$. Observe that, given $f_\tau(\bT_n)$, $\bT_n$ is determined by $\mathrm{v}^\tau(\bT_n)$. We will show that conditional on $f_\tau(\bT_n)$,~$\mathrm{v}^\tau(\bT_n)$ is distributed as a uniformly random subset of leaves in $(f_{\tau}(\bT_n))_1$. We make this formal in the next lemma. 

\begin{lem}\label{lem:uniform_grafting}
Let $\tau > 0$ and write $f_\tau(\bT_n)=(\bT_n',\{\bT_n^1,\dots,\bT_n^m\})$. Fix $m \ge 1$ and let $\Sigma \in_\cU \cS_m$, where $\cS_m$ is the symmetric group of order $m$. Further, let $(\cL_1, \dots, \cL_m)$ be a uniformly random vector of leaves in $\bT'_n$ listed in depth-first order. 
Then, given $f_\tau(\bT_n)$, $\bT_n$ is equal in distribution to
\[
\bT_n'\oplus_{\cL_1,\dots,\cL_m}(\bT_n^{\Sigma(1)},\dots, \bT_n^{\Sigma(m)}).
\]
\end{lem}

\begin{figure}[h!]
    \centering
    \tikz[grow =up,  nodes={circle,draw}, minimum size = 
    1.2cm, inner sep = 0pt, scale = 0.6, transform shape]
  \node {0}
    child {node[above right] {	-5}
    child {node[right] {5}
    child {node {7}}
    child {node {-1}}
    child {node {	3}
    }
    }}
    child {node[above] {	-1}}
    child {node[above left]{	 1}
    child {node {9}
    child {node {	 4}}
    child {node {7}}}
    child {node {5}
    child {node[ left] {7}}}}
    ;
\end{figure}
\vspace{-0.7cm}
\begin{figure}[h!]
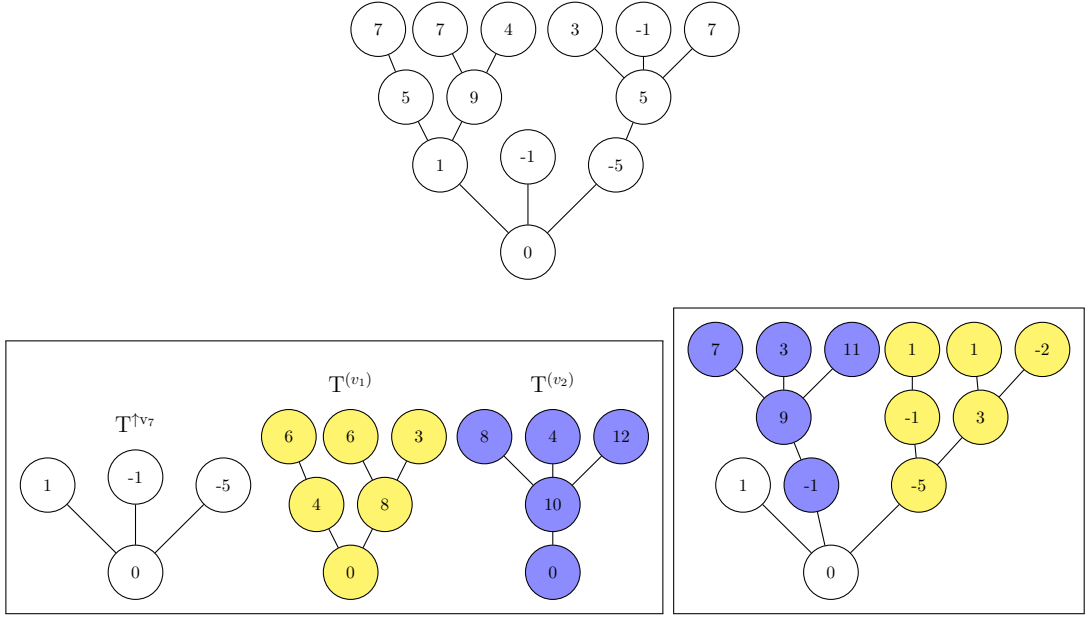
    \centering
       \tmpframe{\tikz[grow =up,  nodes={circle,draw}, minimum size = 
    1.2cm, inner sep = 0pt, scale = 0.6, transform shape]
  \node (A) {0}
    child {node[above right] {-5}}
    child {node[above, label = above: {\Large $\rT^{\uparrow\mathrm{v}_7}$}] {	-1}}
    child {node[above left] {	 1}
    }
    ;

    \tikz[grow=up, nodes={circle,draw}, minimum size=1.2cm, inner sep = 0pt, scale = 0.6, transform shape]
    \node[fill = yellow!70!white]{0}
    child {node[fill = yellow!70!white] {8}
    child {node [fill = yellow!70!white] {3}}
    child {node [fill = yellow!70!white, label=above:{\Large $\rT^{(v_1)}$}] {6}}}
    child {node[fill = yellow!70!white] {4}
    child {node [left, fill = yellow!70!white] {6}}}
    ;
     \tikz[grow=up, nodes={circle,draw}, minimum size=1.2cm, inner sep = 0pt, scale = 0.6, transform shape]
    \node[fill = blue!45!white]{0}
    child {node[fill = blue!45!white] {10}
    child {node[fill = blue!45!white] {12}}
    child {node[fill = blue!45!white, label=above:{\Large $\rT^{(v_2)}$}] {4}}
    child {node[fill = blue!45!white] {8}}
    }
    ;}
   \tmpframe{\tikz[grow =up,  nodes={circle,draw}, minimum size = 
    1.2cm, inner sep = 0pt, scale = 0.6, transform shape]
  \node {0}
    child {node[above right,  fill = yellow!70!white] {-5}
    child {node[right, fill = yellow!70!white]{3}
    child {node [right, fill = yellow!70!white] {-2}}
    child {node [right, fill = yellow!70!white] {1}}}
    child {node[right,  fill = yellow!70!white]{-1}
    child {node [fill = yellow!70!white] {1}}}}
    child {node[above left,fill = blue!45!white] {-1}
    child {node[left, fill = blue!45!white] {9}
    child {node[fill = blue!45!white] {11}}
    child {node[fill = blue!45!white] {3}}
    child {node[fill = blue!45!white] {7}}}}
    child {node[above left] {	 1}}
    ;}
\caption{  On top, a spatial tree. We denote the associated branching walk by $\rT$.  On the bottom left, we depict $f_7(\rT)=(\rT^{\uparrow\mathrm{v}_7},\{\rT^{(v_1)},\rT^{(v_2)}\})$, which is obtained from $\rT$ by pruning the sub-branching walks of $\rT$ that have a displacement with absolute value exceeding $7$ in their first generation. On the bottom right is a spatial tree obtained by grafting the branching walks $(f_7(\rT))_2$ to leaves of $(f_7(\rT))_1$.
$\rT^{(v_1)}$ and $\rT^{(v_2)}$ to leaves of $\rT^{\uparrow\mathrm{v}_7}$.
}\label{fig:prune_graft}
    \end{figure}

\begin{proof} 
Let $(\mathrm{t}',\{\mathrm{t}^1,\dots,\mathrm{t}^m\})$ be in the support of $f_\tau (\bT_n)$. We will first show that
\begin{equation}\label{eq:inv}
f^{-1}_\tau(\mathrm{t}',\{\mathrm{t}^1,\dots,\mathrm{t}^m\})=\left\{ \mathrm{t}'\oplus_{l_1,\dots,l_m}(\mathrm{t}^{\pi(1)},\dots, \mathrm{t}^{\pi(m)}): (l_1,\dots ,l_m)\text{ leaves in }\mathrm{t}'; \pi \in \cS_m\right\},
\end{equation}
where on the right-hand side, $(l_1,\dots,l_m)$ are listed in depth-first order.
Following this, we will show that the law of $\bT_n$ conditional on its degrees and displacement vectors assigns equal mass to all elements of the right-hand set in (\ref{eq:inv}), and that each element of the right-hand set corresponds to the same number of sets of leaves $(l_1,\dots , l_m)$ listed in depth-first order and permutations $\pi$. 

For the inclusion of the left-hand set in the right-hand set, observe that if for some spatial tree $\mathrm{t}$ it holds that $f_\tau(\mathrm{t})=(\mathrm{t}',\{\mathrm{t}^1,\dots,\mathrm{t}^m\})$ then $(l_1,\dots, l_m)$, the minimal vertices in $\mathrm{t}$ that have a displacement vector with sup-norm lower bounded by $\tau$ listed in depth-first order, are leaves in $\mathrm{t}'$. Thus there is some $\pi\in \cS_m$ such that for all $i=1,\dots,m$, $\mathrm{t}^{\pi(i)}=\mathrm{t}^{(l_i)}$. This implies that $\mathrm{t}=\mathrm{t}'\oplus_{l_1,\dots,l_m}(\mathrm{t}^{\pi(1)},\dots, \mathrm{t}^{\pi(m)})$. 

For the other inclusion, it is straightforward to see that for leaves $(l_1,\dots , l_m)$ in $\mathrm{t}'$, listed in depth-first order, and $ \pi \in \cS_m$, it holds that $f_\tau(\mathrm{t}'\oplus_{l_1,\dots,l_m}(\mathrm{t}^{\pi(1)},\dots, \mathrm{t}^{\pi(m)})=(\mathrm{t}',\{\mathrm{t}^1,\dots,\mathrm{t}^m\})$. 

We now show that the law of $\bT_n$ conditional on its degrees and displacement vectors assigns equal mass to all elements of the right-hand set. This follows from the observation that, conditional on its degrees and displacement vectors, $\bT_n$ is uniform on all branching random walks with those degrees and displacement vectors. Each element in 
\begin{equation}\label{eq:set_of_interest}
\left\{ \mathrm{t}'\oplus_{l_1,\dots,l_m}(\mathrm{t}^{\pi(1)},\dots, \mathrm{t}^{\pi(m)}):(l_1,\dots , l_m) \text{ leaves in }\mathrm{t}'; \pi \in \cS_m\right\}
\end{equation} 
with $l_1,\dots, l_m$ listed in depth-first order has the same degrees and displacement vectors. 

Finally, we show that each element of (\ref{eq:set_of_interest}) corresponds to the same number of sets of leaves $(l_1, \dots,  l_m)$ and permutations $\pi\in \cS_m$. To this end, note that every vertex in a spatial tree $\mathrm{t}$ with a displacement vector whose sup-norm is at least $\tau$ has a non-zero number of children, so for each $\mathrm{t}$ in the set (\ref{eq:set_of_interest}), we can recognise $(l_1,\dots ,l_m)$ as the vertices $v$ that are leaves in $\mathrm{t}'$ and not leaves in $\mathrm{t}$; thus, the choice of $(l_1,\dots , l_m)$ is unique. Moreover, if the multiset $\{\mathrm{t}^1,\dots,\mathrm{t}^m\}$ contains $j$ different spatial trees with multiplicities $m_1,\dots, m_j$ respectively, then $\mathrm{t}$ corresponds to $m!/(m_1!\dots m_j!)$ different permutations $\pi$. This number does not depend on $\mathrm{t}$, and the statement follows.  
\end{proof}

For $n \ge 1$, let $\tau_n = n^{1/(4-\eta)-\delta}$. Further, let $\bT_n' = (\rT_n', Y')$ denote the first coordinate of $f_{\tau_n}(\bT_n)$, and $\bF^{\mathrm{pr}}_n = (\bT_n^{(v)})_{v\in v^{\tau_n}(\rT_n)}$ denote the second coordinate of $f_{\tau_n}(\bT_n)$, where we assume that the trees in $\bF^{\mathrm{pr}}_n$ are ordered according to the depth-first order of their roots in $\rT_n$.  We require one further lemma to prove Proposition \ref{prop:conv_spiky_snake}.

\begin{lem}\label{lem:compare} 
Fix $\gamma > 0$. Suppose that [\ref{a1}] holds and that [\ref{a3}] holds for a given measure $\pi$ and $\eta\in [0,2)$. For $n\ge 1$, let $H_n'$ be the height function of $\rT_n'$ and $R_n'$ be the function encoding the spatial locations of $\bT_n'$. Extend their domains to $[0,n]$ by setting $H_n'(t) = R_n'(t) = 0$ for all $t > |\rT_n'|$. If $\eta = 0$, then as $n\rightarrow\infty$,
\begin{equation}\label{eq:skorokhod_T'}
\left(\left(\frac{H'_n( nt )}{\sqrt{n}}, \frac{R'_{n}( nt )}{n^{1/4}}\right)_{0 \le t \le 1}, \frac{L_n^{0,\gamma}}{n^{1/4}}, \right) \convdist \left(\left(\frac{2}{\sigma}\mathbf{e}_t, \beta \sqrt{\frac{2}{\sigma}}\mathbf{r}_t\right)_{0 \le t \le 1}, L^{0,\gamma}\right),
\end{equation}
and if $\eta\in (0,2)$, then 
\begin{equation}\label{eq:skorokhod_T''}
\left(\left(\frac{H'_n( nt )}{\sqrt{n}}, \frac{R'_{n}( nt )}{n^{1/(4-\eta)}}\right)_{0 \le t \le 1}, \frac{L_n^{\eta,\gamma}}{n^{1/(4-\eta)}}, \right) \convdist \left(\left(\frac{2}{\sigma}\mathbf{e}_t, 0\right)_{0 \le t \le 1}, L^{\eta,\gamma}\right),
\end{equation}
with convergence in the first coordinate in $\bC([0,1],\R^2)$ endowed with the topology of uniform convergence, and the convergence in the second coordinate in $\ell_\infty$. 
\end{lem}

\begin{proof}
We prove (\ref{eq:skorokhod_T'}). The proof of (\ref{eq:skorokhod_T''}) then follows by identical arguments. By Proposition~\ref{prop:trimmed_snake} it suffices to prove that as $n\rightarrow \infty$,
\begin{equation}\label{eq:encoding_T'}
\sup_{1\le j \le n}\left\{ n^{-1/2}|H_n(j)-H'_n(j)|\, \vee \, n^{-1/4}|R_{n,\delta}(j)-R'_n(j)| \right\} \convprob 0.
\end{equation}

We also prove (\ref{eq:encoding_T'}) using Proposition \ref{prop:trimmed_snake}. Fix $\varepsilon>0$. The sample paths of both $\mathbf{e}$ and~$ \mathbf{r}$ are almost surely continuous so, since $[0,1]$ is compact, they are in fact almost surely uniformly continuous. This implies that there exists $\rho>0$ so that  
\[
\p{\sup_{0\leq s <t \leq 1, |s-t|<\rho}\left|\frac{2}{\sigma}\mathbf{e}_s-\frac{2}{\sigma}\mathbf{e}_t\right|\vee \left|\beta\sqrt{\frac{2}{\sigma}} \mathbf{r}_s-\beta \sqrt{\frac{2}{\sigma}}\mathbf{r}_t\right|>\varepsilon/2}<\varepsilon/2. 
\]
Then, the convergence in Proposition \ref{prop:trimmed_snake} implies that for $n$ sufficiently large, the probability that the event
\begin{equation*}\label{eq:bad_event}
B_n := \left\{\sup_{0\leq k <\ell \leq n, |k-\ell|<\rho n}\left\{ \frac{\left|H_n(k)-H_n(\ell)\right|}{n^{1/2}}\vee \frac{\left|R_{n,\delta}(k)-R_{n,\delta}(\ell)\right|}{n^{1/4}}\right\}\ge\varepsilon\right\}
\end{equation*}
occurs is less than $\varepsilon$. 

Next, let 
\[
v^*(\rT_n) := \left\{v\in v(\rT_n)\setminus \partial\rT_n:\|Y^{(v)}\|_\infty \le n^{1/4 - \delta}, \exists ~ u\prec v \text{ with } \|Y^{(u)}\|_\infty > n^{1/4-\delta}\right\}.
\] 
By identical methods as those used to prove Lemma~\ref{lem:pendant_small}, it can be seen that $v^*(\rT_n) = o_\bP(n)$ and so for $n$ sufficiently large, $\p{|\rT_n'| \le n- \rho n} \le \varepsilon.$ 

Now suppose that neither of the (unlikely, bad) events $\{|\rT_n'| \le n - \rho n\}$ or $B_n$ hold. 
Observe that $H'_n$ and $R'_n$ can respectively be obtained from $H_n$ and ${R}_{n,\delta}$ by ``skipping'' all the vertices in $v^*(\rT_n)$. To be precise, for $1\leq k \leq |\rT_n'|$, let $P_n(k)$ be the position of the $k$-th vertex that is not in $v^*(\rT_n)$ in the depth-first order of $\rT_n$. Then, 
\[
(H'_n(k),R'_n(k)) =
\begin{cases}
(H_n(P_n(k)),R_{n,\delta}(P_n(k))) & \text{for }k=1,\dots,|\rT_n'|,\\
(0,0)&\text{for }k>|\rT_n'|.
\end{cases}
\]
By our assumption that $n - |\rT_n'| <\rho n$, we have $|P_n(k)-k|<\rho n $ for all $k$; by our assumption that 
\[
\sup_{0\leq k <\ell \leq n, |k-\ell|<\rho n}\left\{ \frac{\left|H_n(k)-H_n(\ell)\right|}{n^{1/2}}\vee \frac{\left|R_{n,\delta}(k)-R_{n,\delta}(\ell)\right|}{n^{1/4}}\right\}<\varepsilon,
\]
we then also have
\begin{equation}\label{eq:for_the_end?}
\sup_{0\leq k \leq n }\left\{\frac{|H_n(k)-H'_n(k)|}{\sqrt{n}}\vee\frac{|R_{n,\delta}(k)-R'_n(k)|}{n^{1/4}}\right\}<\varepsilon.
\end{equation}
Since $\varepsilon > 0$ was arbitrary, the result follows.
\end{proof}

With Lemma \ref{lem:compare} in hand, we proceed to proving Proposition \ref{prop:conv_spiky_snake}. In the proof, the pair $(\bT_n', \bF_n^{\mathrm{pr}})$ is  as in Lemma~\ref{lem:compare}.
Observe that by Lemma \ref{lem:uniform_grafting}, given  $f_{\tau_n}(\bT_n)$, we can obtain an object with the same law as $\bT_n$ by grafting the branching random walks in $\bF^\mathrm{pr}_n$ at uniformly random leaves of the first coordinate of $f_{\tau_n}(\bT_n)$. 
\begin{proof}[Proof of Proposition \ref{prop:conv_spiky_snake}]
Let $n \ge 1$ be large enough so that $n^{1/(4-\eta) - \delta} < \gamma n^{1/(4-\eta)}$. Then if $v\in v(\rT_n)\setminus \partial \rT_n$ is such that $\|Y^{(v)}\|_\infty > \gamma n^{1/(4-\eta)}$, it also holds that $\|Y^{(v)}\|_\infty > n^{1/(4-\eta) - \delta}$. The proof of Proposition \ref{prop:ancestral_rel} can be adapted so that under [\ref{a3}] for a given measure $\pi$ and $\eta\in [0,2)$, for $\delta > 0$ sufficiently small, as $n\rightarrow \infty$
\[\p{\exists u, v\in \bT_n, u\prec v, \text{ such that } \|Y^{(u)}\|_\infty \wedge \|Y^{(v)}\|_\infty > n^{1/(4-\eta) - \delta}} = o(1).\] If follows that at the cost of throwing away an event of asymptotically vanishing probability, we may work on the event that there are no ancestrally related vertices $u,v\in v(\rT_n)$ such that both $ \|Y^{(v)}\|_\infty > n^{1/(4-\eta) - \delta}$ and $\|Y^{(u)}\|_\infty> n^{1/(4-\eta) - \delta}$.

By Skorokhod's representation theorem, we may work on a probability space where the convergence in Lemma \ref{lem:compare} holds almost surely. 

We now use Lemma~\ref{lem:uniform_grafting} to study the asymptotic law of $R^\gamma_{n}$ conditional on $(\bT_n', \bF_n^{\mathrm{pr}})$. Lemma~\ref{lem:uniform_grafting} implies that given $(\bT_n', \bF^{\mathrm{pr}})$, we can obtain an object with the law of $\bT_n$ by grafting each of the branching random walks in $\bF_n^{\mathrm{pr}}$ onto uniformly random leaves in  $\bT'_n$. In fact, in order to obtain the (conditional) law of $R^\gamma_{n}$ we only need to sample the positions of the vertices in $v\in v(\rT_n)\setminus \partial \rT_n$ whose displacement vectors $Y^{(v)}$ satisfy that $\|Y^{(v)}\|_\infty \ge \gamma n^{1/(4-\eta)}$, since the trees of $\bF_n^{\mathrm{pr}}$ attach to these vertices in exchangeable random order. We denote the branching random walks in $\bF_n^{\mathrm{pr}}$ by $\bT^{(v_1)},\dots,\bT^{(v_{M_n})}$ (ordered according to the depth-first order of their roots $v_1,\dots, v_{M_n}\in \rT_n$).  By symmetry we may assume that for $1\le j \le M_n$, the largest and smallest displacement at the root of $\bT^{(v_j)}$ (i.e., $Y^{(v_j, +)}$ and $Y^{(v_j, -)}$) are described by the $j$-th entry of $L^{\eta,\gamma}_n$. 

We claim that as $n\rightarrow \infty$, $M_n\convdist M$ for some finite, random variable $M$. Indeed, as $n\rightarrow \infty$, $n^{-1/(4-\eta)}L_n^{\eta,\gamma}\convas L^{\eta,\gamma}$. Furthermore, since $\gamma> 0$, almost surely $L^{\eta,\gamma}$ has finitely many non-zero terms and each non-zero entry of $n^{-1/(4-\eta)} L_n^{\eta,\gamma}$ is at $\ell^\infty$ distance at least~$\gamma$ from $(0,0)$, there are finite random variables $M$ and $N$ such that $L_n^{\eta,\gamma}$ has $M$ non-zero terms for all $n>N$ large enough; i.e., the number of vertices $v\in v(\rT_n)\setminus \partial \rT_n$ such that $\|Y^{(v)}\|_\infty \ge \gamma n^{1/(4-\eta)}$ is equal to $M$. 

Now, for $k \ge 1$, let $\cL'_n(k)$ denote the number of leaves in $\rT_n'$ which are among the first $k$ vertices in the depth-first order of the vertices of $\rT_n$.  Then $\cL_n'(k)$ is bounded from above by the number of down-steps of the Łukasiewicz path $W_n(k)$ of $\rT_n$ by time $k$. It is bounded from below by this same number minus $|T_n\setminus T'_n|$ which is $o(n)$ by \eqref{eq:skorokhod_T'}.

Therefore, by Lemma \ref{lem:serflln}, as $n\rightarrow \infty$,
\[
\left(\frac{\cL'_n( \lfloor nt\rfloor )}{n}\right)_{0\le t \le 1}\convprob \left(\mu_0 t\right)_{0\le t \le 1}, 
\]
so that the positions of $M_n$ uniform leaves in $\rT'_n$ in depth-first order converge upon rescaling by $n^{-1}$ to $M$ independent uniform samples from $[0,1]$, which we denote by $U_1,\dots, U_M$ respectively. For all $1 \le j \le M_n$, we graft $\bT^{(v_j)}$  (which has size $o(n)$ since $\rT'_n$ has size $n-o(n)$ by \eqref{eq:skorokhod_T'}) onto the $j$-th such leaf of $\bT_n'$, using the operation in Definition \ref{dfn:graft}.

The branching random walk $\bT^{(v_j)}$ contains exactly one vertex (namely the root) with displacement vector $\|Y^{(v)}\|_\infty > \gamma n^{1/(4-\eta)}$ (since we assumed that such vertices are not ancestrally related) and the largest and smallest displacements of this vertex are given by~$L^{\eta,\gamma}_n(j)$. Therefore, asymptotically,  $n^{-1/(4-\eta)}R^\gamma_{n}$  will contain a line segment from $(U_j,-Y_j^-)$ to $(U_j,-Y_j^+)$. This implies that if $\eta = 0$
\[ 
\left(\left(\frac{H'_n( nt)}{\sqrt{n}},\frac{R'_{n}( n t )}{n^{1/4}}\right)_{ 0 \le t \le 1} , U\left(\frac{R^\gamma_{n}}{n^{1/4}},\emptyset\right) \right)\convdist \left(\left(\frac{2}{\sigma}\mathbf{e}_t, \beta\sqrt{\frac{2}{\sigma}} \mathbf{r}_t\right)_{0 \le t \le 1}, U(0,\Xi^\gamma)\right), 
\]
and if $\eta\in (0,2)$
\[ 
\left(\left(\frac{H'_n( nt)}{\sqrt{n}},\frac{R'_{n}( n t )}{n^{1/(4-\eta)}}\right)_{ 0 \le t \le 1} , U\left(\frac{R^\gamma_{n}}{n^{1/(4-\eta)}},\emptyset\right) \right)\convdist \left(\left(\frac{2}{\sigma}\mathbf{e}_t, 0\right)_{0 \le t \le 1}, U(0,\Xi^\gamma)\right).
\]
The result then follows from (\ref{eq:encoding_T'}). 
\end{proof}

%%%%%%%%%%%%%%%%%%%%%%%%%%%%%%%%%%%%%%%%%%%%%%
%% Single Appendix:                         %%
%%%%%%%%%%%%%%%%%%%%%%%%%%%%%%%%%%%%%%%%%%%%%%
\begin{appendix}
\section*{Standard results and remaining proofs}\label{sec:appendix}

\subsection{Standard results} 

In this section we provide standard results which we use throughout this work without proof. 
We start by stating a functional strong law of large numbers for sums of \iid non-negative random variables that we use at multiple points in  proofs of convergence of  finite-dimensional distributions. 

\begin{lem}\label{lem:serflln}
Let $X_1,X_2,\dots$ be \iid random variables with $X_1\ge 0$ almost surely and $\E{X_1}=\mu<\infty$. Then, for any $a_n\uparrow\infty$,
\[
\left(\frac{1}{a_n}\sum_{i=1}^{\lfloor a_n t \rfloor} X_i, t\ge 0\right)\convas \left(\mu t , t\ge 0\right),
\]
uniformly on compact sets as $n\to \infty$.
\end{lem}

The next result is a generalised local central limit theorem, from Theorem 13, Chapter~VII of Petrov \cite{petrov}, which we use to prove tightness in Theorem \ref{thm:main}.

\begin{thm}[Theorem 13, Chapter VII of Petrov \cite{petrov}]\label{thm:gen_clt} 
Let $(X_n)_{n \ge 1}$ be a sequence of \iid integer-valued random variables. Suppose that $\E{X_1} = 0$, $\V{X_1} = \sigma^2 > 0$, $\E{|X_1|^3} < \infty$, and the maximal span of the distribution of $X_1$ is equal to $1$. Let $S_n = \sum_{i=1}^nX_i.$ Then,
\[
\sqrt{2\pi n\sigma}\p{S_n = k} = e^{-k^2/(2\sigma^2 n)}\left(1 + \frac{1}{\sqrt{n}}\frac{\gamma_3}{6\sigma^3}\left(\frac{k^3}{\sigma^3n^{3/2}} - \frac{3k}{\sigma\sqrt{n}}\right)\right) + o(n^{-1/2}),
\]
uniformly in $k\in \Z$, where $\gamma_3$ is the third central moment of $X_1.$
\end{thm}

The last result is a quantitative local central limit theorem proved in \cite[Lemma 5.5]{louigiserteigor} for $k = 1$, which we use in the proof of Theorems \ref{thm:hairy_4} and \ref{thm:hairy_2}. The generalisation to $k \ge 1$ is standard.

\begin{lem}\label{lem:local_move}
Fix $\eta, \beta >0$, $0<\gamma<1/2$ and $k\in \N$. Then, there exist constants $C=C(\eta, \beta, \gamma, k)$ and $M=M(\eta, \beta, \gamma, k)$ so that for all random variables $X$ on  $\Z_{\ge 0}$ that satisfy the following conditions:
\begin{enumerate}
\item the greatest common divisor of the support of $X$ is $1$;
\item $\p{X=0}>\gamma$ and $\p{X=k}>\gamma$;
\item $\E{X^2}<\eta$ and
\item $\E{X^3}<\beta $,
\end{enumerate}  
it holds that for all $m > M$
\[
\sup_{\ell\in \Z} \left|\sqrt{m}\p{ \sum_{i=1}^m X_i=\ell}-\phi\left( \frac{\ell-m\E{X}}{\sqrt{\V{X}m}}\right) \right|\leq \frac{C}{\sqrt{m}},
\]
where $X_1,X_2,\dots$, are \iid copies of $X$, and $\phi(t)=e^{-t^2/2}$ is the standard normal density.
\end{lem}

\subsection{Supporting results from the introduction}\label{app:intro}

\begin{proof}[Proof of Lemma~\ref{lem:noatoms}]
We argue by contradiction. Fix $T>0$. Without loss of generality, assume that $\pi(\{x\}\times \R_+)=\delta>0$ for some $x>0$. We show that this implies that $\pi((x/2,\infty)\times \R_+)> T$, which contradicts the requirement that $\pi((x/2,\infty)\times \R_+)<\infty$ because $T>0$ was chosen arbitrarily. Fix $0<\varepsilon<x/4$ small enough that $\delta\lfloor \tfrac{ x}{8\varepsilon}\rfloor>T$ and $\pi(\{x-\varepsilon,x+\varepsilon\}\times \R_+)=0$. Define $A_0=(x-\varepsilon, x+\varepsilon)$, so that by [\ref{a3}],
 \[
 r^{4-\eta}\p{\frac{1}{r}\max_{1\le i \le \xi}Y^+_{\xi,i}\in A_0}\to \pi(A_0\times \R_+) \ge  \delta \quad \text{ as $r\to \infty$.}
 \]
Then, letting $J= \lfloor \tfrac{ x}{8\varepsilon} \rfloor -1$, for $j\in\{1,\dots, J\}$, we can find $\theta_j\in (0,1]$ such that $A_j:=\theta_j(x-\varepsilon,x+\varepsilon)\subset (x-(2j+1)\varepsilon,x-(2j-1)\varepsilon)$ and $\pi(\{\theta_j(x-\varepsilon),\theta_j(x+\varepsilon)\}\times \R_+)=0$. By definition, $A_0,\ldots,A_J$ are pairwise disjoint, and by our choice for $J$, $\cup_{0\le j\le J}A_j\subset (x/2,x+\varepsilon)$, so $\pi((x/2,\infty)\times \R_+)\ge \sum_{0\le j \le J}\pi(A_j\times \R_+)$. Moreover, 
setting $r=\theta_j s$ in the above limit shows that
\[
s^{4-\eta}\p{\frac{1}{s}\max_{1\le i \le \xi}Y^+_{\xi,i}\in  \theta_j(x-\epsilon, x+\epsilon)} \to \theta_j^{\eta-4} \pi(A_0\times \R_+) \ge \delta\text{ as $s\to \infty$}.
\]
But [\ref{a3}] implies that 
\[
s^{4-\eta}\p{\frac{1}{s}\max_{1\le i \le \xi}Y^+_{\xi,i}\in  \theta_j(x-\epsilon, x+\epsilon)} \to \pi(A_j\times \R_+),
\]
so 
$\pi((x/2,\infty)\times \R_+)\ge (J+1)\delta>T$, which implies the claim. 
\end{proof}

\begin{proof}[Proof of Proposition \ref{prop:big_displacement}]
To ease notation, we write $n/2$ instead of $\lfloor n/2 \rfloor$ throughout the proof. 

First observe that, by assumption, there exist $\varepsilon,\delta>0$ such that, for $\xi_1,\dots, \xi_{n}$ \iid samples from $\mu$, 
\[ 
\limsup_{n\to \infty}\p{\max_{1\le i \le n} \max_{1\le j\le \xi_i}|Y_{\xi_i,j}|>\delta n^{1/4}}>\varepsilon.
\]
By the central limit theorem, we may pick $K$ large enough that 
\[
\liminf_{n\to \infty}\p{n/2-Kn^{1/2}\le \sum_{i=1}^{n/2} \xi_i \le n/2+Kn^{1/2}}>1-\varepsilon/2, 
\]
so that by a union bound
\[
\limsup_{n\to \infty}\p{\max_{1\le i \le n} \max_{1\le j\le \xi_i}|Y_{\xi_i,j}|>\delta n^{1/4},~ n/2-Kn^{1/2}\le \sum_{i=1}^{n/2} \xi_i \le n/2+Kn^{1/2} }>\varepsilon/2.
\]
Denote the event inside the probability by $\cE_n$. We see that
\begin{align*} 
& \p{\max_{v \in v(\rT_n) \setminus \partial \rT_n} \max_{j \ge 1}|Y_{j}^{(v)}|>\delta n^{1/4}} = \p{\max_{1\le i \le n} \max_{1\le j\le D^n_i}|Y_{D_i^n,j}|>\delta n^{1/4}}\\
&\quad \ge \p{\max_{1\le i \le n/2} \max_{1\le j\le D_i^n}|Y_{D_i^n,j}|>\delta n^{1/4},~ n/2-Kn^{1/2}\le \sum_{i=1}^{n/2} D_i^n \le n/2+Kn^{1/2} } \\
&\quad =\frac{\p{\cE_n  \cap \{\sum_{i=1}^n \xi_i=n-1\} }}{\p{\sum_{i=1}^n \xi_i=n-1} }\\
&\quad = \frac{\E{\I{\cE_n} \p{\sum_{i=n/2+1}^n \xi_i=n-1-\sum_{i=1}^{n/2}\xi_i\mid \xi_1,\dots, \xi_{n/2}, Y_{\xi_1},\dots, Y_{\xi_n/2}} }}{\p{\sum_{i=1}^n \xi_i=n-1} }\\
&\quad \ge \p{\cE_n}\frac{\min_{n/2-1-Kn^{1/2}\le m \le n/2-1+Kn^{1/2}}\p{\sum_{i=n/2+1}^n \xi_i=m }}{\p{\sum_{i=1}^n \xi_i=n-1 }}.
\end{align*}
By the local central limit theorem, there exist constants $c,C>0$ such that 
\begin{align*}
&\liminf_{n\to \infty} n^{1/2}\min_{n/2-1-Kn^{1/2}\le m \le n/2-1+Kn^{1/2}} \p{\sum_{i=n/2+1}^n \xi_i=m }>c \\
\text{and} \quad &\limsup_{n\to \infty} n^{1/2}\p{\sum_{i=1}^n \xi_i=n-1 }<C.
\end{align*}
Thus, as claimed 
\[
\limsup_{n\to \infty}\p{\max_{1\le i \le n} \max_{1\le j\le D^n_i}|Y_{D_i^n,j}|>\delta n^{1/4}}\ge \frac{\varepsilon c}{2 C}>0. \qedhere
\]
\end{proof}

\subsection{Measure change}\label{app:measure_change}

For $n \ge 1$ let $\cS_n$ denote the set of permutations of $[n]$. For $(k_1,\ldots,k_n) \in \N^n$, let $\Sigma = \Sigma_{(k_1,\dots,k_n)}$ be the random permutation of $[n]$ with law given by 
\[
\p{\Sigma = \sigma} = \prod_{i=1}^n \frac{k_{\sigma(i)}}{\sum_{j=i}^n k_{\sigma(j)}},\quad \mbox{ for }\sigma\in \cS_n.
\]
We call $(k_{\Sigma(1)},\dots k_{\Sigma(n)})$ the \emph{size-biased random re-ordering} of $(k_1,\dots, k_n).$ It will be convenient to extend this definition to vectors $(k_1,\dots, k_n)$ that contain $0$-valued entries. We start with a size-biased random re-ordering of the non-zero entries of $(k_1,\dots, k_n)$ and then append to this the correct number of zeroes. Formally, if $(k_1,\dots,k_n) \in \Z_{\ge 0}^n$ has $N \ge 0$ non-zero entries, let $\Sigma_{(k_1,\dots, k_n)}$ be the random permutation of $[n]$ with 
\begin{equation}\label{eq:sizebias_def}
\p{\Sigma_{(k_1,\dots, k_n)} = \sigma} = \frac{1}{(n-N)!}\prod_{i=1}^N\frac{k_{\sigma(i)}}{\sum_{j=1}^Nk_{\sigma(j)}},
\end{equation}
for $\sigma\in \cS_n$, and still refer to $(k_{\Sigma(1)},\dots k_{\Sigma(n)})$ as the size-biased random re-ordering of $(k_1,\dots, k_n)$.

For a permutation $\sigma \in \cS_n$ and $r \in \{0,1,\ldots,n\}$ define 
\[
\tau_r(\sigma)
=
\begin{cases}
    \min\{j \in [n]: \sigma(j) \in [r]\} & \mbox{ if }r \in [n], \\
    n+1 & \mbox{ if }r=0\, .
\end{cases}
\]

As discussed in Section \ref{sec:fdds}, the proof of Theorem \ref{thm:main} relies on establishing a change of measure, (\ref{eq:basic_meas_change}), which relates the size-biased random re-ordering of the positive entries of the degree sequence of $\rT_n$, and \iid samples from the offspring distribution. The proofs of Theorems \ref{thm:hairy_4} and \ref{thm:hairy_2} rely on establishing a similar change of measure, which is a generalisation of (\ref{eq:basic_meas_change}) to the situation where instead of an \iid sequence, the first $r$ elements are non-zero and are fixed in advance; the whole sequence is conditioned to have sum $n-1$; and we consider the first $m$ elements of the size-biased random reordering of the sequence. Specifically, let $m,n,r,s\in \Z_{\ge 0}$ with $m,r,s <n$, and $\mu$ be a distribution on $\Z_{\ge 0}$. For $k_1,\dots,k_m\in\N$, we define 
\begin{align} \label{eq:meas_change}
&\Theta_\mu(k_1,\dots,k_m) = \Theta_\mu^{n,r,s}(k_1,\dots,k_m)\nonumber\\
& = \frac{\p{X_{m+1}+\dots + X_{n-r} = n - 1 - s - \sum_{i=1}^mk_i}}{\p{X_1+\dots + X_{n-r} = n - 1 - s}}\cdot \left(\E{X_1}\right)^m\cdot \prod_{i=1}^m\frac{n-r-i+1}{n-1-\sum_{j=1}^{i-1}k_j}
\end{align} 
if $k_1+\dots + k_m \le n - 1 - s$, and otherwise $\Theta_\mu(k_1,\dots, k_m) = 0$, where $(X_i, i \ge 1)$ are \iid random variables with distribution~$\mu$. We note that when $r = s = 0$, and $\mu$ is a critical offspring distribution, we recover (\ref{eq:basic_meas_change}).

\begin{prop}\label{prop:sizebias_swole}
Fix $n,r,s \in \Z_{\ge0}$ with $r,s < n$, and $d_1,\ldots,d_r \in \N$ with $\sum_{i=1}^r d_i=s$. Let $\mu$ be a distribution on $\Z_{\ge 0}$ and $(X_i,i \ge 1)$ be \iid random variables with distribution~$\mu$.  Further, let 
\[
N=N_{n,r}=|\{i \in \{r+1,\ldots,n\}: X_i > 0\}|.
\] 
Let $\vec{Z}=(Z_1,\ldots,Z_n)=(d_1,\ldots,d_r,X_{r+1},\ldots,X_n)$, and conditionally given $\vec{Z}$, let $\Sigma=\Sigma_{\vec{Z}}$ be given by \eqref{eq:sizebias_def}. Finally, let $(\overline{X}_i,i \in [n])$ be \iid samples from the size-biased distribution of $X_1$. Suppose that $\E{X_1}<\infty$. Then for any $m \in [n-r]$ and any function $f:\N^m \to \R$, if $\p{X_{r+1}+\ldots+X_n=n-1-s)}>0$, then
\begin{align}
\begin{split} \label{eq:fzhat}
&\E{f\left( Z_{\Sigma(1)},\ldots , Z_{\Sigma(m)}\right)\I{N \ge m}\I{\tau_r(\Sigma) > m}~\bigg|~X_{r+1}+\ldots+X_n = n-1-s}\\
& \qquad = \E{ f(\overline{X}_1,\ldots,\overline{X}_m)\Theta_\mu(\overline{X}_1,\ldots,\overline{X}_m)}\, ,
\end{split}
\end{align}
where $\Theta_\mu(\overline X_1,\dots \overline X_m) = \Theta_\mu^{n,r,s}(\overline X_1,\dots \overline X_m)$ is as in (\ref{eq:meas_change}). 
\end{prop}

We observe that when $r \neq 0$, $\tau_r(\Sigma) > m$ implies that $N \ge m$ because all positive entries occur before zero-valued entries in the size-biased random reordering. However, when $r = 0$, the former event is vacuously true for all $m\in [n]$, but we still enforce that $N \ge m$ in \eqref{eq:fzhat}. It follows that Proposition~\ref{prop:measure_change_basic} is the special case when $r=0$, $s = 0$ and $X_1, X_2,\dots$ are \iid samples from the offspring distribution $\mu$. 

\begin{proof} 
In this proof, for $n \ge 1$, and $r \ge 1$, we let 
\[
[n]_r = \{(n_1,\dots, n_r)\in \{1,\dots, n\}^r~:~ n_i \neq n_j \text{ for all } i\neq j\}.
\]
Furthermore, for a set $A$ we write $A_r$ for the set of ordered sequences $(s_1,\ldots,s_r)$ of $r$ distinct elements of $A$. We also let $\mu_i=\p{X_1=i}$ for $i \in \Z_{\ge 0}$.

We first prove the proposition assuming that $\mu_0= 0$; we will later generalise this by conditioning on the number of non-zero entries of $\vec{Z}$ and sampling a size-biased re-ordering of only these entries. When $\mu_0=0$, we have $\p{N=n}=1$, so the indicator $\I{N \ge m}$ in \eqref{eq:fzhat} equals $1$ and may be ignored.

For $\sigma \in S_n$, we write $\vec{Z}_{\sigma}=(Z_{\sigma(1)},\dots, Z_{\sigma(n)})$ and $\sigma^{-1}[r]=(\sigma^{-1}(1),\dots,\sigma^{-1}(r))$. Observe that for $m \in [n-r]$, we have the equality of events
\[
\{\tau_r(\Sigma)>m\}=\left\{\Sigma^{-1}[r]\in ([n]\backslash[m])_r\right\}\, .
\]
It is thus useful to determine the law of $(\vec{Z}_{\Sigma}, \Sigma^{-1}[r])$. Note that for any 
$\vec{k}=(k_1,\dots, k_n)\in \N^n$ and $\vec{j}=(j_1,\dots,j_r)\in [n]_r$, if $(\vec{k},\vec{j})$ is in the support 
of $(\vec{Z}_{\Sigma}, \Sigma^{-1}[r])$ then $k_{j_i}=d_i$ for each $i\in [r]$.
For such $(\vec{k}, \vec{j})$, 
\begin{align}\label{eq:append_perms}
\p{\vec{Z}_{\Sigma}=\vec{k}, \Sigma^{-1}[r]=\vec{j} } 
& = \sum_{\sigma\in \cS_n: \sigma^{-1}[r]=\vec{j}}\p{\vec{Z}_{\sigma}=\vec{k}, \Sigma = \sigma}\nonumber\\
& = \prod_{i=1}^{n}\frac{k_i}{\sum_{j=i}^{n}k_j}\cdot \sum_{\sigma\in \cS_n: \sigma^{-1}[r]=\vec{j}}\p{\vec{Z}_{\sigma}=\vec{k}}.
\end{align}
Since we fixed $\sigma^{-1}[r]$, the sum (\ref{eq:append_perms}) ranges over exactly $(n-r)!$ elements of $\cS_{n}$ and each term of the sum is equal to
\[
\prod_{j \in [n]\setminus \{j_1,\ldots,j_r\}}\mu_{k_j}.
\]
Hence, for any $\vec{k}\in \N^n$ and $\vec{j}\in [n]_r$, 
\begin{equation}\label{eq:append_labeled}
\p{\vec{Z}_{\Sigma}=\vec{k}, \Sigma^{-1}[r]=\vec{j} } = (n-r)!\left(\prod_{i=1}^r\I{k_{j_i}=d_i}\right) \left(\prod_{j \in [n]\setminus \{j_1,\ldots,j_r\}} \mu_{k_j}\right)\left(\prod_{i=1}^{n}\frac{k_i}{\sum_{j=i}^{n}k_j}\right).
\end{equation}
Now, fix $m\in [n-r]$  and $k_1,\dots,k_m\in \N$. Note that it suffices to prove \eqref{eq:fzhat} when $f:\N^m \rightarrow \R$ has the form
\begin{equation}\label{eq:specialf}
f(z_1,\dots, z_m) = \prod_{i=1}^m\I{ z_i = k_i},
\end{equation}
so we now restrict our attention to this case. Since 
\[
\sum_{i \in [n]} Z_{\Sigma(i)} = \sum_{i=1}^r d_i + 
\sum_{i=r+1}^n X_i
= s+\sum_{i=r+1}^n X_i\, ,
\]
for any $k_1,\ldots,k_m \in \N$, by summing over the possible values of $Z_{\Sigma(m+1)},\dots,Z_{\Sigma(n)}$, we can use \eqref{eq:append_labeled} to find that
\begin{align}\label{eq:append_last_line}
    &\p{(Z_{\Sigma(1)},\dots,Z_{\Sigma(m)})=(k_1,\dots,k_m), \tau_r(\Sigma)>m, \sum_{i=r+1}^nX_i=n-1-s }\\
    &= \sum_{(k_{m+1},\dots, k_{n})\in \N^{n-m} }\sum_{\vec{j}\in ([n]\backslash[m])_r }\I{\sum_{i=1}^nk_i=n-1}\p{\vec{Z}_\Sigma=(k_1,\dots,k_n), \Sigma^{-1}[r]=\vec{j}}\nonumber\\
    &=(n-r)!\left(\prod_{i=1}^m\mu_{k_i}\right)\left(\prod_{i=1}^m\frac{k_i}{n-1-\sum_{j=1}^{i-1} k_i}\right) \nonumber\\
    & \cdot \hspace{-4pt}\sum_{\substack{(k_{m+1},\dots,k_n)\in \N^{n-m}\\\vec{j}\in ([n]\backslash[m])_r }} \!\! \I{\sum_{i=1}^n k_i=n-1}\left(\prod_{i=1}^r\I{k_{j_i}=d_i}\right)\left(\prod_{i\in ([n]\setminus [m])\setminus\{j_1,\dots,j_r\}} \hspace{-0.5cm}\mu_{k_i}\right)\!\left(\prod_{i=m+1}^{n}\frac{k_i}{\sum_{j=i}^{n}k_{j}}\right). \nonumber
\end{align}
Using that $k\mu_k = \p{\overline X_1 = k}\E{X_1}$ for all $k\in \N$, writing $n'=n-m$, and re-indexing the above sum, this yields that \eqref{eq:append_last_line} is equal to 
\begin{align}
    &\p{(\overline X_1,\dots, \overline X_{m}) = (k_1,\dots, k_{m})}\E{X_1}^m \left(\prod_{i=1}^{m}\frac{n-r-i+1}{n-1 - \sum_{j=1}^{i-1}k_j}\right) (n'-r)!\nonumber\\
    &\quad \cdot\hspace{-0.5cm}\sum_{\substack{(k'_1,\dots,k'_{n'})\in \N^{n'}\\\vec{j}\in [n']_r }}  \I{\sum_{i=1}^{n'} k'_i=n-1-\sum_{i=1}^m k_i} \left(\prod_{i=1}^r\I{k'_{j_i}=d_i}\right)\left(\prod_{i\in [n']\setminus\{j_1,\dots,j_r\}}\mu_{ k'_i} \right)%\nonumber\\
    \cdot \prod_{i\in [n']}\frac{k'_i}{\sum_{j=i}^{n'}k'_{j}}.\nonumber
\end{align}
Now, define $\vec{Z}'=(d_1,\dots ,d_r,X_{r+1},\dots,X_{n-m})$ and, conditionally given $\vec{Z}'$, let $\Sigma'=\Sigma_{\vec{Z}'}$ be given by \eqref{eq:sizebias_def}. Applying \eqref{eq:append_labeled} to $\vec{Z}'$ and $\Sigma'$, we thus find that   \eqref{eq:append_last_line} equals 
\begin{align*}
    &\p{(\overline X_1,\dots, \overline X_{m}) = (k_1,\dots, k_{m})}  \prod_{i=1}^{m}\left(\frac{n-1-i+1}{n-1 - \sum_{j=1}^{i-1}k_j}\E{X_1}\right)\\
    &\quad \cdot\sum_{\substack{(k'_1,\dots,k'_{n'})\in \N^{n'}\\\vec{j}\in [n']_r }}  \I{\sum_{i=1}^{n'} k'_i=n-1-\sum_{i=1}^m k_i} \p{\vec{Z}'_{\Sigma'}=\vec{k}, (\Sigma')^{-1}[r]=\vec{j} } \\
    &=\p{(\overline X_1,\dots, \overline X_{m}) = (k_1,\dots, k_{m})} \prod_{i=1}^{m}\left(\frac{n-1-i+1}{n-1 - \sum_{j=1}^{i-1}k_j}\E{X_1}\right)\\
    &\quad \cdot\p{\sum_{i=1}^{n'} Z'_{\Sigma'(i)}=n-1-\sum_{i=1}^m k_i} . 
\end{align*}
Finally, since the sum of the entries of  $\vec{Z}'_{\Sigma'}$ is unaffected by the random reordering and is the same as $s+\sum_{i=1}^{n'-r+m}X_i=s+\sum_{i=1}^{n-(r+m)}X_i$, we deduce that \eqref{eq:append_last_line} equals  
\begin{align*}
    &\p{(\overline X_1,\dots, \overline X_{m}) = (k_1,\dots, k_{m})}  \prod_{i=1}^{m}\left(\frac{n-1-i+1}{n-1 - \sum_{j=1}^{i-1}k_j}\E{X_1}\right)\\
    &\qquad\cdot\p{\sum_{i=1}^{n-(r+m)}X_i = n-1-s-\sum_{i=1}^{m}k_i}.
\end{align*}
Dividing the above expression by $\p{\sum_{i=1}^{n-r}X_i = n-1-s}$ yields the statement when $\mu_0 = 0$, in the special case that $f$ has the form given in  \eqref{eq:specialf}, and thus for general $f$. 

For the general case with $\mu_0 > 0$, we let $p = 1 - \mu_0.$ Further, we let $\mathbf{X}_1, \mathbf{X}_2,\dots$ be \iid copies of $X_1$ conditioned to be positive. Notice that $\E{\mathbf{X}_1} = p^{-1}\E{X_1}$ and that the size-biased distributions of $\mathbf{X}_1$ and of $X_1$ are identical. We let $\widehat{\mathbf{X}}_1,\widehat{\mathbf{X}}_2,\dots$ denote \iid samples from the size-biased distribution of $\mathbf{X}_1$. Finally, fix $m' \ge 0$ and define 
\[
\vec{\mathbf{Z}}'=(\mathbf{Z}'_1,\dots,\mathbf{Z}'_{m'+r})=(d_1,\dots, d_r, \mathbf{X}_{r+1},\dots, \mathbf{X}_{r+m'})\, , 
\] 
and conditionally given $\vec{\mathbf{Z}}'$, let $\Sigma'=\Sigma_{\vec{\mathbf{Z}}'}$ be given by \eqref{eq:sizebias_def}.

Now fix $m\in [n-r]$ with $m\le m'$. For $k_1,\dots, k_{m}\in \N$, we have that 
\begin{align}\label{eq:append_one}
    &\p{(Z_{\Sigma(1)},\dots,Z_{\Sigma(m)})=(k_1,\dots,k_m), \tau_r(\Sigma)>m, \sum_{i=r+1}^nX_i=n-1-s ~\bigg|~N = m'}\nonumber\\
    &=\p{(\mathbf{Z}'_{\Sigma'(1)},\dots,\mathbf{Z}'_{\Sigma'(m)}) = (k_1,\dots, k_{m}),~ \tau_r(\Sigma') > m,~\sum_{i=r+1}^{m'}\mathbf{X}_i = n - 1 - s}.
\end{align}
By the proof of the case where $\mu_0 = 0$, if $k_1+\dots+ k_m \le n - 1 - s$ this is equal to 
\begin{align*}
    &\p{(\overline{\mathbf{X}}_1, \dots, \overline{\mathbf{X}}_m) = (k_1,\dots, k_m)}\\
    &\quad\cdot \p{\sum_{i=m+1}^{m'}\mathbf{X}_i = n - 1 - s - \sum_{i=1}^{m}k_i} \prod_{i=1}^m\left(\frac{m'-i+1}{n - 1 -\sum_{j=1}^{i-1}k_j}\E{\mathbf{X}_1}\right)\, ,
\end{align*}
and otherwise is equal to $0$. For the remainder of the proof we may thus assume that $k_1 + \dots + k_m \le n - 1 - s$. Since $\overline{\mathbf{X}}_1 \eqdist \overline X_1$ and $\E{\mathbf{X}_1} = p^{-1}\E{X_1}$ this is in turn equal to 
\begin{align}\label{eq:append_two}
    &\p{(\overline{X}_1, \dots, \overline{X}_m) = (k_1,\dots, k_m)}\nonumber\\
    &\quad\cdot \frac{1}{p^m}\p{\sum_{i=m+1}^{m'}\mathbf{X}_i = n - 1 - s - \sum_{i=1}^{m}k_i}\prod_{i=1}^m\left(\frac{m'-i+1}{n - 1 -\sum_{j=1}^{i-1}k_j}\E{X_1}\right).
\end{align}
It then follows from (\ref{eq:append_one}) and (\ref{eq:append_two}) that
\begin{align}\label{eq:append_three}
    &\p{(Z_{\Sigma(1)},\dots,Z_{\Sigma(m)})=(k_1,\dots,k_m) , N\ge m, \tau_r(\Sigma)>m, \sum_{i=r+1}^nX_i=n-1-s} \nonumber\\
    &=\p{(\overline{X}_1, \dots, \overline{X}_m) = (k_1,\dots, k_m)}\nonumber\\
    &\cdot \sum_{m'= m}^{n-r}\frac{\p{N = m'}}{p^m}\p{\sum_{i=m+1}^{m'}\mathbf{X}_i = n - 1 - s - \sum_{i=1}^{m}k_i}\prod_{i=1}^m\left(\frac{m'-i+1}{n - 1 -\sum_{j=1}^{i-1}k_j}\E{X_1}\right).
\end{align}
Notice now that $N \eqdist \mathrm{Binomial}(n-r, p)$. So using the change of variable $\ell = m' - m$ and letting $M$ be a Binomial$(n-(r+m), p)$, by routine algebra we obtain that \eqref{eq:append_three} equals 
\begin{align*}
    &\sum_{\ell = 0}^{n-(r+m)}\p{M = \ell}\hspace{-3.5pt}\p{\sum_{i=m' - \ell + 1}^{m+\ell}\hspace{-8pt}\mathbf{X}_i = n - 1 - s - \sum_{i=1}^{m}k_i} \cdot\prod_{i=1}^m\left(\frac{n-r-i+1}{n-1-\sum_{j=1}^{i-1}k_i}\E{X_1}\right)\\
    &=\sum_{\ell = 0}^{n-(r+m)}\p{M = \ell}\p{\sum_{i=1}^{\ell}\mathbf{X}_i = n - 1 - s - \sum_{i=1}^{m}k_i} \cdot\prod_{i=1}^m\left(\frac{n-r-i+1}{n-1-\sum_{j=1}^{i-1}k_i}\E{X_1}\right)\\
    &= \p{\sum_{i=1}^{M}\mathbf{X}_i = n - 1 - s - \sum_{i=1}^m k_i}\prod_{i=1}^m\left(\frac{n-r-i+1}{n-1-\sum_{j=1}^{i-1}k_i}\E{X_1}\right).
\end{align*}
Since $\sum_{i=1}^{M}\mathbf{X}_i \eqdist \sum_{i=1}^{n-(r+m)}X_i$, dividing the above expression (which is equal to (\ref{eq:append_three})) by $\p{\sum_{i=1}^{n-r}X_i = n-1-s}$ yields the result for the special case that $f$ has the form given in  \eqref{eq:specialf}, and thus for general $f$. 
\end{proof}

The next proposition gives conditions under which the change of measure $\Theta^{n,r,s}_\mu$ appearing in (\ref{eq:meas_change}) is asymptotically unimportant in the specific case when $m = \Theta(\sqrt{n})$ and $(X_i, i \ge 1)$ are \iid samples from the offspring distribution $\mu$ conditioned to yield a displacement vector such that $\max_{1\le j \le X_i}|Y_{X_i,j}| \le \gamma n^{1/(4-\eta)}$.  This then allows us to use the measure change in the proofs of Theorems \ref{thm:hairy_4} and \ref{thm:hairy_2}.  

\begin{lem}\label{lem:conv_measure_change_hairy}

Let $\mu$ be a critical offspring distribution with variance $\sigma^2\in (0,\infty)$, and let $\nu = (\nu_k)_{k \ge 1}$ be such that [\ref{a1}] holds and [\ref{a3}] holds for a given measure $\pi$ with $\eta\in [0,2)$.  
Fix $\gamma > 0$. Let $\xi$ denote a random variable with distribution $\mu$, and for $n \ge 1$ let $\xi^{n}$ be distributed as $\xi$, conditioned to not yield a displacement vector with $\max_{1\le i \le \xi^n}|Y_{\xi^{n}, i}| > \gamma n^{1/(4-\eta)}$. Further, let $\mu^n$ denote the distribution of $\xi^n$, and let $\bar\xi_1^{n}, \bar\xi_2^{n},\dots$ be \iid samples from the size-biased law of $\xi^n$. 

Finally, fix $\varepsilon \in (0,1/6)$ and let $(r_n)_{n\ge 1}$ and $(s_n)_{n\ge 1}$ be sequences such that for all $n\ge 1$, $r_n < n^\varepsilon$, $s_n < n^{1/3+ \varepsilon}$ and $n-1-s_n$ is in the support of $\sum_{i=r_n+1}^{n}\xi_i^{n}$.

Suppose that $m =\Theta(\sqrt{n})$. Then as $n \rightarrow\infty$,
\begin{equation} \label{eq:conv_measure_hairy}
\Theta^{n,r_n, s_n}_{\mu^n}(\bar \xi_1^{n},\dots, \bar \xi_m^{n})\convprob 1,
\end{equation}
and $(\Theta^{n,r_n, s_n}_{\mu^n}(\bar\xi_1^{n},\dots, \bar\xi_m^{n}))_{n\ge 1}$ is a uniformly integrable sequence of random variables.
\end{lem}

The proof of Lemma \ref{lem:conv_measure_change_hairy} is very similar to that of Lemma \ref{lem:conv_measure_change_basic}. However, in this case instead of the standard local central limit theorem, we will require a quantitative local central limit theorem in order to get uniform estimates on local probabilities for the family of random variables $\{\xi^{n}, n \ge 1\}$. 

\begin{lem}\label{lem:local_applied_move}
Let $\mu$ be a critical offspring distribution with variance $\sigma^2\in (0,\infty)$, and let $\nu = (\nu_k)_{k \ge 1}$ be such that [\ref{a1}] holds and [\ref{a3}] holds for a given measure $\pi$ with $\eta\in [0,2)$. Let $\gamma > 0$. Further, let $\xi$ denote a random variable with distribution $\mu$ and for $n \ge 1$ let $\xi^n=(\xi^{n}_i, i \ge 1)$ be \iid copies of $\xi$ each conditioned to satisfy $\{\max_{1\le i \le \xi_{j}^n}|Y_{\xi_j^{n}, i}| \le \gamma n^{1/(4-\eta)}\}$.
Then there exist $C,N >0$ and $M$ such that for all $m,n>N$,
\[
\sup_{k\in \Z} \left|\sqrt{m}\p{\sum_{i=1}^m\xi^{n}_i=k}-\phi\left( \frac{k-m\E{\xi^{n}_1}}{\sqrt{\V{\xi^{n}_1}m}}\right) \right|\leq \frac{C}{\sqrt{m}},
\]
where $\phi(t)=e^{-t^2/2}$ is the standard normal density.
\end{lem}

This lemma is immediate from Lemma \ref{lem:local_move} as soon as we show that the family $\{\xi^n,~n\ge 1\}$ satisfies the conditions of that lemma. This is verified in Lemmas \ref{lem:gcd_move} and \ref{lem:moments_gamma_hairy_new}. 

\begin{lem}\label{lem:gcd_move}
For all $n$ sufficiently large, the support of $\xi^{n}$ has greatest common divisor~$1$. 
\end{lem}

\begin{proof}
By assumption, the support of $\xi$ has greatest common divisor $1$, so we can find an $M$ such that the greatest common divisor of the support of $\xi$ restricted to $\{0,\dots, M\}$ is~$1$. Since $\gamma n^{1/(4-\eta)} > M$ for $n$ sufficiently large, the result follows.
\end{proof}

\begin{lem}\label{lem:moments_gamma_hairy_new} 
As $n\rightarrow \infty$,
\begin{equation}\label{eq:conv_hairy_moment}
\E{\left(\xi^{n}\right)^j} \rightarrow \E{\xi^j} \quad \text{ for } j = 1,2,3
\end{equation}
and 
\begin{equation}\label{eq:precise_hairy_moment}
\left|\E{\xi^{n}} - 1\right| = O(n^{-2/3}).
\end{equation}
\end{lem}

\begin{proof}
For  $j \in \{1,2,3\}$ we have
\begin{align*}
\E{\left(\xi^{n}\right)^j} 
& = \sum_{k = 1}^\infty k^j\p{\xi^{n} = k}\\
& \le \sum_{k=1}^\infty k^j \frac{\p{\xi = k}}{\p{\max_{1\le i \le \xi}|Y_{\xi, i}| \le \gamma n^{1/(4-\eta)}}} = \left(1 + O\left(\frac{1}{n}\right)\right)\E{\xi^j},
\end{align*}
where the final equality follows by assumption [\ref{a3}]. By the bounded convergence theorem, as $n\rightarrow \infty$, 
\[
\E{\left(\xi^{n}\right)^j} \ge \E{\xi^j} - \E{\xi^j\I{\max_{1\le i \le \xi}|Y_{\xi, i}| > \gamma n^{1/(4-\eta)}}} \rightarrow \E{\xi^j},
\]
where we have used assumption [\ref{a3}] again. (\ref{eq:conv_hairy_moment}) follows. 

To get the more precise lower bound for $j = 1$ in (\ref{eq:precise_hairy_moment}), observe that 
\[
\E{\xi\I{\max_{1\le i \le \xi}|Y_{\xi,i}| > \gamma n^{1/(4-\eta)}}} \le n^{1/3}\p{\max_{1\le i \le \xi}|Y_{\xi, i}| > \gamma n^{1/(4-\eta)}} + \E{\xi\I{\xi > n^{1/3}}}.
\]
The first term on the right-hand side of this inequality is $O(n^{-2/3})$ by [\ref{a3}]. Also, $\E{\xi^3}<\infty$ and so the second term is also $O(n^{-2/3})$, thus establishing (\ref{eq:precise_hairy_moment}).
\end{proof}

The last tool that we need to prove Lemma \ref{lem:conv_measure_change_hairy}  is an upper bound on the total variation distance between $\bar\xi_1^{n}$ and $\bar\xi$ where $\bar\xi$, a sample from the size-biased law of $\xi$.  

\begin{lem}\label{lem:bound_dTV_move}
Let $X$ be a random variable taking values in  $\N$ such that $\E{X} \ge 0$ and $\E{X^3}<\infty$. Let $(\cE_n)_{n\ge 1}$ be a sequence of events with  $\p{\cE_n} = 1-O(1/n)$. Let $X_n$ be distributed as $X$ conditional on $\cE_n$. Let $\overline{X}_n$ have the size-biased law of $X_n$ and let $\overline{X}$ have the size-biased law of $X$. Then,
\[
d_\mathrm{TV}(\overline X_n, X) = \frac{1}{2}\sum_{k=1}^\infty \left|\p{\overline X_n = k} - \p{\overline X = k}\right| =O(n^{-2/3}).
\]
\end{lem}

\begin{proof}
By definition, 
\begin{align}
\begin{split} \label{eq:size_bias_def_move}
&\p{\overline{X}_n=k}=\frac{k \p{X=k, \cE_n}}{\E{X\I{\cE_n}}}, \text{ and } ~\p{\overline{X}=k}=\frac{k \p{X=k}}{\E{X}}.
\end{split}
\end{align}
Since $\E{X^3} < \infty$ we have that $\p{X > n^{1/3}} = o(n^{-1})$ as $n\rightarrow \infty$ and so Hölder's inequality yields that
\[
\E{X\I{X > n^{1/3}}} \le \E{X^3}^{1/3}\p{X > n^{1/3}}^{2/3} = o(n^{-2/3}).
\]
Next, 
\[
{\E{X\I{\cE^c_n}}}\le n^{1/3}\p{\cE^c_n}+\E{X\I{X>n^{1/3}}}=O(n^{-2/3}),
\]
so that $\E{X\mathbf{1}_{\cE_n}} = \E{X} + O(n^{-2/3})$ and the difference between the denominators in~\eqref{eq:size_bias_def_move} is $O(n^{-2/3})$. It follows that
\begin{align*}
     &\sum_{k=1}^\infty \left|\p{\overline X_n = k} - \p{\overline X = k}\right| \\ 
     &\quad= \sum_{k=1}^\infty \left|
      \frac{k\p{X = k,~ \cE_n}}{\E{X\I{\cE^c_n}}}
      - \frac{k\p{X = k}}{\E{X}}
     \right|
     \\
    &\quad\le \frac{1}{\E{X\I{\cE_n}}}\left(\sum_{k=1}^\infty k\p{X = k, \cE_n^c} \right) + O(n^{-2/3})\\
    &\quad\le \frac{1}{\E{X\I{\cE_n}}}\left(\sum_{k=1}^{n^{1/3}} k\p{X = k, \cE_n^c} + \E{X\I{X> n^{1/3}}}\right)+ O(n^{-2/3})\\
    &\quad\le \frac{n^{1/3}\p{\cE_n^c}}{\E{X\I{\cE_n}}} + \frac{\E{X\I{X > n^{1/3}}}}{\E{X\I{\cE_n}}} + O(n^{-2/3}).
\end{align*}
The first term in the last line is also is $O(n^{-2/3})$ since $\p{\cE_n^c} = O(1/n)$. 
\end{proof}

This lemma has the following corollary. 

\begin{cor}\label{cor:dtv_cond}
Let $\mu$ be a critical offspring distribution with variance $\sigma^2\in (0,\infty)$, and let $\nu = (\nu_k)_{k \ge 1}$ be such that [\ref{a1}] holds and [\ref{a3}] holds for a given measure $\pi$ with $\eta\in [0,2)$.  
Fix $\gamma > 0$. Let $\xi$ denote a random variable with distribution $\mu$ and let $\bar\xi_1, \bar\xi_2,\dots$ be \iid samples from the size-biased law of $\xi$. For $n \ge 1$ let $\xi^{n}$ be distributed as $\xi$, conditioned to not yield a displacement vector with $\max_{1\le i \le \xi^n}|Y_{\xi^{n}, i}| > \gamma n^{1/(4-\eta)}$. Further, let $\mu^n$ denote the distribution of $\xi^n$, and let $\bar\xi_1^{n}, \bar\xi_2^{n},\dots$ be \iid samples from the size-biased law of $\xi^n$. 
Then for $m = \Theta(\sqrt{n})$,
\[
d_{\mathrm TV}((\bar\xi_1^{n},\dots, \bar\xi_m^{n}),(\bar\xi_1,\dots, \bar\xi_m))=O(n^{-1/6}).
\]
\end{cor}

\begin{proof}
By [\ref{a3}], $\bar\xi_1^{n}$ is obtained from $\bar\xi_1$ by conditioning on an event which occurs with probability $1 - O(1/n)$. Therefore, by Lemma \ref{lem:bound_dTV_move}, the total variation distance between $\bar\xi^{n}_1$ and $\bar\xi_1$ is $O(n^{-2/3})$. Since $m = \Theta(\sqrt{n})$, the conclusion follows.
\end{proof}

We now prove Lemma \ref{lem:conv_measure_change_hairy}. Since the proof is very similar to that of Lemma \ref{lem:conv_measure_change_basic} we will be brief. 

\begin{proof}[Proof of Lemma \ref{lem:conv_measure_change_hairy}]
As in the proof of Lemma \ref{lem:conv_measure_change_basic}, we may assume that there exists $t > 0$ such that $m/\sqrt{n}\rightarrow t$ as $n \rightarrow \infty.$

Suppose that $k_1,\dots, k_m\in \Z_{\ge 0}$. Then by almost identical techniques to those used to prove (\ref{eq:basic_meas_need}) (replacing the local central limit theorem by Lemma \ref{lem:local_applied_move}), we obtain that 
\begin{align}
    \label{eq:hairy_num_dem}
    &\frac{\p{\sum_{i=(r_n +m)+1}^n\xi_i^{n} = n - 1 - s_n - \sum_{i=1}^mk_i}}{\p{\sum_{i=r_n +1}^n \xi_i^{n} = n - 1 - s_n}}\nonumber\\
    &= \exp\left(-\left(\left(\frac{1 + s_n - r_n + m\sigma^2 + \sum_{i=1}^m(k_i - (1+\sigma^2))}{\sqrt{2\sigma^2(n - (r_n +m))}}\right)^2 +o(1)\right)\right)+ o(1).
\end{align}

Recall that, for $i\in [m]$,  $\bar\xi_i$ is  sample from the size-biased distribution of $\xi$. We claim that instead of substituting $\bar\xi_1^{n},\dots, \bar\xi^{n}_m$ in the place of $k_1,\dots,k_m$ we can substitute $\bar\xi_1,\dots,\bar\xi_m$.  Indeed, by Corollary~\ref{cor:dtv_cond}, the total variation distance between $\bar\xi_1^{n},\dots, \bar\xi^{n}_m$ and $\bar\xi_1,\dots,\bar\xi_m$ tends to $0$ as $n\rightarrow \infty$. Therefore, by (\ref{eq:conv_gen}), we obtain that (\ref{eq:hairy_num_dem}) tends to $\exp(-(t^2\sigma^2)/2)$ in probability as $n\rightarrow\infty$. This convergence is analogous to (\ref{eq:conv_gen}) in the proof of Lemma \ref{lem:conv_measure_change_basic}. 

It remains to establish an analogue of (\ref{eq:not_hairy_prod}), i.e.,
\begin{equation}\label{eq:hairy_prod}
\prod_{i=1}^m\left(\frac{n - r_n - i+1}{n - 1 - \sum_{j=1}^{i-1}\bar\xi^n_j}\E{\xi^{n}}\right)=\E{\xi^{n}}^m\prod_{i=1}^m\left(\frac{n - r_n - i+1}{n - 1 - \sum_{j=1}^{i-1}\bar\xi_j}\right)\convprob \exp\left(\frac{t^2\sigma^2}{2}\right),
\end{equation}
as $n\rightarrow \infty$. By Lemma \ref{lem:moments_gamma_hairy_new}, 
\[
\E{\xi^{n}} = \E{\xi} + O(n^{-2/3}) = 1 + O(n^{-2/3}),
\] 
and so, since $m = (1+o(1))t\sqrt{n}$, we obtain that $\E{\xi^{n}}^m = 1 + o(1).$ Therefore (\ref{eq:hairy_prod}) follows from (\ref{eq:not_hairy_prod}).

We now prove uniform integrability of the family $(\Theta_{\mu^n}^{n, r_n, s_n}(\bar\xi^{n},\dots, \bar\xi^{n}_m))_{n \ge 1}$. Again, by the generalised Scheffé lemma \cite[Theorem 5.12]{kallenberg}, since $\Theta^{n, r_n, s_n}_{\mu^n}(\bar\xi^{n}_1,\dots, \bar\xi_m^{n})\convprob 1$ it suffices to show that $\E{\Theta_{\mu^n}^{n, r_n, s_n}(\bar\xi_1^{n},\dots, \bar\xi_m^{n})}\rightarrow 1$ as $n\rightarrow \infty$. By Proposition \ref{prop:sizebias_swole} with $f \equiv 1$,
\begin{equation}\label{eq:unif_int_gen}
\E{\Theta_{\mu^n}^{n, r_n, s_n}(\bar\xi_1^{n},\dots, \bar\xi_m^{n})} = \p{ N\ge m~\tau_{r_n}(\Sigma) >m ~\bigg|~ \sum_{i=r_n+1}^{n}\xi_i^{n} = n -1-s_n},
\end{equation}
where $\Sigma = \Sigma_{\vec{Z}}$ with $\vec{Z} = (Z_1,\dots, Z_n) = (d_1,\dots, d_{r_n}, \xi^{n}_{r_n+1},\dots, \xi^{n}_n)$ such that $d_1 = s_n$, and $d_2,\dots d_{r_n}=0$. (Indeed, any fixed choice of $d_1,\dots, d_{r_n}$ with $\sum_{i=1}^{r_n}d_i = s_n$ would suffice.) To see that the probability on the right-hand side of (\ref{eq:unif_int_gen}) tends to $1$ as $n\rightarrow \infty$, first note that 
$N \eqdist$ Binomial$(n-r_n, 1- \mu_0)$ where $r_n < n^\varepsilon$. So even after conditioning on the event $\{\sum_{i=r_n+1}^{n}\xi^n_i = n - 1 - s_n\}$, which occurs with probability $O(n^{-1/2})$, there are $(1+o_\bP(1))(n-r_n)(1-\mu_0)$ non-zero entries of $(\xi^n_{r_n+1},\dots, \xi^n_n)$. Therefore, to prove uniform integrability it remains to show that $\tau_{r_n}(\Sigma) = \omega_\bP(\sqrt{n}).$ 
     
To see this, observe that for any $k\in [n]$, 
\[
\p{\tau_{r_n}(\Sigma) = k+1 ~|~(Z_{\Sigma(1)},\dots, Z_{\Sigma(k)}),~ \tau_{r_n}(\Sigma) \ge k} = \frac{s_n}{\sum_{i = k+1}^nZ_{\Sigma(i)}}.
\]
Since $\vec{Z}$ contains $(1+o_\bP(1))(n-r_n)(1-\mu_0) + 1$ positive entries, this denominator is $(1+o_\bP(1))(n-r_n)(1-\mu_0) + 1$ uniformly over all $k \le m = O(\sqrt{n}),$ and all labeled random reorderings of $\vec{Z}.$ Moreover, since $s_n = o(\sqrt{n})$ by assumption, 
\[
\p{\tau_{r_n}(\Sigma) = k+1~|~ \tau_{r_n}(\Sigma) \ge k} = o(n^{-1/2}),
\]
uniformly across all $k \le m$. The claim follows by summing these probabilities over $k\le m$, since by the above $\p{\tau_{r_n}(\Sigma) > k} = (1-o(n^{-1/2}))^k$, and in particular for $T > 0,$ \begin{equation}\label{eq:christina_lemma}
\p{\tau_{r_n}(\Sigma) > T\sqrt{n}} = (1-o(n^{-1/2}))^{T\sqrt{n}},
\end{equation} 
which tends to $1$ as $n\rightarrow \infty$.
\end{proof}

\subsection{Backstage at the hairy tour}

To control the restrictions of the discrete snake introduced in the proofs of Theorems \ref{thm:main}, \ref{thm:hairy_4} and \ref{thm:hairy_2} we require a couple of technical lemmas. The first of these results shows that if we truncate the displacements of the discrete snake by $n^{1/(4-\eta) - \delta}$, then the global moments agree with assumption [\ref{a1}] in the limit.

Fix $\eta\in (0,2]$, and $\delta\in (0,1/(4-\eta))$. For $n \ge 1$, and $k \ge 1$ let 
\[
Y_k^{n,\delta} = (Y_{k,1}^{n,\delta},\dots, Y_{k,k}^{n,\delta}) = 
\begin{cases}
(Y_{k,1},\dots, Y_{k,k}) & \text{if } \max_{1\le j \le k}|Y_{k,j}| \le n^{1/(4-\eta) - \delta}, \\
0 &\text{else}.
\end{cases}
\]

\begin{lem}\label{lem:moments_truncation_move}
As $n\to \infty$, it holds that 
\[
\left|\E{Y_{\bar\xi, U_{\bar\xi}}^{n,\delta}}\right|=O\left((n^{1/(4-\eta) - \delta})^{1-2(4-\eta)/3}\right)\quad\text{ and }\quad\va \left(Y_{\bar\xi, U_{\bar\xi}}^{n,\delta}\right)\to \beta^2.
\]
\end{lem}

\begin{proof}
First, observe that by H{\"o}lder's inequality, there exists a constant $c > 0$ such that 
\[\p{\max_{1\le i \le \bar\xi}|Y_{\bar\xi,i}| > y} \le \E{\xi^3}^{1/3}\p{\max_{1\le i\le \xi}|Y_{\xi,i}| > y}^{2/3} \le c y^{-2(4-\eta)/3}.\]
Then, by global centering,
\begin{align*}
    \left|\E{Y_{\bar\xi, U_{\bar\xi}}^{n,\delta}}\right| &= \left|\E{ Y_{\bar\xi, U_{\bar\xi}}\I{\max_{1\le i \le \bar\xi}|Y_{\bar\xi,i}| >n^{1/(4-\eta) - \delta}}}\right|\\
    &\le \int_0^\infty \p{\left|Y_{\bar\xi, U_{\bar\xi}}\I{\max_{1\le i \le \bar\xi}|Y_{\bar\xi,i}| >  n^{1/(4-\eta)-\delta}}\right| > y}dy\\
    &\le n^{1/(4-\eta) - \delta}\p{\max_{1\le i \le \bar\xi}|Y_{\bar\xi,i}| > n^{1/(4-\eta) - \delta}} \\
    &+ \int_{n^{1/(4-\eta) - \delta}}^\infty \p{ \max_{1\le i \le \bar\xi}|Y_{\bar\xi,i}| > y}dy\\
    &\le c(n^{1/(4-\eta) - \delta})^{1-2(4-\eta)/3} - \frac{c}{2(4-\eta)/3 - 1}\left[y^{-2(4-\eta)/3 + 1 }\right]_{n^{1/(4-\eta) - \delta}}^\infty\\
    &= O\left((n^{1/(4-\eta) - \delta})^{1-2(4-\eta)/3}\right),
\end{align*}
as claimed. 

As for the variance,
\begin{align*}
    \va \left(Y^{n,\delta}_{\bar\xi, U_{\bar\xi}}\right) &=\E{\left(Y^{n,\delta}_{\bar\xi,U_{\bar\xi}}\right)^2} - \left(\E{Y^{n,\delta}_{\bar\xi,U_{\bar\xi}}}\right)^2\\
    &= \E{Y^2_{\bar\xi, U_{\bar\xi}}\I{\max_{1\le i \le \bar\xi}|Y_{\bar\xi,i}| \le n^{1/(4-\eta) - \delta}}} -\E{Y_{\bar\xi, U_{\bar\xi}}\I{\max_{1\le i \le \bar\xi}|Y_{\bar\xi,i}| \le n^{1/(4-\eta) - \delta}}}^2\\
    &\rightarrow \E{Y_{\bar\xi, U_{\bar\xi}}^2}=\beta^2,
\end{align*}
as $n\rightarrow \infty$, by dominated convergence and the result for the mean.
\end{proof}

The above lemma pertains to snakes where the displacements which are above $n^{1/(4-\eta) - \delta}$ are all set to $0$. The next lemma in this section will help us to understand the asymptotics of the head of the discrete snake where displacements which are \emph{below} $n^{1/(4-\eta) - \delta}$ are set to $0$. More specifically, we present a tail bound for the size of a set of marked vertices in random trees, which we apply in Proposition \ref{prop:ancestral_rel} where the marked vertices pertain to vertices $v\in v(\rT_n)\setminus \partial \rT_n$ for which $\|Y^{(v)}\|_\infty  > n^{1/(4-\eta)-\delta}$.  

\begin{lem}\label{lem:a_no_ganging_up_d_tree}
Let $\bd=(d_1,\dots,d_{n})$ be a degree sequence, fix $\cB\subset[n]$ and write $K = |\cB|$, and $\Delta = \max_{1 \le i \le n}d_i$. Let $B_{\bd}$ be the smallest distance between two vertices in $\cB$ that are ancestrally related in $T_\bd = B(\Pi_\bd)$ (with $B_{\bd}=\infty$ if no vertices in $\cB$ are ancestrally related). Then, for any $b\ge 0$ 
\[
\p{B_{\bd}\leq b}\leq K\left(1-\left(1-\frac{K\Delta}{n-1-b\Delta}\right)^b\right).
\]
\end{lem}

\begin{proof}
It suffices to show the statement for integer $b$ since for general $b$, $\p{B_\bd \le b}=\p{B_\bd \le \lfloor b \rfloor }$ and the upper bound is increasing in $b$. 

Fix a degree sequence and a set $\cB$. Without loss of generality, assume that $\cB=\{n-K+1,\dots, n\}$. 

For $v\in [n]$, let $p(v)$ be the parent of $v$ in $\rT_{\mathbf{d}}$ (with $p(v)=v$ if $v$ is the root of $\rT_{\mathbf{d}}$). Also set $p^0(v)=v$ and recursively for $k\ge 1$ define the $k$-th ancestor of $v$ as $p^k(v)=p(p^{k-1}(v))$. 

We will show that
\begin{equation} \label{eq:distance_large} 
\p{\{p^{1}(n), \dots, p^{b}(n)\}\cap \{n-K+1,\dots, n-1\}=\emptyset }\geq \left( 1-\frac{K\Delta}{n-1 - b\Delta}\right)^b,
\end{equation}
after which the statement follows by symmetry and the union bound. 

We will prove \eqref{eq:distance_large} by induction. To ease notation, write $p^k=p^k(n)$ for $k\ge 0$. For $(i,c)$ such that $i\in[n],c\in [d_i]$ write $\Pi_\bd^{-1}(i,c)$ for the position of $(i,c)$ in $\Pi_\bd$. 

We will define a sequence of $\sigma$-algebras $(\cF_k)_{k  \ge 0}$ such that for each $k\ge 1$, $\cF_k$ is the $\sigma$-algebra generated by the first $k$ ancestors of $n$ and the positions of their corresponding entries in $\Pi_\bd$. Let, $\cF_0=\sigma(\Pi_\bd^{-1}(n,c): c\in [d_n])$ contain the information on the position of vertex $n$ in $\Pi_\bd$.  

If $d_n=0$ and $\{\Pi_\bd^{-1}(n,c): c\in [d_n]\}=\emptyset$ then $n$ is the final vertex in the final path of the line-breaking construction, and the last entry of $\Pi_\bd$ gives its parent. Thus, in this case, we reveal $\Pi_{\bd}(n-1)$ and we have $\Pi_{\bd}(n-1)=(p^1,c')$ for some $c'\in [d_{p^1}]$. Then, we reveal all other entries of the form $(p^1,c), c\in [d_{p^1}]$ in $\Pi_{\bd}$ and this yields $\cF_1$. 
    
If $d_n> 0$, then set $m_0=\min\{\Pi_\bd^{-1}(n,c): c\in [d_n]\}$. If $m_0=1$ then $n$ is the root of $\bT_\bd$ so $p^\ell = n$ for all $\ell \ge 1$, and so we let $\cF_\ell=\cF_0$ for all $\ell\ge 1$.  Otherwise, the entry before the first occurrence of an entry of the form $(n,c), c\in [d_n]$ in $\Pi_\bd$ gives the parent of $n$ so then we obtain $\cF_1$ as follows. We reveal $\Pi_{\bd}(m_0-1)$. In that case $\Pi_{\bd}(m_0-1)=(p^1,c')$ for some $c'\in [d_{p^1}]$.  Secondly, we reveal all other entries of the form $(p^1,c), c\in [d_{p^1}]$ in $\Pi_{\bd}$ and this yields $\cF_1$. 
    
For $k \ge 1$, given $\cF_k$, let 
\[
m_k=\min\{\Pi^{-1}_\bd(p^{k},c):c\in[d_{p^{k}}]\}.
\]
If $m_k=1$ then $p^k$ is the root of $\bT_\bd$ so $p^\ell=p^{k}$ and we take $\cF_\ell=\cF_k$ for all $\ell\ge k$. If $m_k>1$, we obtain $\cF_{k+1}$ as follows. First, we reveal $\Pi_{\bd}(m_k-1)$. In that case $\Pi_{\bd}(m_k-1)=(p^{k+1},c')$ for some $c'\in [d_{p^{k+1}}]$. Secondly, we reveal all other entries of the form $(p^{k+1},c), c\in [d_{p^{k+1}}]$ in $\Pi_{\bd}$ and this yields $\cF_{k+1}$.
    
Now, observe that, for $k\ge 0$, given $\cF_k$, the unrevealed entries of $\Pi_\bd$ occur in an order given by a uniformly random permutation. So given $\cF_k$ if $m_k>1$ the $k$-th ancestor of $n$ is the first coordinate of a uniformly random sample from 
\[
\{(i,c): i\in [n]\backslash \{p^0,\dots, p^k\}, c\in [d_i]\}
\]
and 
\begin{align*} 
&\probC{p^{k+1}\in \{n-K+1,\dots,n-1\} }{\cF_k, \{p^{1}, \dots, p^{k}\}\cap \{n-K+1,\dots,n-1\} =\emptyset }\\
&\qquad=\frac{d_{n-K+1}+\dots+d_{n-1} }{n-1-\sum_{i=0}^k d_{p^i}} \leq \frac{K\Delta}{n-1-(k+1)\Delta}.
\end{align*}
If $m_k=1$ then the conditional probability above is $0$ so the inequality also holds.
    
Therefore, we see inductively that 
\begin{align*} 
\p{\{p^{1}, \dots, p^{b}\}\cap \{n-K+1,\dots,n-1\}=\emptyset}
&\geq \prod_{k=1}^{b} \left(1-\frac{K\Delta}{n-1-k\Delta}\right)\\
&\geq \left(1-\frac{K\Delta}{n-1-b\Delta}\right)^b. 
\end{align*}
\end{proof}

\end{appendix}

%%%%%%%%%%%%%%%%%%%%%%%%%%%%%%%%%%%%%%%%%%%%%%
%% Multiple Appendixes:                     %%
%%%%%%%%%%%%%%%%%%%%%%%%%%%%%%%%%%%%%%%%%%%%%%
%\begin{appendix}
%\section{???}
%
%\section{???}
%
%\end{appendix}

%%%%%%%%%%%%%%%%%%%%%%%%%%%%%%%%%%%%%%%%%%%%%%
%% Support information, if any,             %%
%% should be provided in the                %%
%% Acknowledgements section.                %%
%%%%%%%%%%%%%%%%%%%%%%%%%%%%%%%%%%%%%%%%%%%%%%
%\begin{acks}[Acknowledgments]
% The authors would like to thank ...
%\end{acks}
%%%%%%%%%%%%%%%%%%%%%%%%%%%%%%%%%%%%%%%%%%%%%%
%% Funding information, if any,             %%
%% should be provided in the                %%
%% funding section.                         %%
%%%%%%%%%%%%%%%%%%%%%%%%%%%%%%%%%%%%%%%%%%%%%%
\begin{funding}
This work was partially supported by the Marie Sk\l{}odowska–Curie RISE grant RandNET (Randomness and Learning in Networks), grant agreement ID 101007705. LAB acknowledges the financial support of NSERC and of the Canada Research Chairs program. SD acknowledges the financial support of the CogniGron research center and the Ubbo Emmius Funds (University of Groningen). Her research was also partially supported by the Marie Sk\l{}odowska-Curie grant GraPhTra (Universality in phase transitions in random graphs), grant agreement ID 101211705. RM’s work is supported by the EPSRC Centre for Doctoral Training in the Mathematics of Random Systems: Analysis, Modelling, and Simulation (EP/S023925/1). The final part of the work was carried out while LAB, SD and CG were in residence at the Simons Laufer Mathematical Sciences Institute in Berkeley, California, during the Spring 2025 semester. They gratefully acknowledge the financial support of the National Science Foundation under Grant No.\ DMS-1928930.
\end{funding}

%%%%%%%%%%%%%%%%%%%%%%%%%%%%%%%%%%%%%%%%%%%%%%
%% Supplementary Material, including data   %%
%% sets and code, should be provided in     %%
%% {supplement} environment with title      %%
%% and short description. It cannot be      %%
%% available exclusively as external link.  %%
%% All Supplementary Material must be       %%
%% available to the reader on Project       %%
%% Euclid with the published article.       %%
%%%%%%%%%%%%%%%%%%%%%%%%%%%%%%%%%%%%%%%%%%%%%%
%\begin{supplement}
%\stitle{???}
%\sdescription{???.}
%\end{supplement}

%%  MR numbers will be added by VTeX.                      %%
%%                                                         %%
%%  Use \cite{...} to cite references in text.             %%
%%                                                         %%
%%%%%%%%%%%%%%%%%%%%%%%%%%%%%%%%%%%%%%%%%%%%%%%%%%%%%%%%%%%%%

%% if your bibliography is in bibtex format, uncomment commands:
\bibliographystyle{imsart-number} % Style BST file (imsart-number.bst or imsart-nameyear.bst)
\bibliography{snakes}       % Bibliography file (usually '*.bib')

%% or include bibliography directly:
% \begin{thebibliography}{}
% \bibitem{b1}
% \end{thebibliography}

\end{document}